\documentclass{amsart}%
\usepackage{amssymb}
\usepackage{amsfonts}
\usepackage{geometry}
\usepackage{graphicx}
\usepackage{amsmath}%
\providecommand{\U}[1]{\protect\rule{.1in}{.1in}}
\newtheorem{theorem}{Theorem}
\theoremstyle{plain}

\newtheorem{corollary}[theorem]{Corollary}

\newtheorem{definition}[theorem]{Definition}

\newtheorem{lemma}[theorem]{Lemma}
\newtheorem{notation}[theorem]{Notation}

\newtheorem{remark}[theorem]{Remark}

\numberwithin{equation}{section}
\begin{document}
\title[A smooth Alpert testing characterization of convolution type]{A smooth Alpert testing characterization of convolution type for the Fourier
extension conjecture on paraboloids}
\author{Cristian Rios}
\address{University of Calgary\\
Calgary, Alberta, Canada}
\email{crios@ucalgary.ca}
\author[E. T. Sawyer]{Eric T. Sawyer$^{\dagger}$}
\address{Eric T. Sawyer, Department of Mathematics and Statistics\\
McMaster University\\
1280 Main Street West\\
Hamilton, Ontario L8S 4K1 Canada}
\email{sawyer@mcmaster.ca}
\thanks{$\dagger$ Research supported in part by a grant from the National Science and
Engineering Research Council of Canada.}
\date{\today }

\begin{abstract}
We prove a smooth Alpert testing characterization of convolution type for the
Fourier extension conjecture on paraboloids, that further extends the
characterizations in \cite{RiSa3}.

\end{abstract}
\maketitle
\tableofcontents

\section{Introduction}

The Fourier extension conjecture in $d\geq2$ dimensions arose from unpublished
work of E. Stein in 1967, see e.g. \cite[see the Notes at the end of Chapter
IX, p. 432, where Stein proved the extension conjecture for $q>\frac{4d}{d-1}%
$]{Ste2} and \cite{Ste}, and is the inequality%
\begin{equation}
\left\Vert \widehat{f\sigma_{d-1}}\right\Vert _{L^{q}\left(  \mathbb{R}%
^{d}\right)  }\leq\operatorname*{Four}_{L^{q}\rightarrow L^{q}}\left\Vert
f\right\Vert _{L^{q}\left(  \sigma_{d-1}\right)  },\ \text{for }f\in
L^{q}\left(  \sigma_{d-1}\right)  \text{ and }q>\frac{2d}{d-1},\label{FEC}%
\end{equation}
where $\sigma_{d-1}$ is surface measure on the sphere $\mathbb{S}^{d-1}$ or a
smooth compactly supported measure on the paraboloid $\mathbb{P}^{d-1}$, and
$\operatorname*{Four}_{L^{q}\rightarrow L^{q}}$ is the least constant for
which (\ref{FEC}) holds. The case $d=2$ of the Fourier extension conjecture
for the circle was proved over half a century ago by L. Carleson and P.
Sj\"{o}lin \cite{CaSj}, see also C. Fefferman \cite{Fef} and A. Zygmund
\cite{Zyg}.

There is a long history of progress on this conjecture in the past half
century, and we refer the reader to the excellent survey articles and books by
Thomas Wolff \cite{Wol}, Terence Tao \cite{Tao}, Pertti Mattila \cite{Mat},
Betsy Stovall \cite{Sto} and Ciprian Demeter \cite{Dem} for this history up to
2020, as well as for connections with related conjectures and topics.

The purpose of this paper is to give the following smooth Alpert testing
characterization of \emph{convolution} type for the Fourier extension
conjecture on the paraboloid.

Let $h_{I_{s};\kappa}^{\eta}\left(  x\right)  =2^{\frac{d-1}{2}s}h_{\left[
0,1\right)  ^{d-1};\kappa}^{\eta}\left(  2^{s}x\right)  $ be the $L^{2}%
$-normalized dilation of the smooth Alpert mother wavelet $h_{\left[
0,1\right)  ^{d-1};\kappa}^{\eta}$ constructed in \cite{Saw7}, and having
vanishing moments up to order $\kappa$. Set $k_{I_{s};\kappa}^{\eta}%
=2^{\frac{d-1}{2}s}h_{I_{s};\kappa}^{\eta}$ to be the $L^{1}$ normalized
wavelet. Finally, let $\mathcal{E}_{S}f\equiv\widehat{\Phi_{\ast}f}$ be the
Fourier extension operator for the surface $S$ given as the graph of $\Phi$,
and where $\Phi_{\ast}f$ is the pushforward of the measure $f\left(  x\right)
dx$.

\begin{theorem}
\label{reform}The Fourier extension conjecture holds for a smooth compact
piece $S$ of the paraboloid $\mathbb{P}^{d-1}$ in $\mathbb{R}^{d}$ \emph{if
and only if} for every $\varepsilon>0$ there is a positive constant
$C_{\varepsilon}$ and a positive integer $\kappa=\kappa_{\varepsilon}$ such
that%
\begin{align}
& \left\Vert \mathcal{E}_{S}\left(  k_{I_{s};\kappa}^{\eta}\ast f\right)
\right\Vert _{L^{\frac{2d}{d-1}}\left(  B\left(  0,2^{\left(  1+\varepsilon
\right)  s}\right)  \right)  }\leq C_{\varepsilon}2^{\varepsilon s}\left\Vert
f\right\Vert _{L^{\frac{2d}{d-1}}\left(  U\right)  },\label{reform'}\\
& \ \ \ \ \ \ \ \ \ \ \text{for all }s\in\mathbb{N}\text{ and }f\in
L^{\frac{2d}{d-1}}\left(  U\right)  .\nonumber
\end{align}

\end{theorem}

\begin{remark}
Tao's $\varepsilon$-removal theorem \cite{Tao1} shows that the Fourier
extension conjecture is implied by the local inequality,%
\[
\left\Vert \mathcal{E}f\right\Vert _{L^{\frac{2d}{d-1}}\left(  B\left(
0,2^{s}\right)  \right)  }\leq C_{\varepsilon}2^{\varepsilon s}\left\Vert
f\right\Vert _{L^{\frac{2d}{d-1}}\left(  U\right)  }.
\]
Thus Theorem \ref{reform} can be viewed as a variant in which $f$ is
preconvolved with a smooth Alpert wavelet having as many vanishing moments as
we wish.
\end{remark}

\section{Preliminaries}

We first recall the definition and properties of smooth Alpert wavelets from
\cite[Section 2]{Saw7} and \cite{Saw8} in $\mathbb{R}^{d}$ using a bounded
invertible operator $T_{\mathcal{G}}$, and establish a localization property
of this operator $T_{\mathcal{G}}$. Then we recall some known work on bilinear
inequalities for Fourier extension conjecture.

\subsection{Smooth Alpert wavelets}

Recall from \cite{Saw7} the smooth Alpert projections $\left\{  \bigtriangleup
_{Q;\kappa}\right\}  _{Q\in\mathcal{D}}$ and corresponding wavelets $\left\{
h_{Q;\kappa}^{a}\right\}  _{Q\in\mathcal{G},\ a\in\Gamma_{d}}$ of moment
vanishing order $\kappa$ in $\mathbb{R}^{d}$. In fact, $\left\{  h_{Q;\kappa
}^{a}\right\}  _{a\in\Gamma}$ is an orthonormal basis for the finite
dimensional vector subspace of $L^{2}$ that consists of linear combinations of
the indicators of\ the children $\mathfrak{C}\left(  Q\right)  $ of $Q$
multiplied by polynomials of degree at most $\kappa-1$, and such that the
linear combinations have vanishing moments on the cube $Q$ up to order
$\kappa-1$:%
\[
L_{Q;k}^{2}\left(  \mu\right)  \equiv\left\{  f=%
{\displaystyle\sum\limits_{Q^{\prime}\in\mathfrak{C}\left(  Q\right)  }}
\mathbf{1}_{Q^{\prime}}p_{Q^{\prime};k}\left(  x\right)  :\int_{Q}f\left(
x\right)  x_{i}^{\ell}d\mu\left(  x\right)  =0,\ \ \ \text{for }0\leq\ell
\leq\kappa-1\text{ and }1\leq i\leq d\right\}  ,
\]
where $p_{Q^{\prime};\kappa}\left(  x\right)  =\sum_{\alpha\in\mathbb{Z}%
_{+}^{n}:\left\vert \alpha\right\vert \leq\kappa-1\ }a_{Q^{\prime};\alpha
}x^{\alpha}$ is a polynomial in $\mathbb{R}^{d}$ of degree $\left\vert
\alpha\right\vert =\alpha_{1}+...+\alpha_{d}$ at most $\kappa-1$, and
$x^{\alpha}=x_{1}^{\alpha_{1}}x_{2}^{\alpha_{2}}...x_{d-1}^{\alpha_{d-1}}$.
Let $d_{Q;\kappa}\equiv\dim L_{Q;\kappa}^{2}\left(  \mu\right)  $ be the
dimension of the finite dimensional linear space $L_{Q;\kappa}^{2}\left(
\mu\right)  $. Moreover, for each $a\in\Gamma_{d}$, we may assume the wavelet
$h_{Q;\kappa}^{a}$ is a translation and dilation of the unit wavelet
$h_{Q_{0};\kappa}^{a}$, where $Q_{0}=\left[  0,1\right)  ^{d}$ is the unit
cube in $\mathbb{R}^{d}$.

Continuing to recall \cite{Saw7}, for a small positive constant $\eta>0$, we
define a smooth approximate identity by $\phi_{\eta}\left(  x\right)
\equiv\eta^{-n}\phi\left(  \frac{x}{\eta}\right)  $ where $\phi\in
C_{c}^{\infty}\left(  B_{\mathbb{R}^{d}}\left(  0,1\right)  \right)  $ has
unit integral, $\int_{\mathbb{R}^{d}}\phi\left(  x\right)  dx=1$, and
vanishing moments of \emph{positive} order less than $\kappa$, i.e.
\begin{equation}
\int\phi\left(  x\right)  x^{\gamma}dx=\delta_{\left\vert \gamma\right\vert
}^{0}=\left\{
\begin{array}
[c]{ccc}%
1 & \text{ if } & \left\vert \gamma\right\vert =0\\
0 & \text{ if } & 0<\left\vert \gamma\right\vert <\kappa
\end{array}
\right.  .\label{van pos}%
\end{equation}
The \emph{smooth} Alpert `wavelets' are defined by%
\[
h_{Q;\kappa}^{a,\eta}\equiv h_{Q;\kappa}^{a}\ast\phi_{\eta\ell\left(
Q\right)  }=h_{Q;\kappa}^{a}\ast\phi_{s}^{\eta},\ \ \ \ \ \text{if }%
\ell\left(  Q\right)  =2^{-s},
\]
where a consequence of the moment vanishing of $\phi_{s}^{\eta}$ is that
$h_{Q;\kappa}^{a,\eta}$ and $h_{Q;\kappa}^{a}$ coincide outside a small $\eta
$-halo of the skeleton $\operatorname*{Skel}Q\equiv\bigcup\partial Q^{\prime}$
of $Q$, where the union is taken over the dyadic children $Q^{\prime}$ of $Q$,
which in turn leads to the boundedness and invertibility on $L^{p}\left(
\mathbb{R}^{d}\right)  $ of the linear operator $S_{\kappa,\eta}^{\mathcal{G}%
}$ defined below. We\ also have for $0\leq\left\vert \beta\right\vert <\kappa
$,%
\begin{align*}
&  \int h_{Q;\kappa}^{a,\eta}\left(  x\right)  x^{\beta}dx=\int\phi_{\eta
\ell\left(  I\right)  }\ast h_{Q;\kappa}^{a}\left(  x\right)  x^{\beta}%
dx=\int\int\phi_{\eta\ell\left(  I\right)  }\left(  y\right)  h_{Q;\kappa}%
^{a}\left(  x-y\right)  x^{\beta}dx\\
&  =\int\phi_{\eta\ell\left(  I\right)  }\left(  y\right)  \left\{  \int
h_{Q;\kappa}^{a}\left(  x-y\right)  x^{\beta}dx\right\}  dy=\int\phi_{\eta
\ell\left(  I\right)  }\left(  y\right)  \left\{  \int h_{Q;\kappa}^{a}\left(
x\right)  \left(  x+y\right)  ^{\beta}dx\right\}  dy\\
&  =\int\phi_{\eta\ell\left(  I\right)  }\left(  y\right)  \left\{  0\right\}
dy=0,
\end{align*}
by translation invariance of Lebesgue measure.

Define $S_{\kappa,\eta}^{\mathcal{G}}$ to be the bounded invertible linear
operator on $L^{p}\left(  \mathbb{R}^{d}\right)  $ that satisfies
$S_{\kappa,\eta}^{\mathcal{G}}h_{Q;\kappa}^{a}=h_{Q;\kappa}^{a,\eta}$, and
then define%
\[
\bigtriangleup_{I;\kappa}^{\eta}f\equiv\sum_{a\in\Gamma_{d}}\left\langle
\left(  S_{\eta}^{\mathcal{G}}\right)  ^{-1}f,h_{I;\kappa}^{a}\right\rangle
h_{I;\kappa}^{a,\eta}=\left(  \bigtriangleup_{I;\kappa}f\right)  \ast
\phi_{\eta\ell\left(  I\right)  }\ .
\]
The existence of such an operator $S_{\kappa,\eta}^{\mathcal{G}}$ was
established in \cite{Saw7}.

\begin{theorem}
[\cite{Saw7}]\label{reproducing}Let $d\geq2$ and $\kappa\in\mathbb{N}$ with
$\kappa>\frac{d}{2}$. Then there is $\eta_{0}>0$ depending on $n$ and $\kappa
$\ such that for all $0<\eta<\eta_{0}$, and for all grids $\mathcal{G}$ in
$\mathbb{R}^{d}$, and all $1<p<\infty$, there is a bounded invertible operator
$S_{\eta}^{\mathcal{G}}=S_{\kappa,\eta}^{\mathcal{G}}$ on $L^{p}$, and a
positive constant $C_{p,d,\eta}$ such that the collection of functions
$\left\{  h_{I;\kappa}^{a,\eta}\right\}  _{I\in\mathcal{G},\ a\in\Gamma_{d}}$
is a $C_{p,d,\eta}$-frame for $L^{p}$, by which we mean,%
\begin{align}
f\left(  x\right)   &  =\sum_{I\in\mathcal{G},\ a\in\Gamma_{d}}\bigtriangleup
_{I;\kappa}^{\eta}f\left(  x\right)  ,\ \ \ \ \ \text{for a.e. }x\in
\mathbb{R}^{d}\text{, and for all }f\in L^{p},\label{bounded below}\\
\text{where }\bigtriangleup_{I;\kappa}^{\eta}f &  \equiv\sum_{a\in\Gamma_{d}%
}\left\langle \left(  S_{\eta}^{\mathcal{G}}\right)  ^{-1}f,h_{I;\kappa}%
^{a}\right\rangle \ h_{I;\kappa}^{a,\eta}\ ,\nonumber
\end{align}
and with convergence of the sum in both the $L^{p}$ norm and almost
everywhere, and%
\begin{equation}
\frac{1}{C_{p,d,\eta}}\left\Vert f\right\Vert _{L^{p}}\leq\left\Vert \left(
\sum_{I\in\mathcal{G}}\left\vert \bigtriangleup_{I;\kappa}^{\eta}f\right\vert
^{2}\right)  ^{\frac{1}{2}}\right\Vert _{L^{p}}\leq C_{p,d,\eta}\left\Vert
f\right\Vert _{L^{p}},\ \ \ \ \ \text{for all }f\in L^{p}.\label{square est}%
\end{equation}

\end{theorem}

\begin{notation}
\label{Notation Alpert} We will often drop the index $a$ parameterized by the
finite set $\Gamma_{d}$ as it plays no essential role in most of what follows,
and it will be understood that when we write
\[
\bigtriangleup_{Q;\kappa}^{\eta}f=\left\langle \left(  S_{\eta}^{\mathcal{G}%
}\right)  ^{-1}f,h_{Q;\kappa}\right\rangle h_{Q;\kappa}^{\eta}=\breve
{f}\left(  Q\right)  h_{Q;\kappa}^{\eta},
\]
we \emph{actually} mean the Alpert \emph{pseudoprojection},%
\[
\bigtriangleup_{Q;\kappa}^{\eta}f=\sum_{a\in\Gamma_{d}}\left\langle \left(
S_{\eta}^{\mathcal{G}}\right)  ^{-1}f,h_{Q;\kappa}^{a}\right\rangle
h_{Q;\kappa}^{\eta,a}=\sum_{a\in\Gamma_{d}}\breve{f}_{a}\left(  Q\right)
h_{Q;\kappa}^{a,\eta}\ ,
\]
where $\breve{f}_{a}\left(  Q\right)  $ is a convenient abbreviation for the
inner product $\left\langle \left(  S_{\eta}^{\mathcal{G}}\right)
^{-1}f,h_{Q;\kappa}^{a}\right\rangle $ when $\kappa$ is understood. More
precisely, one can view $\breve{f}\left(  Q\right)  =\left\{  \breve{f}%
_{a}\left(  Q\right)  \right\}  _{a\in\Gamma_{d}}$ and $h_{Q;\kappa}^{\eta
}=\left\{  h_{Q;\kappa}^{a,\eta}\right\}  _{a\in\Gamma_{d}}$ as sequences of
numbers and functions indexed by $\Gamma_{d}$, in which case $\breve{f}\left(
Q\right)  h_{Q;\kappa}^{\eta}$ is the dot product of these two finite sequences.
\end{notation}

We will also need the following extension of this theorem from \cite[Section
2]{Saw8}.

\begin{notation}
In the theorem below, we have \emph{redefined} $\bigtriangleup_{I;\kappa
}^{\eta}f$ to be $\sum_{a\in\Gamma_{d}}\left\langle \left(  T_{\eta
}^{\mathcal{G}}\right)  ^{-1}f,h_{I;\kappa}^{a,\eta}\right\rangle
\ h_{I;\kappa}^{a,\eta}$, rather than $\sum_{a\in\Gamma_{d}}\left\langle
\left(  S_{\eta}^{\mathcal{G}}\right)  ^{-1}f,h_{I;\kappa}^{a}\right\rangle
\ h_{I;\kappa}^{a,\eta}$. Thus the Alpert wavelet inside the inner product is
now smooth at the expense of replacing the bounded invertible operator
$S_{\eta}^{\mathcal{G}}$ with the self-adjoint invertible operator $T_{\eta
}^{\mathcal{G}}=S_{\eta}^{\mathcal{G}}\left(  S_{\eta}^{\mathcal{G}}\right)
^{\ast}$.
\end{notation}

\begin{theorem}
\label{reproducing'}Let $d\geq2$ and $\kappa\in\mathbb{N}$ with $\kappa
>\frac{d}{2}$. Define $T_{\eta}^{\mathcal{G}}=S_{\eta}^{\mathcal{G}}\left(
S_{\eta}^{\mathcal{G}}\right)  ^{\ast}$ for $0<\eta<\eta_{0}$, where $S_{\eta
}^{\mathcal{G}}$ and $\eta_{0}$ are as in Theorem \ref{reproducing}. Then for
all $0<\eta<\eta_{0}$, and for all grids $\mathcal{G}$ in $\mathbb{R}^{d} $,
and all $1<p<\infty$, there is a positive constant $C_{p,d,\eta}$ such that
the collection of functions $\left\{  h_{I;\kappa}^{a,\eta}\right\}
_{I\in\mathcal{G},\ a\in\Gamma_{d}}$ is a $C_{p,d,\eta}$-frame for $L^{p}$, by
which we mean,%
\begin{align}
f\left(  x\right)   &  =\sum_{I\in\mathcal{G}}\bigtriangleup_{I;\kappa}^{\eta
}f\left(  x\right)  ,\ \ \ \ \ \text{for all }f\in L^{p}%
,\label{bounded below'}\\
\text{where }\bigtriangleup_{I;\kappa}^{\eta}f  &  \equiv\sum_{a\in\Gamma_{d}%
}\left\langle \left(  T_{\eta}^{\mathcal{G}}\right)  ^{-1}f,h_{I;\kappa
}^{a,\eta}\right\rangle \ h_{I;\kappa}^{a,\eta}\ ,\nonumber
\end{align}
and with convergence of the sum in both the $L^{p}$ norm and almost
everywhere, and%
\begin{equation}
\frac{1}{C_{p,d,\eta}}\left\Vert f\right\Vert _{L^{p}}\leq\left\Vert \left(
\sum_{I\in\mathcal{G}}\left\vert \bigtriangleup_{I;\kappa}^{\eta}f\right\vert
^{2}\right)  ^{\frac{1}{2}}\right\Vert _{L^{p}}\leq C_{p,d,\eta}\left\Vert
f\right\Vert _{L^{p}},\ \ \ \ \ \text{for all }f\in L^{p}.\label{square est'}%
\end{equation}
Moreover, the smooth Alpert wavelets $\left\{  h_{I;\kappa}^{a,\eta}\right\}
_{I\in\mathcal{G},\ a\in\Gamma_{d}}$ are translation and dilation invariant in
the sense that $h_{I;\kappa}^{a,\eta}$ is a translate and dilate of the mother
Alpert wavelet $h_{I_{0};\kappa}^{a,\eta}$ where $I_{0}$ is the unit cube
centered at the origin in $\mathbb{R}^{d}$.
\end{theorem}

\begin{description}
\item[Convention] For the remainder of this paper, we assume that this doubly
smooth decomposition is in force whenever we write pseudoprojections
$\mathsf{Q}_{s}^{\eta}f$.
\end{description}

Despite being called pseudoprojections in \cite{Saw8}, it turns out that the
operators $\mathsf{Q}_{s,U}^{\eta}$ are actually orthogonal projections. The
key to proving this is the identity%
\begin{align}
\left\langle T_{\mathcal{G}}^{-1}h_{I;\kappa}^{\eta},h_{J;\kappa}^{\eta
}\right\rangle  &  =\left\langle \left(  S^{-1}\right)  ^{\ast}S^{-1}%
h_{I;\kappa}^{\eta},h_{J;\kappa}^{\eta}\right\rangle =\left\langle
S^{-1}h_{I;\kappa}^{\eta},S^{-1}h_{J;\kappa}^{\eta}\right\rangle ,\label{id}\\
&  =\left\langle h_{I;\kappa},h_{J;\kappa}\right\rangle =\delta_{I}%
^{J}\ ,\ \ \ \ \ \text{for }I,J\in\mathcal{G},\nonumber
\end{align}
where we have written $T_{\mathcal{G}}$ in place of $T_{\eta}^{\mathcal{G}}$.

\begin{lemma}
\label{Q is proj}For each grid $\mathcal{G}$, the collection of operators
$\left\{  \bigtriangleup_{I;\kappa}^{\eta}\right\}  _{I\in\mathcal{G}}$ form a
family of orthogonal projections on $L^{2}\left(  \mathbb{R}^{d}\right)  $,
i.e.%
\[
\bigtriangleup_{I;\kappa}^{\eta}\bigtriangleup_{J;\kappa}^{\eta}=\left\{
\begin{array}
[c]{ccc}%
\bigtriangleup_{I;\kappa}^{\eta} & \text{ if } & I,J\in\mathcal{G}\text{ and
}I=J\\
0 & \text{ if } & I,J\in\mathcal{G}\text{ and }I\not =J,
\end{array}
\right.  .
\]

\end{lemma}

\begin{proof}
For $f\in L^{2}\left(  \mathbb{R}^{d}\right)  $ we have%
\begin{align*}
\bigtriangleup_{I;\kappa}^{\eta}\bigtriangleup_{J;\kappa}^{\eta}f &
=\bigtriangleup_{I;\kappa}^{\eta}\left\langle T_{\mathcal{G}}^{-1}%
f,h_{J;\kappa}^{\eta}\right\rangle h_{J;\kappa}^{\eta}=\left\langle
T_{\mathcal{G}}^{-1}f,h_{J;\kappa}^{\eta}\right\rangle \bigtriangleup
_{I;\kappa}^{\eta}h_{J;\kappa}^{\eta}\\
&  =\left\langle T_{\mathcal{G}}^{-1}f,h_{J;\kappa}^{\eta}\right\rangle
\left\langle T_{\mathcal{G}}^{-1}h_{J;\kappa}^{\eta},h_{I;\kappa}^{\eta
}\right\rangle h_{I;\kappa}^{\eta}=\delta_{J}^{I}\left\langle T_{\mathcal{G}%
}^{-1}f,h_{J;\kappa}^{\eta}\right\rangle h_{I;\kappa}^{\eta},
\end{align*}
where the final equality follows from (\ref{id}).
\end{proof}

\subsubsection{Pointwise and discrete multipliers\label{subsub point}}

Given $\psi\in C_{c}^{\infty}\left(  \mathbb{R}^{d}\right)  $, define the
\emph{pointwise} multiplier $M_{\psi}$ on $L^{2}\left(  \mathbb{R}^{d}\right)
$ by%
\[
\left(  M_{\psi}f\right)  \left(  x\right)  \equiv\psi\left(  x\right)
f\left(  x\right)  ,\ \ \ \ \ \text{for }x\in\mathbb{R}^{d}.
\]
Given a grid $\mathcal{G}$, define the \emph{discrete} multiplier $M_{\psi
}^{\mathcal{G}}$ on $L^{2}\left(  \mathbb{R}^{d}\right)  $ by%
\begin{equation}
\left(  M_{\psi}^{\mathcal{G}}f\right)  \left(  x\right)  \equiv\sum
_{I\in\mathcal{G}}\psi\left(  b_{I}\right)  \left\langle T_{\mathcal{G}}%
^{-1}f,h_{I;\kappa}^{\eta}\right\rangle h_{I;\kappa}^{\eta}%
\ ,\label{def disc mult}%
\end{equation}
where $f=\sum_{I\in\mathcal{G}}\left\langle T_{\mathcal{G}}^{-1}f,h_{I;\kappa
}^{\eta}\right\rangle h_{I;\kappa}^{\eta}$ is the smooth $\mathcal{G}$ -Alpert
expansion of $f$, and $b_{I}$ is the bottom left vertex of $I$. Finally, given
a grid $\mathcal{G}$, define the \emph{perturbed} multiplier $\widetilde
{M}_{\psi}^{\mathcal{G}}$ on $L^{2}\left(  \mathbb{R}^{d}\right)  $ by%
\[
\left(  \widetilde{M}_{\psi}^{\mathcal{G}}f\right)  \left(  x\right)
\equiv\left(  M_{\psi}^{\mathcal{G}}T_{\mathcal{G}}f\right)  \left(  x\right)
=\sum_{I\in\mathcal{G}}\psi\left(  b_{I}\right)  \left\langle T_{\mathcal{G}%
}^{-1}T_{\mathcal{G}}f,h_{I;\kappa}^{\eta}\right\rangle h_{I;\kappa}^{\eta
}=\sum_{I\in\mathcal{G}}\psi\left(  b_{I}\right)  \left\langle f,h_{I;\kappa
}^{\eta}\right\rangle h_{I;\kappa}^{\eta}\ .
\]
The first two multiplier operators can be compared by writing%
\begin{align*}
M_{\psi}f\left(  x\right)   & =\sum_{I\in\mathcal{G}}\psi\left(  x\right)
\left\langle T_{\mathcal{G}}^{-1}f,h_{I;\kappa}^{\eta}\right\rangle
h_{I;\kappa}^{\eta}\left(  x\right)  \ \ \ \ \ \text{for }x\in\mathbb{R}%
^{d},\\
M_{\psi}^{\mathcal{G}}f\left(  x\right)   & =\sum_{I\in\mathcal{G}}\psi\left(
b_{I}\right)  \left\langle T_{\mathcal{G}}^{-1}f,h_{I;\kappa}^{\eta
}\right\rangle h_{I;\kappa}^{\eta}\left(  x\right)  \ \ \ \ \ \text{for }%
x\in\mathbb{R}^{d},
\end{align*}
where the discrete multiplier differs from the pointwise multiplier only in
that it \emph{selects} a particular point $b_{I}$ in the support of
$h_{I;\kappa}^{\eta}$ at which to evaluate the multiplier $\psi$.

Define the lattice
\begin{equation}
\Gamma_{s}^{\infty}\equiv2^{-s}\mathbb{Z}^{d-1}.\label{def Gamma_s}%
\end{equation}
Now let $\mathcal{G}=\mathcal{D}+v$ where $\mathcal{D}$ is the standard grid
on $\mathbb{R}^{d}$, and suppose $I\in\mathcal{G}_{s}$ satisfies $I=\tau
_{n+v}I_{s}$ for (a uniquely determined) $n=n\left(  I\right)  \in\Gamma_{s}$,
and where $I_{s}\in\mathcal{D}_{s}$ is the unique cube with bottom left vertex
at the origin. Then if we define $h_{s;\kappa}^{\eta}\equiv h_{I_{s};\kappa
}^{\eta}$ for notational convenience, and use the identity%
\[
\left\langle T_{\mathcal{G}}^{-1}h_{J;\kappa}^{\eta},h_{I;\kappa}^{\eta
}\right\rangle =\left\langle S_{\mathcal{G}}^{-1}h_{J;\kappa}^{\eta
},S_{\mathcal{G}}^{-1}h_{I;\kappa}^{\eta}\right\rangle =\left\langle
h_{J;\kappa},h_{I;\kappa}\right\rangle =\delta_{I}^{J}\ ,
\]
we have%
\[
M_{\psi}^{\mathcal{G}}h_{J;\kappa}^{\eta}=\sum_{I\in\mathcal{G}}\psi\left(
b_{I}\right)  \left\langle T_{\mathcal{G}}^{-1}h_{J;\kappa}^{\eta}%
,h_{I;\kappa}^{\eta}\right\rangle h_{I;\kappa}^{\eta}=\psi\left(
b_{I}\right)  h_{I;\kappa}^{\eta}=\psi\left(  n+v\right)  \tau_{n+v}%
h_{s;\kappa}^{\eta}\ ,\ \ \ \ \ \text{\ for }I\in\mathcal{G}_{s},
\]
and so also the following useful formula,%
\begin{align}
M_{\psi}^{\mathcal{G}}\mathsf{Q}_{s}^{\eta,\mathcal{G}}f  & =M_{\psi
}^{\mathcal{G}}\left(  \sum_{I\in\mathcal{G}_{s}}\left\langle T_{\mathcal{G}%
}^{-1}f,h_{I;\kappa}^{\eta}\right\rangle h_{I;\kappa}^{\eta}\right)
=\sum_{I\in\mathcal{G}_{s}}\left\langle T_{\mathcal{G}}^{-1}f,h_{I;\kappa
}^{\eta}\right\rangle M_{\psi}^{\mathcal{G}}h_{I;\kappa}^{\eta}%
\label{useful form}\\
& =\sum_{n\in\Gamma_{s}}\left\langle T_{\mathcal{G}}^{-1}f,\tau_{n+v}%
h_{s;\kappa}^{\eta}\right\rangle \psi\left(  n+v\right)  \tau_{n+v}%
h_{s;\kappa}^{\eta}.\nonumber
\end{align}

\subsection{Localization estimates for $SS^{\ast}$}

First we recall a result from \cite[Lemma 19]{Saw7}. For $K\in\mathcal{G}$,
define the collection of Carleson cubes of the skeleton $\operatorname*{Skel}%
\left(  K\right)  \equiv\bigcup_{K\in\mathfrak{C}_{\mathcal{G}}\left(
J\right)  }\partial K$ of $K$ by%
\[
\operatorname*{Car}\left(  K\right)  \equiv\left\{  L\in\mathcal{G}%
:\ell\left(  L\right)  \leq\ell\left(  K\right)  \text{ and }L\cap
\operatorname*{Skel}\left(  K\right)  \neq\emptyset\right\}  .
\]
Define the $\eta$-halo of the skeleton $\operatorname*{Skel}\left(  K\right)
=\bigcup_{K\in\mathfrak{C}_{\mathcal{G}}\left(  J\right)  }\partial K$ of the
square $J$ by
\[
\mathcal{H}_{\eta}\left(  J\right)  \equiv\bigcup_{K\in\mathfrak{C}%
_{\mathcal{G}}\left(  J\right)  }\left\{  \left(  1+\eta\right)
K\setminus\left(  1-\eta\right)  K\right\}  ,
\]
where $\mathfrak{C}_{\mathcal{G}}\left(  J\right)  $ denotes the set of
$2^{n}$ dyadic children of $J$.

\begin{lemma}
[{\cite[Lemma 19]{Saw7}}]\label{inner est}Suppose $\kappa\in\mathbb{N}$ with
$\kappa>\frac{n}{2}$, $0<\eta=2^{-k}<1$, and $I,J\in\mathcal{G}$, where
$\mathcal{G}$ is a grid in $\mathbb{R}^{n}$. Then we have%
\begin{align*}
\left\vert \left\langle h_{J;\kappa},h_{J;\kappa}^{\eta}\right\rangle
\right\vert  & \approx1\text{ and }\left\vert \left\langle h_{J;\kappa}^{\eta
},h_{J^{\prime};\kappa}\right\rangle \right\vert \lesssim\eta
,\ \ \ \ \ \text{for }J\text{ and }J^{\prime}\text{ siblings},\\
\left\vert \left\langle h_{J;\kappa},h_{I;\kappa}^{\eta}\right\rangle
\right\vert  & \lesssim\eta\left(  \frac{\ell\left(  I\right)  }{\ell\left(
J\right)  }\right)  ^{\frac{n}{2}},\ \ \ \ \ \text{for }I\in
\operatorname*{Car}\left(  J\right)  ,\\
\left\vert \left\langle h_{J;\kappa},h_{I;\kappa}^{\eta}\right\rangle
\right\vert  & \lesssim\eta\left(  \frac{\ell\left(  J\right)  }{\ell\left(
I\right)  }\right)  ^{\frac{n}{2}-1},\ \ \ \ \ \text{for }J\in
\operatorname*{Car}\left(  I\right)  \text{ and }\ell\left(  J\right)
\geq\eta\ell\left(  I\right)  ,\\
\left\vert \left\langle h_{J;\kappa},h_{I;\kappa}^{\eta}\right\rangle
\right\vert  & \lesssim\frac{1}{\eta^{\kappa}}\left(  \frac{\ell\left(
J\right)  }{\ell\left(  I\right)  }\right)  ^{\kappa+\frac{n}{2}%
},\ \ \ \ \ \text{for }\ell\left(  J\right)  \leq\eta\ell\left(  I\right)
\text{ and }J\cap\mathcal{H}_{\eta}\left(  I\right)  \neq\emptyset,\\
\left\langle h_{J;\kappa},h_{I;\kappa}^{\eta}\right\rangle  &
=0,\ \ \ \ \ \text{in all other cases}.
\end{align*}

\end{lemma}

We will use the graph distance $\delta_{\operatorname*{graph}}$ in
$\mathcal{G}$ as in \cite{Saw8}, and we also adapt some of the arguments found
there. Viewing $\mathcal{G}$ with its \emph{tree} structure, where an edge is
placed between each cube $I$ and its parent $\pi_{\mathcal{G}}I$, we define
the `tree distance' $\delta_{\operatorname*{tree}}\left(  J,I\right)  $
between two dyadic cubes in $\mathcal{G}$ to be the length of the (unique)
shortest path joining $I$ and $J$ in the tree $\mathcal{G}$. Viewing
$\mathcal{G}$ instead with its \emph{graph} structure, where in addition to
the edges in the tree structure, edges are also placed between adjacent cubes
in $\mathcal{G}$ having the same sidelength, we define the `graph distance'
$\delta_{\operatorname*{graph}}\left(  J,I\right)  $ between two dyadic cubes
in $\mathcal{G}$ to be the length of a shortest path joining $I$ and $J$ in
the tree $\mathcal{G}$ (shortest paths are not always unique in this graph structure).

\begin{notation}
From now on we will write $T_{\mathcal{G}}^{\eta}$ or $T_{\mathcal{G}}$ or
simply $T$ to denote the operator $T_{\eta}^{\mathcal{G}}=S_{\eta
}^{\mathcal{G}}\left(  S_{\eta}^{\mathcal{G}}\right)  ^{\ast}$ in Theorem
\ref{reproducing'} above when $\eta$ and $\mathcal{G}$ are understood from context.
\end{notation}

We now note that both%
\begin{align*}
\left\vert \left\langle Sh_{J;\kappa},h_{I;\kappa}\right\rangle \right\vert  &
=\left\vert \left\langle h_{J;\kappa}^{\eta},h_{I;\kappa}\right\rangle
\right\vert \leq2^{C-\kappa\delta_{\operatorname*{graph}}\left(  I,J\right)
},\\
\left\vert \left\langle S^{\ast}h_{J;\kappa},h_{I;\kappa}\right\rangle
\right\vert  & =\left\vert \left\langle h_{J;\kappa},h_{I;\kappa}^{\eta
}\right\rangle \right\vert \leq2^{C-\kappa\delta_{\operatorname*{graph}%
}\left(  I,J\right)  }.
\end{align*}
Then using
\begin{align*}
S^{\ast}h_{J;\kappa}  & =\sum_{I\in\mathcal{G}}\left\langle S^{\ast
}h_{J;\kappa},h_{I;\kappa}\right\rangle h_{I;\kappa}=\left\langle S^{\ast
}h_{J;\kappa},h_{J;\kappa}\right\rangle h_{J;\kappa}+\sum_{I\in\mathcal{G}%
:\ I\neq J}\left\langle S^{\ast}h_{J;\kappa},h_{I;\kappa}\right\rangle
h_{I;\kappa}\\
& =a_{\eta}h_{J;\kappa}+\sum_{I\in\mathcal{G}:\ I\neq J}\left\langle
h_{J;\kappa},h_{I;\kappa}^{\eta}\right\rangle h_{I;\kappa}\ ,
\end{align*}
we obtain%
\begin{align*}
\left\langle Th_{J;\kappa},h_{I;\kappa}\right\rangle  & =\left\langle S\left(
S^{\ast}h_{J;\kappa}\right)  ,h_{I;\kappa}\right\rangle =\left\langle S\left(
a_{\eta}h_{J;\kappa}+\sum_{K\in\mathcal{G}:\ K\neq J}\left\langle h_{J;\kappa
},h_{K;\kappa}^{\eta}\right\rangle h_{K;\kappa}\right)  ,h_{I;\kappa
}\right\rangle \\
& =\left\langle S\left(  a_{\eta}h_{J;\kappa}\right)  ,h_{I;\kappa
}\right\rangle +\sum_{K\in\mathcal{G}:\ K\neq J}\left\langle S\left(
\left\langle h_{J;\kappa},h_{K;\kappa}^{\eta}\right\rangle h_{K;\kappa
}\right)  ,h_{I;\kappa}\right\rangle \\
& =a_{\eta}\left\langle h_{J;\kappa}^{\eta},h_{I;\kappa}\right\rangle
+\sum_{K\in\mathcal{G}:\ K\neq J}\left\langle h_{J;\kappa},h_{K;\kappa}^{\eta
}\right\rangle \left\langle h_{K;\kappa}^{\eta},h_{I;\kappa}\right\rangle .
\end{align*}

Thus in the special case where there is $n\in\mathbb{N}$ such that
\begin{equation}
\ell\left(  I\right)  \leq1\ \text{and }\operatorname*{dist}\left(
J,I\right)  >2^{n}\text{ and }\ell\left(  J\right)  \leq2^{n},\label{special}%
\end{equation}
we must have $\ell\left(  K\right)  \geq2^{n}$, and we have the\ following
estimate from the fourth line in Lemma \ref{inner est},
\[
\left\vert \left\langle h_{I;\kappa},h_{K;\kappa}^{\eta}\right\rangle
\right\vert \lesssim\frac{1}{\eta^{\kappa}}\left(  \frac{\ell\left(  I\right)
}{\ell\left(  K\right)  }\right)  ^{\kappa+\frac{d-1}{2}},\ \ \ \ \ \text{for
}\ell\left(  I\right)  \leq\eta\ell\left(  K\right)  \text{ and }%
I\cap\mathcal{H}_{\eta}\left(  K\right)  \neq\emptyset.
\]
Altogether we obtain for $I$ and $J$ as in (\ref{special}) that%
\begin{align*}
& \left\vert \sum_{K\in\mathcal{G}:\ K\neq J}\left\langle h_{J;\kappa
},h_{K;\kappa}^{\eta}\right\rangle \left\langle h_{K;\kappa}^{\eta
},h_{I;\kappa}\right\rangle \right\vert \lesssim\sum_{K\in\mathcal{G}:\ K\neq
J}\left\vert \left\langle h_{J;\kappa},h_{K;\kappa}^{\eta}\right\rangle
\right\vert \left\vert \left\langle h_{K;\kappa}^{\eta},h_{I;\kappa
}\right\rangle \right\vert \\
& \lesssim\sum_{K\in\mathcal{G}:\ I\in\mathcal{H}^{\eta}\left(  K\right)
\text{ and }\ell\left(  K\right)  \geq2^{n}}\frac{1}{\eta^{\kappa}}\left(
\frac{\ell\left(  I\right)  }{\ell\left(  K\right)  }\right)  ^{\kappa
+\frac{d-1}{2}}=\sum_{m=n}^{\infty}C_{d-1}\frac{1}{\eta^{\kappa}}\left(
\frac{\ell\left(  I\right)  }{2^{m}}\right)  ^{\kappa+\frac{d-1}{2}}\leq
C_{d-1}\frac{1}{\eta^{\kappa}}\left(  \frac{\ell\left(  I\right)  }{2^{n}%
}\right)  ^{\kappa+\frac{d-1}{2}},
\end{align*}
and so%
\[
\left\vert \left\langle Th_{J;\kappa},h_{I;\kappa}\right\rangle \right\vert
\leq\left\vert a_{\eta}\left\langle h_{J;\kappa}^{\eta},h_{I;\kappa
}\right\rangle \right\vert +C_{d-1}\frac{1}{\eta^{\kappa}}\left(  \frac
{\ell\left(  I\right)  }{2^{n}}\right)  ^{\kappa+\frac{d-1}{2}}=C_{d-1}%
\frac{1}{\eta^{\kappa}}\left(  \frac{\ell\left(  I\right)  }{2^{n}}\right)
^{\kappa+\frac{d-1}{2}},
\]
since $\left\langle h_{J;\kappa}^{\eta},h_{I;\kappa}\right\rangle =0$. We
summarize this in the following lemma.

\begin{lemma}
\label{lem special}For $n\in\mathbb{N}$, and $I$ and $J$ as in (\ref{special}%
), we have%
\[
\left\vert \left\langle Th_{J;\kappa},h_{I;\kappa}\right\rangle \right\vert
\leq C_{d-1}\frac{1}{\eta^{\kappa}}\left(  \frac{\ell\left(  I\right)  }%
{2^{n}}\right)  ^{\kappa+\frac{d-1}{2}}.
\]

\end{lemma}

\subsection{Bilinear inequalities}

Here we restrict attention to the case when $\mathcal{E}=\mathcal{E}_{S}$
where $S$ is a smooth compact piece of the paraboloid $\mathbb{P}^{d-1}$. We
use the following definition from \cite{BoGu}.

\begin{definition}
For $q>\frac{2d}{d-1}$ and $R>1$ define%
\[
\mathfrak{N}_{R}^{\left(  q\right)  }\equiv\sup_{\left\Vert f\right\Vert
_{L^{\infty}\left(  U\right)  \leq1}}\left\Vert \mathcal{E}f\right\Vert
_{L^{q}\left(  A\left(  0,R\right)  \right)  },
\]
where $A\left(  0,R\right)  \equiv B\left(  0,R\right)  \setminus B\left(
0,\frac{1}{2}R\right)  $ is the annulus with radii $R$ and $\frac{1}{2}R$.
\end{definition}

The following theorem is well known.

\begin{theorem}
\label{main}The Fourier extension conjecture holds for a smooth compact piece
$S$ of the paraboloid $\mathbb{P}^{d-1}$ \emph{if and only if} for every
$q>\frac{2d}{d-1}$ and $\varepsilon>0$, there is a positive constant
$C_{q,\varepsilon}$ such that $\mathfrak{N}_{R}^{\left(  q\right)  }\leq
C_{q,\varepsilon}R^{\varepsilon}$ for all $R>1$.
\end{theorem}

\begin{proof}
We first note that we can replace the annulus $A\left(  0,R\right)  $ in the
definition of $\mathfrak{N}_{R}^{\left(  q\right)  }$ with the ball $B\left(
0,R\right)  $. Indeed, we use that $B\left(  0,R\right)  $ is a union of the
unit ball and at most $r$ annuli of the form $A\left(  0,2^{t}\right)
=B\left(  0,2^{t}\right)  \setminus B\left(  0,2^{t-1}\right)  $ where
$R=2^{r}$, in order to obtain,%
\[
\int_{B\left(  0,R\right)  }\left\vert \mathcal{E}f\left(  \xi\right)
\right\vert ^{q}d\xi=\sum_{t=1}^{r}\int_{A\left(  0,2^{t}\right)  }\left\vert
\mathcal{E}f\left(  \xi\right)  \right\vert ^{q}d\xi+\int_{B\left(
0,1\right)  }\left\vert \mathcal{E}f\left(  \xi\right)  \right\vert ^{q}%
d\xi\lesssim\sum_{t=1}^{r}\left(  C_{\varepsilon}2^{\varepsilon t}\left\Vert
f\right\Vert _{L^{\infty}}\right)  ^{q}+\left\Vert f\right\Vert _{L^{\infty}%
}^{q}\lesssim R^{\varepsilon q}\left\Vert f\right\Vert _{L^{\infty}}^{q}.
\]
Now we can use the well known fact that the Fourier extension conjecture
follows from this inequality by Nikishin-Maurey-Pisier factorization as in
\cite{Bou}, and Tao's $\varepsilon$-removal theorem in \cite{Tao1}.
\end{proof}

In order to bound $\mathfrak{N}_{R}^{\left(  q\right)  }$ we recall a variant
in \cite{RiSa} of an inequality of Bourgain and Guth \cite[Theorem 1$^{\prime
}$]{BoGu} that was proved using their pigeonholing technique.

\begin{definition}
For $d\geq3$ and $\nu>0$ define%
\begin{equation}
\Gamma_{\nu}\left[  U\right]  \equiv\left\{  \left(  U_{1},U_{2}\right)
:\ell\left(  U_{1}\right)  ,\ell\left(  U_{1}\right)  ,\operatorname*{dist}%
\left(  U_{1},U_{2}\right)  \geq\nu\right\}  .\label{weak sep var}%
\end{equation}

\end{definition}

The following theorem is due to Tao, Vargas and Vega \cite{TaVaVe}. In
\cite{RiSa} we gave a proof (in the trilinear setting) using the pigeonholing
technique of Bourgan and Guth \cite{BoGu}, but unknown to the authors at that
time, this pigeonholing proof had essentially appeared in C. Demeter's book
\cite[see Proposition 7.20]{Dem}.

\begin{theorem}
[Theorem 6 in \cite{RiSa}]\label{Thm 6 copy(1)}Let $d\geq3$, $q>\frac{2d}%
{d-1}$ and $R>1$. Then there is a positive constant $\nu=\nu\left(
q,d\right)  $ such that for every $\varepsilon>0$, there is another positive
constant $C_{\varepsilon,q}$ such that,%
\begin{equation}
\mathfrak{N}_{R}^{\left(  q\right)  }\leq C_{\varepsilon,q}R^{\varepsilon}%
\sup_{\left\Vert f\right\Vert _{L^{\infty}\left(  U\right)  \leq1}}%
\sup_{\left(  U_{1},U_{2}\right)  \in\Gamma_{\nu}\left[  U\right]  }\left\Vert
\left(  \mathcal{E}\mathbf{1}_{U_{1}}f\right)  \left(  \mathcal{E}%
\mathbf{1}_{U_{2}}f\right)  \right\Vert _{L^{\frac{q}{2}}\left(  A\left(
0,R\right)  \right)  }^{\frac{1}{2}}.\label{frak N_R}%
\end{equation}

\end{theorem}

\begin{proof}
The reader can easily check that the proof of this inequality is contained in
the proof of Theorem 6 in \cite{RiSa}, where a trilinear version was proved in
dimension $d=3$. Indeed, the only obstacle to extending the three-dimensional
proof in \cite{RiSa} to higher dimensions lies in the fact that if we are
considering a disjoint $N$-linear bound in $L^{\frac{q}{N}}$ for $q>\frac
{2d}{d-1}$, the Bourgain Guth pigeonholing argument used in \cite{RiSa}
requires $\frac{q}{N}>\frac{1}{N}\frac{2d}{d-1}\geq1$. Indeed, in the display%
\begin{align*}
& \left\vert Tf\left(  \xi\right)  \right\vert ^{q}\lesssim\left\vert
Tf\left(  \xi\right)  \right\vert ^{3}\leq\left(  2^{\left(  6+3\alpha\right)
\lambda}w_{I_{0}^{a}}^{a}w_{J_{0}^{a}}^{a}w_{K_{0}^{a}}^{a}\right)  ^{\frac
{q}{3}}\\
& \approx2^{\left(  6+3\alpha\right)  \lambda\frac{q}{3}}\left(
\int_{\mathbb{R}^{3}}\left\vert T_{I_{0}^{a}}f\left(  z_{1}\right)
\right\vert \zeta_{\lambda}\left(  z_{1}-a\right)  dz_{1}\right)  ^{\frac
{q}{3}}\left(  \int_{\mathbb{R}^{3}}\left\vert T_{J_{0}^{a}}f\left(
z_{2}\right)  \right\vert \zeta_{\lambda}\left(  z_{2}-a\right)
dz_{2}\right)  ^{\frac{q}{3}}\left(  \int_{\mathbb{R}^{3}}\left\vert
T_{K_{0}^{a}}f\left(  z_{3}\right)  \right\vert \zeta_{\lambda}\left(
z_{3}-a\right)  dz_{3}\right)  ^{\frac{q}{3}}\\
& \approx2^{\left(  6+3\alpha\right)  \lambda\frac{q}{3}}\int_{\mathbb{R}^{3}%
}\int_{\mathbb{R}^{3}}\int_{\mathbb{R}^{3}}\left\vert T_{I_{0}^{a}}f\left(
\xi-z_{1}\right)  T_{J_{0}^{a}}f\left(  \xi-z_{2}\right)  T_{K_{0}^{a}%
}f\left(  \xi-z_{3}\right)  \right\vert ^{\frac{q}{3}}\left(  \zeta_{s}\left(
z_{1}\right)  \zeta_{s}\left(  z_{2}\right)  \zeta_{s}\left(  z_{3}\right)
\right)  ^{\frac{q}{3}}dz_{1}dz_{2}dz_{3}\ .
\end{align*}
near the beginning of \textbf{Case 1} of the proof of Theorem 6 in
\cite{RiSa}, we use that $\frac{q}{3}\geq1$ in order to apply H\"{o}lder's
inequality to each integral in the second line. However, if the dimension
$d\geq4$, and we replace the trilinear inequality with an $N$-linear
inequality, then we would need $N\leq\frac{2d}{d-1}<3$, i.e. a bilinear
characterization. In fact, a straightforward generalization of the argument in
\cite{RiSa} shows that the bilinear characterization holds in all dimensions
$d\geq2$.
\end{proof}

We can trivially replace the indicator cutoff functions $\mathbf{1}_{U_{1}}$
and $\mathbf{1}_{U_{2}}$ with smooth cutoff functions $\psi_{1}$ and $\psi
_{2}$ adapted to $U_{1}$ and $U_{2}$ respectively in (\ref{frak N_R}), in
order to replace (\ref{frak N_R}) by
\begin{align}
\mathfrak{N}_{R}^{\left(  q\right)  }  & \leq C_{\varepsilon,q}R^{\varepsilon
}\sup_{\left\Vert f\right\Vert _{L^{\infty}\left(  U\right)  \leq1}}%
\sup_{\left(  \operatorname*{Supp}\psi_{1},\operatorname*{Supp}\psi
_{2}\right)  \in\Gamma_{\nu}\left[  U\right]  }\left\Vert \left(
\mathcal{E}M_{\psi_{1}}f\right)  \left(  \mathcal{E}M_{\psi_{2}}f\right)
\right\Vert _{L^{\frac{q}{2}}\left(  A\left(  0,R\right)  \right)  }^{\frac
{1}{2}}\label{important}\\
& \leq C_{\varepsilon,q}R^{\varepsilon}\otimes_{2}\Lambda_{\delta}\left(
R\right)  ,\nonumber
\end{align}
in Theorem \ref{Thm 6 copy(1)}.

\begin{remark}
\label{think}Using parabolic rescaling (\ref{dil''}) with $t>1$ as in
\cite{TaVaVe}, we may assume that $\nu=\nu\left(  q\right)  $ in
(\ref{important}) is a fixed small constant independent of $q>\frac{2d}{d-1}$,
and for convenience we can think of $\nu$ as being equal to $1$ for the
remainder of the paper.
\end{remark}

\begin{remark}
We also mention that in a \emph{transverse} variant of the multilinear
analogue of (\ref{important}), in which any $d$-tuple of normals from each of
the $d$ patches\ are assumed to span $\mathbb{R}^{d}$, that was introduced by
Bennett, Carbery and Tao in \cite{BeCaTa}, they proved the remarkable result
that both the \emph{transverse} multilinear Fourier extension inequality and
the \emph{transverse} multilinear Kakeya inequality hold. They first proved
the transverse multilinear Kakeya inequality using a beautiful probabilistic
argument applied to sliding Gaussian tubes, and then derived the transverse
multilinear Fourier extension inequality using a clever bootstrapping argument
involving the transverse multilinear Kakeya inequality. However, the
transverse multilinear inequalities were not enough to prove, in any obvious
way, the corresponding linear inequalities needed for the Fourier extension theorem.
\end{remark}

\subsubsection{Convolution of $\nu$-disjoint singular measures on the
paraboloid}

We repeat a subsection on convolution from \cite{RiSa} but in higher
dimensions, and with some misprints corrected. Let $\mu^{1}\equiv\Phi_{\ast
}M_{\psi_{1}}\mathsf{Q}_{s_{1}}^{\eta}f_{1}$ and $\mu^{2}\equiv\Phi_{\ast
}M_{\psi_{2}}\mathsf{Q}_{s_{2}}^{\eta}f_{2}$ denote singular measures on the
paraboloid, that are pushforwards of smooth Alpert projections at levels
$s_{1}<s_{2}$, and where $\operatorname*{diam}\left(  U_{1}\right)
\approx\operatorname*{diam}\left(  U_{2}\right)  \approx\operatorname*{dist}%
\left(  U_{1},U_{2}\right)  \gtrsim\nu>0$. For $z\in\mathbb{R}^{d}$, denote by
$\omega_{z}$ the translate of a measure $\omega$ by $z$. We use duality to
compute the convolution $\mu^{1}\ast\mu^{2}$ in terms of the measure-valued
integral $\int_{\mathbb{R}^{d}}\left[  \mu_{w}^{1}\left(  \cdot\right)
\right]  d\mu^{2}\left(  w\right)  $ as follows. For $F$ a continuous function
on $\mathbb{R}^{3}$, write
\begin{align*}
& \left\langle F,\Phi_{\ast}M_{\psi_{1}}\mathsf{Q}_{s_{1}}^{\eta}f_{1}\ast
\Phi_{\ast}M_{\psi_{2}}\mathsf{Q}_{s_{2}}^{\eta}f_{2}\right\rangle
=\left\langle F,\mu^{1}\ast\mu^{2}\right\rangle =\left\langle F\left(
\cdot\right)  ,\int_{\mathbb{R}^{d}}\mu_{w}^{1}\left(  \cdot\right)  d\mu
^{2}\left(  w\right)  \right\rangle \\
& =\int_{\mathbb{R}^{d}}F\left(  z\right)  d\left[  \int_{\mathbb{R}^{d}}%
\mu_{w}^{1}\left(  z\right)  d\mu^{2}\left(  w\right)  \right]  =\int
_{\mathbb{R}^{d}}\left\{  \int_{\mathbb{R}^{d}}F\left(  z\right)  d\mu_{w}%
^{1}\left(  z\right)  \right\}  d\mu^{2}\left(  w\right) \\
& =\int_{\mathbb{R}^{d}}\left\{  \int_{\mathbb{R}^{d}}F\left(  z-w\right)
d\mu^{1}\left(  z\right)  \right\}  d\mu^{2}\left(  w\right)  ,
\end{align*}
and using the definitions of $\mu^{1}$ and $\mu^{2}$ as push forwards
respectively of $M_{\psi_{1}}\mathsf{Q}_{s_{1}}^{\eta}f_{1}\left(  v\right)
dv$ and $M_{\psi_{2}}\mathsf{Q}_{s_{2}}^{\eta}f_{2}\left(  u\right)  du$ by
$\Phi$, we see that%
\begin{align*}
\left\langle F,\mu^{1}\ast\mu^{2}\right\rangle  & =\int_{\mathbb{R}^{d}%
}\left\{  \int_{\mathbb{R}^{d}}F\left(  z-w\right)  d\mu^{1}\left(  z\right)
\right\}  d\mu^{2}\left(  w\right) \\
& =\int_{\mathbb{R}^{d}}\left\{  \int_{\mathbb{R}^{d-1}}F\left(  \Phi\left(
v\right)  -w\right)  M_{\psi_{1}}\mathsf{Q}_{s_{1}}^{\eta}f_{1}\left(
v\right)  dv\right\}  d\mu^{2}\left(  w\right) \\
& =\int_{\mathbb{R}^{d-1}}\int_{\mathbb{R}^{d-1}}F\left(  \Phi\left(
v\right)  -\Phi\left(  u\right)  \right)  M_{\psi_{1}}\mathsf{Q}_{s_{1}}%
^{\eta}f_{1}\left(  v\right)  dvM_{\psi2}\mathsf{Q}_{s_{2}}^{\eta}f_{2}\left(
u\right)  du.
\end{align*}

Taking limits we can let $F=\delta_{a}$, so that for $\left(  v,u\right)
\in\mathbb{R}^{d-1}\times\mathbb{R}^{d-1}$,%
\begin{align*}
\left\langle \delta_{a},\mu^{1}\ast\mu^{2}\right\rangle  & =\iint
_{\mathbb{R}^{d-1}\times\mathbb{R}^{d-1}}\delta_{a}\left(  \Phi\left(
v\right)  -\Phi\left(  u\right)  \right)  M_{\psi_{1}}\mathsf{Q}_{s_{1}}%
^{\eta}f_{1}\left(  v\right)  dvM_{\psi2}\mathsf{Q}_{s_{2}}^{\eta}f_{2}\left(
u\right)  du\\
& =\iint_{E_{a}}M_{\psi_{1}}\mathsf{Q}_{s_{1}}^{\eta}f_{1}\left(  v\right)
M_{\psi2}\mathsf{Q}_{s_{2}}^{\eta}f_{2}\left(  u\right)  d\sigma_{E_{a}%
}\left(  u,v\right)  ,
\end{align*}
where $E_{a}$ is a $\left(  d-2\right)  $-dimensional smooth surface in
$\mathbb{R}^{d-1}\times\mathbb{R}^{d-1}$ given by%
\begin{align*}
E_{a}  & \equiv\left\{  \left(  v,u\right)  \in\mathbb{R}^{d-1}\times
\mathbb{R}^{d-1}:\Phi\left(  v\right)  -\Phi\left(  u\right)  =a\right\} \\
& =\left\{  \left(  v,u\right)  \in\mathbb{R}^{d-1}\times\mathbb{R}%
^{d-1}:v-u=a^{\prime}\text{ and }\left\vert v\right\vert ^{2}-\left\vert
u\right\vert ^{2}=a_{3}\right\} \\
& =\left\{  \left(  u+a^{\prime},u\right)  :\left\vert u+a^{\prime}\right\vert
^{2}-\left\vert u\right\vert ^{2}=a_{3}\right\}  =\left\{  \left(
u+a^{\prime},u\right)  :2a^{\prime}\cdot u=a_{3}-\left\vert a^{\prime
}\right\vert ^{2}\right\} \\
& =\left\{  \left(  v,v-a^{\prime}\right)  :\left\vert v\right\vert
^{2}-\left\vert v-a^{\prime}\right\vert ^{2}=a_{3}\right\}  =\left\{  \left(
v,v-a^{\prime}\right)  :2a^{\prime}\cdot v=a_{3}+\left\vert a^{\prime
}\right\vert ^{2}\right\}  ,
\end{align*}
and $\sigma_{E_{a}}$ is surface measure on $E_{a}$. This can be seen from an
application of the implicit function theorem using $\left\vert a^{\prime
}\right\vert =\left\vert v-u\right\vert \geq\nu>0$ by the separation condition
(\ref{weak sep var}), or by direct calculation on the paraboloid using
separation. Thus we have obtained the convolution formula%
\begin{equation}
\mu^{1}\ast\mu^{2}\left(  a\right)  =\iint_{E_{a}}M_{\psi_{1}}\mathsf{Q}%
_{s_{1}}^{\eta}f_{1}\left(  v\right)  M_{\psi2}\mathsf{Q}_{s_{2}}^{\eta}%
f_{2}\left(  u\right)  d\sigma_{E_{a}}\left(  u,v\right)  ,\label{conv form}%
\end{equation}
with $\sigma_{E_{a}}$ as above.

We further conclude from the implicit function theorem that $a\rightarrow
\left\langle \delta_{a},\mu^{1}\ast\mu^{2}\right\rangle $ is a smooth function
of $a$ which is adapted to scale $\min\left\{  2^{-s_{1}},2^{-s_{2}}\right\}
=2^{-s_{2}}$. Moreover, the density `$\mu^{1}\ast\mu^{2}\left(  a\right)  $'
of the absolutely continuous measure $\mu^{1}\ast\mu^{2}=\left(  \mu^{1}%
\ast\mu^{2}\right)  \left(  a\right)  da$ is given by%
\begin{equation}
\mu^{1}\ast\mu^{2}\left(  a\right)  =\left\langle \delta_{a},\mu\ast
\omega\right\rangle ,\label{density}%
\end{equation}
and so altogether we have proved the following lemma.

\begin{lemma}
\label{scales}For $\mu^{1}$ and $\mu^{2}$ as above with $\left\Vert
f_{1}\right\Vert _{L^{\infty}}=\left\Vert f_{2}\right\Vert _{L^{\infty}}=1$
and $s_{1}\leq s_{2}$, the following derivative estimates hold,%
\[
\left\vert \nabla_{a}^{m}\left(  \mu^{1}\ast\mu^{2}\right)  \left(  a\right)
\right\vert \lesssim C_{m}2^{ms_{2}},\ \ \ \ \ \text{for }m\geq0\text{ and
}a\in\operatorname*{Supp}\left(  \mu^{1}\ast\mu^{2}\right)  ,
\]
i.e. $\mu^{1}\ast\mu^{2}$ is smooth at scale $\max\left\{  s_{1}%
,s_{2}\right\}  $.
\end{lemma}

\subsection{A norm calculation}

Fix a cube $U$ centered at the origin in $\mathbb{R}^{n-1}$ with side length
$\frac{1}{4}$ and let $\mathcal{G}$ be a dyadic grid as in \cite{Saw7}, and
for $s\in\mathbb{N}$, let $\mathcal{G}_{s}\left[  U\right]  $ consist of the
$\mathcal{G}$-dyadic subcubes of $U$ with side length $2^{-s}$. Now define the
smooth Alpert projection $\mathsf{Q}_{s}^{n-1,\eta}$ at level $s\in\mathbb{N}$
by%
\[
\mathsf{Q}_{s}^{n-1,\eta}f\equiv\sum_{I\in\mathcal{G}_{s}\left[  U\right]
}\bigtriangleup_{I;\kappa}^{n-1,\eta}f.
\]
We will use the following norm calculation valid in all dimensions $n\geq3$
and for all exponents $1<q<\infty$:%
\begin{align*}
&  \left\Vert \mathsf{Q}_{s}^{n-1,\eta}f\right\Vert _{L^{q}\left(  U\right)
}^{q}=\int_{\mathbb{R}^{n-1}}\left\vert \sum_{I\in\mathcal{G}_{s}\left[
U\right]  }\left\langle \left(  T_{\kappa,\eta}\right)  ^{-1}f,h_{I;\kappa
}^{n-1}\right\rangle h_{I;\kappa}^{n-1,\eta}\left(  x\right)  \right\vert
^{q}dx\\
&  \approx\sum_{I\in\mathcal{G}_{s}\left[  U\right]  }\left\vert \left\langle
\left(  T_{\kappa,\eta}\right)  ^{-1}f,h_{I;\kappa}^{n-1}\right\rangle
\right\vert ^{q}\int_{\mathbb{R}^{n-1}}\left\vert h_{I;\kappa}^{n-1,\eta
}\left(  x\right)  \right\vert ^{q}dx=\sum_{I\in\mathcal{G}_{s}\left[
U\right]  }\left\vert \breve{f}\left(  I\right)  \right\vert ^{q}%
\int_{\mathbb{R}^{n-1}}\left\vert h_{I;\kappa}^{n-1,\eta}\left(  x\right)
\right\vert ^{q}dx\\
&  \approx\sum_{I\in\mathcal{G}_{s}\left[  U\right]  }\left\vert \breve
{f}\left(  I\right)  \right\vert ^{q}\left(  \frac{1}{\sqrt{\left\vert
I\right\vert }}\right)  ^{q}\left\vert I\right\vert =\sum_{I\in\mathcal{G}%
_{s}\left[  U\right]  }\left\vert \breve{f}\left(  I\right)  \right\vert
^{q}\left\vert I\right\vert ^{1-\frac{q}{2}}=2^{s\left(  n-1\right)  \left(
\frac{q}{2}-1\right)  }\left\vert \breve{f}\right\vert _{\ell^{q}\left(
\mathcal{G}_{s}\left[  U\right]  \right)  }^{q}\ ,
\end{align*}
which gives%
\begin{equation}
\left\Vert \mathsf{Q}_{s}^{n-1,\eta}f\right\Vert _{L^{q}\left(  U\right)
}\approx2^{s\left(  n-1\right)  \left(  \frac{1}{2}-\frac{1}{q}\right)
}\left\vert \breve{f}\right\vert _{\ell^{q}\left(  \mathcal{G}_{s}\left[
U\right]  \right)  }\ .\label{norm f}%
\end{equation}
Note that when $q=n^{\ast}=\frac{2n}{n-1}$, we have%
\begin{equation}
\left\Vert \mathsf{Q}_{s}^{n-1,\eta}f\right\Vert _{L^{n^{\ast}}\left(
U\right)  }\approx2^{\frac{s}{n^{\ast}}}\left\vert \breve{f}\right\vert
_{\ell^{n^{\ast}}\left(  \mathcal{G}_{s}\left[  U\right]  \right)
}\ .\label{n star}%
\end{equation}

\section{A reduction to averaged and discretely mollified smooth Alpert
testing}

In this section we reduce the proof of the Fourier extension conjecture for a
smooth compact piece $S$ of the paraboloid $\mathbb{P}^{d-1}$, to an averaged
smooth Alpert testing condition with a perturbed pseudoprojection
$\widetilde{\mathsf{Q}}_{s}^{\eta,\mathcal{G}}$. But first we establish the
testing condition with the doubly smooth Alpert projection $\mathsf{Q}%
_{s}^{\eta,\mathcal{G}}$ at scale $s$ given by%
\[
\mathsf{Q}_{s}^{\eta,\mathcal{G}}f=\sum_{I\in\mathcal{G}_{s}}\left\langle
\left(  T_{\mathcal{G}}^{\eta}\right)  ^{-1}f,h_{I;\kappa}^{\eta}\right\rangle
h_{I;\kappa}^{\eta},
\]
in which the bounded invertible operator $\left(  T_{\mathcal{G}}^{\eta
}\right)  ^{-1}$ appears in the inner product. This is then removed in the
perturbed projection.

\begin{definition}
As mentioned in the introduction, we denote the expectation over the family
$\mathbb{G}$ of all dyadic grids by $\mathbb{E}_{\mathbb{G}}$, e.g. as in
\cite{Hyt}.
\end{definition}

\begin{definition}
\label{LSST}Let $d\geq3$, $\kappa\in\mathbb{N}$ and $\delta>0$. With $\psi$ a
fixed cutoff function such that $\operatorname*{Supp}\psi\subset U$, we say
that the \emph{Linear Single Scale Testing condition for the triple }$\left(
d,\kappa,\delta\right)  $, which we abbreviate as $\operatorname*{LSST}\left(
d,\kappa,\delta\right)  $, holds if for every $\varepsilon^{\prime}>0$, there
is a positive constant $C_{d,\varepsilon^{\prime},\delta,\kappa}$ such that%
\begin{equation}
\left\Vert \mathbb{E}_{\mathbb{G}}\mathcal{E}M_{\psi}^{\mathcal{G}}%
\mathsf{Q}_{s}^{\eta,\mathcal{G}}f\right\Vert _{L^{q}\left(  B\left(
0,2^{r}\right)  \right)  }\leq C_{\varepsilon^{\prime},\delta,\kappa
,q}2^{s\varepsilon^{\prime}}\left\Vert f\right\Vert _{L^{\infty}\left(
U\right)  },\label{lin d}%
\end{equation}
taken over all $s\in\mathbb{N}$ with $\frac{r}{1+\delta}\leq s\leq\frac
{r}{1-\delta}$, all smooth Alpert projections $\mathsf{Q}_{s}^{\eta
,\mathcal{G}}$ in the grid $\mathcal{G}$ with moment vanishing parameter
$\kappa$, and all $f\in L^{q}\left(  U\right)  $ with $q\geq\frac{2d}{d-1}$.
\end{definition}

\begin{theorem}
\label{Thm LSST}The Fourier extension conjecture for a smooth compact piece
$S$ of the paraboloid $\mathbb{P}^{d-1}$ in $\mathbb{R}^{d}$ is equivalent to
the linear single scale testing condition $\operatorname*{LSST}\left(
d,\kappa,\delta\right)  $.
\end{theorem}

Of course we may replace the $L^{\infty}\left(  U\right)  $ norm with the
$L^{p}\left(  U\right)  $ in Theorem \ref{Thm LSST} for any $q\leq p\leq
\infty$. As already alluded to, our proof of Theorem \ref{Thm LSST} proceeds
by using Theorem \ref{Thm 6 copy(1)} to reduce matters to a bilinear
inequality, and then expanding the function $f$ in (\ref{important}) in its
smooth Alpert decomposition. Using variants of arguments in \cite{RiSa} and
\cite{Saw7}, we compute that most of the scales can be controlled using
smoothness and moment vanishing of the smooth Alpert wavelets, leaving only
the scales described in the definition of $\otimes_{2}\Lambda_{\delta}^{\ast
}\left(  R\right)  $ in (\ref{def Lambda}) below, to be addressed. Then we use
a good lambda inequality to further reduce the standard multiplier to a
discrete multiplier, which is necessary if wants to average a quadratic
exponential sum into an oscillatory integral with periodic amplitude.

\subsection{A reduction of scales}

For $R=2^{r}\geq1$ define%
\begin{align}
\Lambda_{\delta}\left(  R\right)   & \equiv\sup_{\psi}\sup_{\left\Vert
f\right\Vert _{L^{\infty}}\leq1}\left\Vert \mathcal{E}M_{\psi}f\right\Vert
_{L^{q}\left(  B\left(  0,R\right)  \right)  }\ ,\label{def Lambda}\\
\otimes_{2}\Lambda_{\delta}\left(  R\right)   & \equiv\sup_{\psi_{1},\psi_{2}%
}\sup_{\left\Vert f\right\Vert _{L^{\infty}}\leq1}\left\Vert \prod_{j=1}%
^{2}\mathcal{E}M_{\psi_{j}}f\right\Vert _{L^{\frac{q}{2}}\left(  B\left(
0,R\right)  \right)  }^{\frac{1}{2}}\ ,\nonumber\\
\Lambda_{\delta}^{\ast}\left(  R\right)   & \equiv\sup_{\frac{r}{1+\delta
}<s<\frac{r}{1-\delta}}\sup_{\psi}\sup_{\left\Vert f\right\Vert _{L^{\infty}%
}\leq1}\left\Vert \mathbb{E}_{\mathbb{G}}\mathcal{E}M_{\psi}^{\mathcal{G}%
}\mathsf{Q}_{s}^{\eta,\mathcal{G}}f\right\Vert _{L^{q}\left(  B\left(
0,R\right)  \right)  }\ ,\nonumber\\
\text{and }\otimes_{2}\Lambda_{\delta}^{\ast}\left(  R\right)   & \equiv
\sup_{\delta r<s_{1}\leq s_{2}\text{ and }\frac{r}{1+\delta}<s_{2}<\frac
{r}{1-\delta}}\sup_{\psi_{1},\psi_{2}}\sup_{\left\Vert f\right\Vert
_{L^{\infty}}\leq1}\left\Vert \prod_{j=1}^{2}\mathbb{E}_{\mathbb{G}%
}\mathcal{E}M_{\psi_{j}}^{\mathcal{G}}\mathsf{Q}_{s_{j}}^{\eta,\mathcal{G}%
}f\right\Vert _{L^{\frac{q}{2}}\left(  B\left(  0,R\right)  \right)  }%
^{\frac{1}{2}},\nonumber
\end{align}
where the suprema are taken over all $f\in L^{\infty}\left(  U\right)  $, and
all functions $\psi_{k}\in C_{c}^{\infty}\left(  \pi^{\left(  2\right)  }%
U_{k}\right)  $ that are uniformly smooth at the scale $\ell\left(
U_{k}\right)  $ with constant $\left\Vert \psi_{k}\right\Vert _{C_{d,\delta
}^{\ast}}\leq A_{d,\delta}$ and with $\psi_{j}=1$ on $U_{j}$, where $\left(
U_{1},U_{2}\right)  \subset U^{2}$ are such that $U_{j}$ is an interior
grandchild of its grandparent $\pi^{\left(  2\right)  }U_{j}$ and such that
the $\nu$-disjoint condition (\ref{weak sep var}) holds. Note that all four
quantities are finite since $\mathcal{E}M_{\psi_{j}}f$, $\mathcal{E}%
M_{\psi_{j}}^{\mathcal{G}}f$, $M_{\psi}^{\mathcal{G}}\mathsf{Q}_{s}%
^{\eta,\mathcal{G}}f$ and $M_{\psi_{j}}^{\mathcal{G}}\mathsf{Q}_{s_{j}}%
^{\eta,\mathcal{G}}$ are bounded functions, thus permitting the absorptions
below. Note also that the expressions involving $\Lambda_{\delta}^{\ast}$
differ from those involving $\Lambda_{\delta}$ without the asterisk, in three ways:

\begin{enumerate}
\item the scales of the perturbed pseudoprojections are restricted,

\item there is an expectation in front of all functions,

\item and the multipliers are discrete.
\end{enumerate}

\begin{description}
\item[Note on averaging over grids] Independent expectations $\mathbb{E}%
_{\mathbb{G}}=\mathbb{E}_{\mathbb{G}}^{\mathcal{G}_{1}}$ and $\mathbb{E}%
_{\mathbb{G}}=\mathbb{E}_{\mathbb{G}}^{\mathcal{G}_{2}}$ are applied to the
two factors $\mathcal{E}M_{\psi_{1}}^{\mathcal{G}}\mathsf{Q}_{s_{1}}%
^{\eta,\mathcal{G}}f_{1}$ and $\mathcal{E}M_{\psi_{2}}^{\mathcal{G}}%
\mathsf{Q}_{s_{2}}^{\eta,\mathcal{G}}f_{2}$ in $\otimes_{2}\Lambda_{\delta
}^{\ast}\left(  R\right)  $, and as we will see, our deterministic bilinear
\emph{error} estimates below hold uniformly for every pair of grids
$\mathcal{G}_{1}$ and $\mathcal{G}_{2}$. Furthermore, since $U_{1}$ and
$U_{2}$ are $\nu$-disjoint, we may view $f_{j}$ as the restriction to $U_{j}$
of a single bounded function $f$, and we will sometimes write $f$ in place of
$f_{j}$.
\end{description}

Our goal in this section is to establish the good lambda inequality,%
\begin{equation}
\mathfrak{N}_{R}^{\left(  q\right)  }\lesssim R^{\varepsilon}\delta
r\sqrt{\mathfrak{N}_{R}^{\left(  q\right)  }\Lambda_{\delta}^{\ast}\left(
R\right)  }+R^{\varepsilon}R^{\frac{d\delta}{2}}\delta r\sqrt{1+\Lambda
_{\delta}^{\ast}\left(  R\right)  },\label{gli}%
\end{equation}
which is accomplished in the final subsection of this section.

With $R=2^{r}$ we first claim that the following reduction of scales in
$\otimes_{2}\Lambda_{\delta}\left(  R\right)  $ holds, where the expectation
operator $\mathbb{E}_{\mathbb{G}}$ makes its first appearance in the second
line :%
\begin{align}
& \sup_{\left\Vert f\right\Vert _{L^{\infty}}\leq1}\left\Vert \prod_{j=1}%
^{2}\mathcal{E}M_{\psi_{j}}f\right\Vert _{L^{\frac{q}{2}}\left(  A\left(
0,R\right)  \right)  }^{\frac{1}{2}}\label{claim red}\\
& \leq C\left(  \delta r\right)  ^{2}\sup_{\delta r<s_{1}\leq s_{2}\text{ and
}\frac{r}{1+\delta}<s_{2}<\frac{r}{1-\delta}}\sup_{\left\Vert f\right\Vert
_{L^{\infty}}\leq1}\left\Vert \prod_{j=1}^{2}\mathbb{E}_{\mathbb{G}%
}\mathcal{E}M_{\psi_{j}}\mathsf{Q}_{s_{j}}^{\eta,\mathcal{G}}f\right\Vert
_{L^{\frac{q}{2}}\left(  A\left(  0,R\right)  \right)  }^{\frac{1}{2}%
}\nonumber\\
& +C\left(  \delta r\right)  ^{2}\left(  \sup_{\frac{r}{1+\delta}<s_{2}%
<\frac{r}{1-\delta}}\sup_{\left\Vert f\right\Vert _{L^{\infty}}\leq
1}\left\Vert \mathbb{E}_{\mathbb{G}}\mathcal{E}M_{\psi_{2}}\mathsf{Q}_{s_{2}%
}^{\eta,\mathcal{G}}f\right\Vert _{L^{\frac{q}{2}}\left(  A\left(  0,R\right)
\right)  }^{\frac{1}{2}}+\left\Vert f\right\Vert _{\infty}\right)  .\nonumber
\end{align}
Note that once (\ref{claim red}) is proved for the annulus $A\left(
0,R\right)  $, then the same inequality holds for the ball $B\left(
0,R\right)  $ upon writing the ball as a union of annuli of radius $2^{-n}R$
for $n\geq0$.

Since $\mathcal{E}M_{\psi_{j}}f$ is independent of the grid $\mathcal{G}%
\in\mathbb{G}$, we have%
\begin{align*}
& \sup_{\left\Vert f\right\Vert _{L^{\infty}}\leq1}\left\Vert \prod_{j=1}%
^{2}\mathcal{E}M_{\psi_{j}}f\right\Vert _{L^{\frac{q}{2}}\left(  A\left(
0,R\right)  \right)  }^{\frac{1}{2}}\\
& =\sup_{\left\Vert f\right\Vert _{L^{\infty}}\leq1}\left\Vert \prod_{j=1}%
^{2}\mathbb{E}_{\mathbb{G}}\mathcal{E}M_{\psi_{j}}f\right\Vert _{L^{\frac
{q}{2}}\left(  A\left(  0,R\right)  \right)  }^{\frac{1}{2}}=\sup_{\left\Vert
f\right\Vert _{L^{\infty}}\leq1}\left\Vert \prod_{j=1}^{2}\mathbb{E}%
_{\mathbb{G}}\mathcal{E}M_{\psi_{j}}\sum_{s_{j}\in\mathbb{Z}}\mathsf{Q}%
_{s_{j}}^{\eta,\mathcal{G}}f\right\Vert _{L^{\frac{q}{2}}\left(  A\left(
0,R\right)  \right)  }^{\frac{1}{2}},
\end{align*}
upon expanding $f$ into its smooth Alpert expansion. We now note that in order
to prove (\ref{claim red}), it suffices to show that%
\begin{align}
& \sup_{\left\Vert f\right\Vert _{L^{\infty}}\leq1}\left\Vert \prod_{j=1}%
^{2}\mathbb{E}_{\mathbb{G}}\mathcal{E}M_{\psi_{j}}\sum_{s_{j}\in\mathbb{Z}%
}\mathsf{Q}_{s_{j}}^{\eta,\mathcal{G}}f\right\Vert _{L^{\frac{q}{2}}\left(
A\left(  0,R\right)  \right)  }^{\frac{1}{2}}\label{suff cutoff}\\
& \leq C\sup_{\left\Vert f\right\Vert _{L^{\infty}}\leq1}\left\Vert
\prod_{j=1}^{2}\mathbb{E}_{\mathbb{G}}\mathcal{E}M_{\psi_{j}}\sum_{\delta
r<s_{1}\leq s_{2}\text{ and }\frac{r}{1+\delta}<s_{2}<\frac{r}{1-\delta}%
}\mathsf{Q}_{s_{j}}^{\eta,\mathcal{G}}f\right\Vert _{L^{\frac{q}{2}}\left(
A\left(  0,R\right)  \right)  }^{\frac{1}{2}}\nonumber\\
& +C\left(  \delta r\right)  ^{2}\left(  \sup_{\frac{r}{1+\delta}<s_{2}%
<\frac{r}{1-\delta}}\sup_{\left\Vert f\right\Vert _{L^{\infty}}\leq
1}\left\Vert \mathbb{E}_{\mathbb{G}}\mathcal{E}M_{\psi_{2}}\mathsf{Q}_{s_{2}%
}^{\eta,\mathcal{G}}f\right\Vert _{L^{\frac{q}{2}}\left(  A\left(  0,R\right)
\right)  }^{\frac{1}{2}}+\left\Vert f\right\Vert _{\infty}\right)  ,\nonumber
\end{align}
where the multipliers $M_{\psi_{j}}$ are standard here, not discrete.

To prove (\ref{suff cutoff}), we will exploit the smoothness and moment
vanishing properties of smooth Alpert wavelets to obtain geometric decay
in\ convolutions of singular measures supported on the paraboloid, and we will
exploit the argument proving the reduction to a truncated inequality in
\cite[Subsubsection 1.4.1]{Saw7}. Recall that $\operatorname*{Supp}\psi
_{j}\subset U_{j}^{\ast}$ and $U_{j}$ is an interior grandchild of
$U_{j}^{\ast}=\pi^{\left(  2\right)  }U_{j}$. We use the following
decomposition of a function $f\in L^{p}\left(  \mathbb{R}^{d-1}\right)  $,%

\begin{align}
f  & =\sum_{I_{j}\in\mathcal{G}}\bigtriangleup_{I_{j};\kappa}^{\eta}%
f=A_{j}+B_{j}\text{ and }A_{j}\equiv\sum_{I\in\mathcal{G}^{\ast}\left[
U_{j}^{\ast}\right]  }\bigtriangleup_{I;\kappa}^{\eta}f_{j}\text{ and }%
B_{j}\equiv\sum_{I\in\mathcal{G\setminus G}^{\ast}\left[  U_{j}^{\ast}\right]
}\bigtriangleup_{I;\kappa}^{\eta}f_{j}\label{f decomp}\\
\text{and }\mathcal{G}^{\ast}\left[  U\right]   & \equiv\left\{
I\in\mathcal{G}:\left(  1+\eta\right)  I\cap U\neq\emptyset\text{ and }%
\ell\left(  I\right)  \leq\ell\left(  U\right)  \right\}  \ .\nonumber
\end{align}

\begin{lemma}
The function $\psi_{j}B_{j}$ is smooth at scale $\nu\approx1$ for $j=1,2$, and
bounded by $\left\Vert f\right\Vert _{L^{p}}$, for any $1<p\leq\infty$.
\end{lemma}

\begin{proof}
We have%
\[
\psi_{j}B_{j}=\sum_{k=1}^{\infty}\sum_{I\in\mathcal{N}\left(  \pi^{\left(
k\right)  }U_{0}\right)  }\psi_{j}\bigtriangleup_{I;\kappa}^{n-1,\eta}%
f=\sum_{k=1}^{\infty}\sum_{I\in\mathcal{N}\left(  \pi^{\left(  k\right)
}U_{0}\right)  }\left\langle \left(  S_{\kappa,\eta}\right)  ^{-1}%
f,h_{I;\kappa}\right\rangle \psi_{j}h_{I;\kappa}^{\eta}\ ,
\]
and so we see that $\psi_{j}B_{j}$ is smooth at scale $1$, compactly supported
and bounded by $\left\Vert f\right\Vert _{L^{p}}$, upon using that
(\textbf{i}) the functions $\psi_{j}h_{I;\kappa}^{\eta}$ are smooth and
compactly supported uniformly in $k$, and that (\textbf{ii}) we have the
pointwise inquality,
\begin{align*}
& \left\vert \sum_{k=1}^{\infty}\sum_{I\in\mathcal{N}\left(  \pi^{\left(
k\right)  }U_{0}\right)  }\left\langle \left(  S_{\kappa,\eta}\right)
^{-1}f,h_{I;\kappa}\right\rangle \psi_{j}h_{I;\kappa}^{\eta}\right\vert
\lesssim\left\Vert \psi_{j}\right\Vert _{L^{\infty}}\sum_{k=1}^{\infty
}\left\Vert \left(  S_{\kappa,\eta}\right)  ^{-1}f\right\Vert _{L^{p}%
}\left\Vert h_{I;\kappa}^{\eta}\right\Vert _{L^{\infty}}\left\Vert
h_{I;\kappa}\right\Vert _{L^{p^{\prime}}}\\
& \lesssim\left\Vert f\right\Vert _{L^{p}}\sum_{k=-\infty}^{v_{k}}%
2^{k\frac{d-1}{2}}\left(  2^{k\frac{n-1}{2}p^{\prime}}2^{-k\left(  n-1\right)
}\right)  ^{\frac{1}{p^{\prime}}}\lesssim\left\Vert f\right\Vert _{L^{p}}%
\sum_{k=-\infty}^{v}2^{k\frac{n-1}{p}}\lesssim\left\Vert f\right\Vert _{L^{p}%
}2^{v\frac{n-1}{p}}\lesssim\left\Vert f\right\Vert _{L^{p}},
\end{align*}
for any $1<p<\infty$.
\end{proof}

Now we continue with the proof of (\ref{suff cutoff}). Without loss of
generality we assume that $s_{1}\leq s_{2}$ and $R=2^{r}$. We will repeatedly
use the crude inequality%
\begin{align*}
\left\vert \widehat{f}\left(  I\right)  \right\vert  & =\left\vert
\left\langle \left(  S_{\kappa,\eta}^{\mathcal{D}}\right)  ^{-1}f,h_{I;\kappa
}\right\rangle \right\vert \leq\left\Vert \left(  S_{\kappa,\eta}%
^{\mathcal{D}}\right)  ^{-1}f\right\Vert _{L^{p}}\left\Vert h_{I;\kappa
}\right\Vert _{L^{p^{\prime}}}\\
& \lesssim\left\Vert f\right\Vert _{L^{p}\left(  U\right)  }\left\Vert
h_{I;\kappa}\right\Vert _{L^{\infty}}\left\Vert \mathbf{1}_{I}\right\Vert
_{L^{p^{\prime}}}\lesssim\left\Vert f\right\Vert _{L^{\infty}\left(  U\right)
}\frac{1}{\ell\left(  I\right)  }\ell\left(  I\right)  ^{\frac{1}{p^{\prime}}%
}=\ell\left(  I\right)  ^{-\frac{1}{p}}\left\Vert f\right\Vert _{L^{\infty}%
}\ ,
\end{align*}
for any $1<p<\infty$, which gives%
\[
\left\vert \bigtriangleup_{I;\kappa}^{\eta}f\right\vert =\left\vert
\widehat{f}\left(  I\right)  h_{I;\kappa}^{\eta}\right\vert \lesssim
\ell\left(  I\right)  ^{-\left(  1+\frac{1}{p}\right)  }\left\Vert
f\right\Vert _{L^{\infty}}\mathbf{1}_{\left(  1+\eta\ell\left(  I\right)
\right)  I},\ \ \ \ \ 1<p<\infty.
\]
We now prove (\ref{suff cutoff}) in three cases, followed by a wrapup.

\subsubsection{Case 1: $s_{2}\gg r$}

First suppose that $s_{1}\leq s_{2}$ and $s_{2}>\frac{r}{1-\delta}$. Then for
$\left\vert \xi\right\vert \approx2^{r}$, we have upon using $\kappa$-moment
vanishing of $\bigtriangleup_{I_{2};\kappa}^{\eta}f_{2}$ and differentiating
$e^{-i\Phi\left(  x\right)  \cdot\xi}=e^{-i\Phi\left(  c_{I_{2}}\right)
\cdot\xi}e^{-i\left[  \Phi\left(  x\right)  -\Phi\left(  c_{I_{2}}\right)
\right]  \cdot\xi}$ with respect to $x$, that for $I_{2}\in\mathcal{G}_{s_{2}%
}$,
\begin{align*}
& \left\vert \left[  \Phi_{\ast}M_{\psi_{2}}\bigtriangleup_{I_{2};\kappa
}^{\eta}f_{2}\right]  ^{\wedge}\left(  \xi\right)  \right\vert =\left\vert
\int e^{-iz\cdot\xi}\Phi_{\ast}M_{\psi_{2}}\bigtriangleup_{I_{2};\kappa}%
^{\eta}f_{2}\left(  z\right)  dz\right\vert =\left\vert \int e^{-i\Phi\left(
x\right)  \cdot\xi}\psi_{2}\left(  x\right)  \bigtriangleup_{I_{2};\kappa
}^{\eta}f_{3}\left(  x\right)  dx\right\vert \\
& \lesssim C_{\kappa}\int\left\vert 2^{-s_{2}}\xi\right\vert ^{\kappa
}\left\vert \bigtriangleup_{I_{2};\kappa}^{\eta}f_{2}\left(  x\right)
\right\vert dx\lesssim C_{\kappa}\int\left(  2^{-s_{2}}\left\vert
\xi\right\vert \right)  ^{\kappa}2^{s_{2}\left(  1+\frac{1}{p}\right)
}\left\Vert f_{2}\right\Vert _{L^{\infty}}\mathbf{1}_{\left(  1+\eta
\ell\left(  I_{2}\right)  \right)  I_{2}}dx\\
& \leq C_{\kappa}2^{-\left(  \kappa-1-\frac{1}{p}\right)  s_{2}}2^{\kappa
r}\left\Vert f_{2}\right\Vert _{L^{\infty}}=C_{\kappa}2^{-\kappa\left(
s_{2}-r\right)  }2^{\left(  1+\frac{1}{p}\right)  s_{2}}\left\Vert
f_{2}\right\Vert _{L^{\infty}},
\end{align*}
where the terms $B_{j}$ can only occur as $B_{1}A_{2}$, since the scale of
smoothness of any term $B_{j}$ is roughly $0$, which is less than $s_{1}$, so
in the ordering of scales, $B_{j}$ will occur first. Moreover, since this
estimate is independent of any pair of grids $\mathcal{G}_{1}$ and
$\mathcal{G}_{2}$, we will obtain the same estimate for the product of
averaged extension operators. Thus for the case $B_{1}A_{2}$ we obtain the
same estimate as for $A_{1}A_{2}$, which is
\begin{align*}
& \left(  \int_{A\left(  0,2^{r}\right)  }\left\vert \mathbb{E}_{\mathbb{G}%
}\mathcal{E}M_{\psi_{1}}\mathsf{Q}_{s_{1}}^{\eta}f_{1}\left(  \xi\right)
\mathbb{E}_{\mathbb{G}}\mathcal{E}M_{\psi_{2}}\mathsf{Q}_{s_{2}}^{\eta}%
f_{2}\left(  \xi\right)  \right\vert ^{\frac{q}{2}}\ d\xi\right)  ^{\frac
{2}{q}}\\
& \lesssim\left(  \int_{A\left(  0,2^{r}\right)  }\left(  \#\mathcal{G}%
_{s_{1}}\left[  U\right]  \#\mathcal{G}_{s_{2}}\left[  U\right]  \left\Vert
f_{1}\right\Vert _{L^{\infty}}\left\Vert f_{2}\right\Vert _{L^{\infty}%
}C_{\kappa}2^{-\kappa\left(  s_{2}-r\right)  }2^{-\frac{1}{p^{\prime}}s_{2}%
}\right)  ^{\frac{q}{2}}\ d\xi\right)  ^{\frac{2}{q}}\\
& \lesssim C_{\kappa}2^{-\kappa\left(  s_{2}-r\right)  }2^{1+\frac{1}{p}s_{2}%
}\left(  \int_{A\left(  0,2^{r}\right)  }d\xi\right)  ^{\frac{2}{q}}\left\Vert
f_{1}\right\Vert _{L^{\infty}}\left\Vert f_{2}\right\Vert _{L^{\infty}}\\
& \leq C_{\kappa}2^{-\kappa\left(  s_{2}-\left(  1-\delta\right)
s_{2}\right)  }2^{\left(  2\left(  d-1\right)  +1+\frac{1}{p}\right)  s_{2}%
}2^{\frac{2d}{q}\left(  1-\delta\right)  s_{2}}\left\Vert f_{1}\right\Vert
_{L^{\infty}}\left\Vert f_{2}\right\Vert _{L^{\infty}}\\
& =C_{\kappa}2^{-\kappa\delta s_{2}}2^{\left(  \frac{2d}{q}\left(
1-\delta\right)  +2d-\frac{1}{p^{\prime}}\right)  s_{2}}\left\Vert
f_{1}\right\Vert _{L^{\infty}}\left\Vert f_{2}\right\Vert _{L^{\infty}}\leq
C_{\kappa}2^{-s_{2}}\left\Vert f_{1}\right\Vert _{L^{\infty}}\left\Vert
f_{2}\right\Vert _{L^{\infty}}\ ,
\end{align*}

if $\kappa$ is chosen sufficiently large. Then summing in $s_{2}>\frac
{r}{1-\delta}$ gives
\begin{align}
& \left(  \int_{A\left(  0,2^{r}\right)  }\left(  \sum_{s_{1}\leq s_{2}\text{
and }s_{2}>\frac{r}{1-\delta}}\left\vert \mathbb{E}_{\mathbb{G}}%
\mathcal{E}M_{\psi_{1}}\mathsf{Q}_{s_{1}}^{\eta}f_{1}\left(  \xi\right)
\mathbb{E}_{\mathbb{G}}\mathcal{E}M_{\psi_{2}}\mathsf{Q}_{s_{2}}^{\eta}%
f_{2}\left(  \xi\right)  \right\vert \right)  ^{\frac{q}{2}}\ d\xi\right)
^{\frac{2}{q}}\label{first half}\\
& \leq\sum_{s_{1}\leq s_{2}\text{ and }s_{2}>\frac{r}{1-\delta}}\left(
\int_{B\left(  0,2^{r}\right)  }\left\vert \mathbb{E}_{\mathbb{G}}%
\mathcal{E}M_{\psi_{1}}\mathsf{Q}_{s_{1}}^{\eta}f_{1}\left(  \xi\right)
\mathbb{E}_{\mathbb{G}}\mathcal{E}M_{\psi_{2}}\mathsf{Q}_{s_{2}}^{\eta}%
f_{2}\left(  \xi\right)  \right\vert ^{\frac{q}{2}}\ d\xi\right)  ^{\frac
{2}{q}}\nonumber\\
& \lesssim C_{\kappa}\sum_{s_{2}>\frac{r}{1-\delta}}s_{2}2^{-s_{2}}\left\Vert
f_{1}\right\Vert _{L^{\infty}}\left\Vert f_{2}\right\Vert _{L^{\infty}%
}\lesssim C_{\kappa}r2^{-\frac{r}{1-\delta}}\left\Vert f_{1}\right\Vert
_{L^{p}}\left\Vert f_{2}\right\Vert _{L^{p}}.\nonumber
\end{align}

\subsubsection{Case 2: $s_{2}\ll r$}

Next we suppose $s_{1}<s_{2}\leq\frac{r}{1+\delta}$, and $I_{1}\in
\mathcal{G}_{s_{1}}\left[  U_{1}\right]  $ and $I_{2}\in\mathcal{G}_{s_{2}%
}\left[  U_{2}\right]  $. Then from Lemma \ref{scales}, we obtain that%
\[
F_{I_{1},I_{2}}\equiv\Phi_{\ast}\bigtriangleup_{I_{1};\kappa}^{\eta}f\ast
\Phi_{\ast}\bigtriangleup_{I_{2};\kappa}^{\eta}f\left(  z\right)
\]
is compactly supported in the $d$-dimensional rectangle $2\left(  \Phi\left(
I_{1}\right)  +\Phi\left(  I_{2}\right)  \right)  $, and smoothly adapted to
scale $2^{-s_{2}}$, or better in the case $B_{1}A_{2}$. Thus for $\xi\in
A\left(  0,2^{r}\right)  $ in the case $B_{1}A_{2}$, we get the same estimate
as we get for $A_{1}A_{2}$, which is
\begin{align*}
& \left\vert \left[  \Phi_{\ast}\bigtriangleup_{I_{1};\kappa}^{\eta}f\ast
\Phi_{\ast}\bigtriangleup_{I_{2};\kappa}^{\eta}f\right]  ^{\wedge}\left(
\xi\right)  \right\vert =\left\vert \int e^{-iz\cdot\xi}\Phi_{\ast}%
F_{I_{1},I_{2}}\left(  z\right)  dz\right\vert \lesssim\left(  \frac{2^{s_{2}%
}}{\left\vert \xi\right\vert }\right)  ^{N}\int\left\vert \nabla^{N}%
F_{I_{1},I_{2}}\right\vert \left(  z\right)  dz\\
& \leq C_{N}\left(  \frac{2^{s_{2}}}{\left\vert \xi\right\vert }\right)
^{N}\left\vert \Phi\left(  I_{1}\right)  +\Phi\left(  I_{2}\right)
\right\vert \left\Vert f_{1}\right\Vert _{L^{\infty}}\left\Vert f_{2}%
\right\Vert _{L^{\infty}}\approx C_{N}\left(  \frac{2^{s_{2}}}{2^{r}}\right)
^{N}2^{-s_{1}-2s_{2}}\left\Vert f_{1}\right\Vert _{L^{\infty}}\left\Vert
f_{2}\right\Vert _{L^{\infty}}\ .
\end{align*}

As a consequence of these estimates being independent of any pair of grids
$\mathcal{G}_{1}$ and $\mathcal{G}_{2}$, we have%
\begin{align*}
& \left(  \int_{A\left(  0,R\right)  }\left\vert \mathbb{E}_{\mathbb{G}%
}\mathcal{E}M_{\psi_{1}}\mathsf{Q}_{s_{1}}^{\eta}f_{1}\left(  \xi\right)
\mathbb{E}_{\mathbb{G}}\mathcal{E}M_{\psi_{2}}\mathsf{Q}_{s_{2}}^{\eta}%
f_{2}\left(  \xi\right)  \right\vert ^{\frac{q}{2}}d\xi\right)  ^{\frac{2}{q}%
}\\
& =\left(  \int_{A\left(  0,2^{r}\right)  }\left(  \left\vert \left[
\mathbb{E}_{\mathbb{G}}\Phi_{\ast}M_{\psi_{1}}\mathsf{Q}_{s_{1}}^{\eta}%
f_{1}\ast\mathbb{E}_{\mathbb{G}}M_{\psi_{2}}\mathsf{Q}_{s_{2}}^{\eta}%
f_{2}\right]  ^{\wedge}\left(  \xi\right)  \right\vert \right)  ^{\frac{q}{2}%
}d\xi\right)  ^{\frac{2}{q}}\\
& \leq\sum_{I_{1}\in\mathcal{G}_{s_{1}}\left[  U_{1}\right]  }\sum_{I_{2}%
\in\mathcal{G}_{s_{2}}\left[  U_{2}\right]  }\mathbb{E}_{\mathbb{G}}%
\mathbb{E}_{\mathbb{G}}\left(  \int_{A\left(  0,2^{r}\right)  }\left\vert
\left\{  \left[  \Phi_{\ast}M_{\psi_{1}}\mathsf{Q}_{s_{1}}^{\eta}f_{1}\ast
M_{\psi_{2}}\mathsf{Q}_{s_{2}}^{\eta}f_{2}\right]  ^{\wedge}\left(
\xi\right)  \right\}  \right\vert ^{\frac{q}{2}}d\xi\right)  ^{\frac{2}{q}}\\
& \leq\sum_{I_{1}\in\mathcal{G}_{s_{1}}\left[  U_{1}\right]  }\sum_{I_{2}%
\in\mathcal{G}_{s_{2}}\left[  U_{2}\right]  }\left(  \int_{A\left(
0,2^{r}\right)  }\left(  C_{N}\left(  \frac{2^{s_{2}}}{2^{r}}\right)
^{N}2^{-s_{1}-2s_{2}}\left\Vert f_{1}\right\Vert _{L^{\infty}}\left\Vert
f_{2}\right\Vert _{L^{\infty}}\right)  ^{\frac{q}{2}}d\xi\right)  ^{\frac
{2}{q}},
\end{align*}
which is approximately,%
\begin{align*}
& C_{N}2^{2s_{1}}2^{2s_{2}}\left(  \frac{2^{s_{2}}}{2^{r}}\right)
^{N}2^{-s_{1}-2s_{2}}\left\Vert f_{1}\right\Vert _{L^{\infty}}\left\Vert
f_{2}\right\Vert _{L^{\infty}}2^{r\frac{2d}{q}}\lesssim C_{N}\left(
\frac{2^{s_{2}}}{2^{r}}\right)  ^{N}2^{s_{1}}\left\Vert f_{1}\right\Vert
_{L^{\infty}}\left\Vert f_{2}\right\Vert _{L^{\infty}}2^{r\frac{2d}{q}}\\
& \leq C_{N}2^{\left(  \frac{2d}{q}-N\right)  r}2^{\left(  N+1\right)  s_{2}%
}\left\Vert f_{1}\right\Vert _{L^{\infty}}\left\Vert f_{2}\right\Vert
_{L^{\infty}}<C_{N}2^{\left(  \frac{2d}{q}-N\right)  r}2^{\left(  N+1\right)
\frac{r}{1+\delta}}\left\Vert f_{1}\right\Vert _{L^{\infty}}\left\Vert
f_{2}\right\Vert _{L^{\infty}}\ .
\end{align*}
If we choose $\frac{N+1}{1+\delta}+\frac{2d}{q}-N<-1$, then $2^{\left(
\frac{2d}{q}-N\right)  r}2^{\left(  N+1\right)  \frac{r}{1+\delta}}<2^{-r}$
and%
\[
\sum_{s_{1},s_{2}=1}^{\frac{r}{1+\delta}}\left(  \int_{A\left(  0,R\right)
}\left\vert \mathbb{E}_{\mathbb{G}}\mathcal{E}M_{\psi_{1}}\mathsf{Q}_{s_{1}%
}^{\eta}f_{1}\left(  \xi\right)  \mathbb{E}_{\mathbb{G}}\mathcal{E}M_{\psi
_{2}}\mathsf{Q}_{s_{2}}^{\eta}f_{2}\left(  \xi\right)  \right\vert ^{\frac
{q}{2}}d\xi\right)  ^{\frac{2}{q}}\lesssim C_{N}r^{2}2^{-r},
\]
for all $s_{2}\leq\frac{r}{1-\delta}$ and for $N$ sufficiently large,
independent of $s_{2}$. Altogether we have%
\begin{align}
& \left(  \int_{A\left(  0,2^{r}\right)  }\left(  \sum_{s_{1}\leq s_{2}%
\leq\frac{r}{1+\delta}}\left\vert \mathbb{E}_{\mathbb{G}}\mathcal{E}%
M_{\psi_{1}}\mathsf{Q}_{s_{1}}^{\eta}f_{1}\left(  \xi\right)  \mathbb{E}%
_{\mathbb{G}}\mathcal{E}M_{\psi_{2}}\mathsf{Q}_{s_{2}}^{\eta}f_{2}\left(
\xi\right)  \right\vert \right)  ^{\frac{q}{2}}\ d\xi\right)  ^{\frac{2}{q}%
}\label{second half}\\
& \leq\sum_{s_{1}\leq s_{2}\leq\frac{r}{1+\delta}}\left(  \int_{A\left(
0,2^{r}\right)  }\left\vert \mathbb{E}_{\mathbb{G}}\mathcal{E}M_{\psi_{1}%
}\mathsf{Q}_{s_{1}}^{\eta}f_{1}\left(  \xi\right)  \mathbb{E}_{\mathbb{G}%
}\mathcal{E}M_{\psi_{2}}\mathsf{Q}_{s_{2}}^{\eta}f_{2}\left(  \xi\right)
\right\vert ^{\frac{q}{2}}\ d\xi\right)  ^{\frac{2}{q}}\nonumber\\
& \lesssim\left(  C_{N}r^{2}2^{-r}\right)  \frac{r}{1-\delta}\lesssim
C_{N}r^{3}2^{-r}.\nonumber
\end{align}

\subsubsection{Case 3: $s_{1}\ll\ll r$}

Here we suppose that $s_{1}\leq\delta r$ and $\frac{r}{1+\delta}\leq s_{2}%
\leq\frac{r}{1-\delta}$. Now $M_{\psi_{1}}\mathsf{Q}_{\kappa,s_{1}}^{\eta
}f_{1}$ is smooth at scale $2^{-s_{1}}\geq2^{-\delta r}$, so if
$x_{\operatorname*{crit}}=\frac{-\xi^{\prime}}{\xi_{d}}\in
2\operatorname*{Supp}\psi_{1}=B\left(  0,2\right)  $, i.e. $\left\vert
\xi^{\prime}\right\vert \leq2\xi_{d}$, then $\xi_{d}\approx\left\vert
\xi\right\vert \approx2^{r}$, and stationary phase gives the estimate
\begin{align*}
& \left\vert \mathcal{E}M_{\psi_{1}}\mathsf{Q}_{\kappa,s_{1}}^{\eta}%
f_{1}\left(  \xi\right)  \right\vert =\left\vert \int e^{-i\left(  x\cdot
\xi^{\prime}+\left\vert x\right\vert ^{2}\xi_{d}\right)  }\left\{  \psi
_{1}\left(  x\right)  \mathsf{Q}_{\kappa,s_{1}}^{\eta}f_{1}\left(  x\right)
\right\}  dx\right\vert \\
& \lesssim C_{d}2^{d\delta r}\left(  \frac{1}{2^{r}}\right)  ^{\frac{d-1}{2}%
}\left\Vert f_{1}\right\Vert _{L^{\infty}}\approx2^{d\delta r}2^{-r\frac
{d-1}{2}}\left\Vert f_{1}\right\Vert _{L^{\infty}}.
\end{align*}
Of course if $x_{\operatorname*{crit}}\not \in 2\operatorname*{Supp}\psi_{1}$,
then we obtain a better rapid decay estimate from integration by parts.

Using H\"{o}lder's inequality, we thus obtain%
\begin{align}
& \sup_{0\leq s_{1}\leq\delta r\text{ and }\frac{r}{1+\delta}<s_{2}<\frac
{r}{1-\delta}}\sup_{\left\Vert f\right\Vert _{L^{\infty}}\leq1}\left\Vert
\left(  \mathbb{E}_{\mathbb{G}}\mathcal{E}M_{\psi_{1}}\mathsf{Q}_{\kappa
,s_{1}}^{\eta}f_{1}\right)  \left(  \mathbb{E}_{\mathbb{G}}\mathcal{E}%
M_{\psi_{2}}\mathsf{Q}_{\kappa,s_{2}}^{\eta}f_{2}\right)  \right\Vert
_{L^{\frac{q}{2}}\left(  A\left(  0,2^{r}\right)  \right)  }\label{third half}%
\\
& =\sup_{0\leq s_{1}\leq\delta r\text{ and }\frac{r}{1+\delta}<s_{2}<\frac
{r}{1-\delta}}\sup_{\left\Vert f\right\Vert _{L^{\infty}}\leq1}\left\Vert
\mathbb{E}_{\mathbb{G}}\mathcal{E}M_{\psi_{1}}\mathsf{Q}_{\kappa,s_{1}}^{\eta
}f\right\Vert _{L^{q}\left(  A\left(  0,2^{r}\right)  \right)  }\left\Vert
\mathbb{E}_{\mathbb{G}}\mathcal{E}M_{\psi_{2}}\mathsf{Q}_{\kappa,s_{2}}^{\eta
}f\right\Vert _{L^{q}\left(  A\left(  0,2^{r}\right)  \right)  }\nonumber\\
& \lesssim\sup_{\frac{r}{1+\delta}<s_{2}<\frac{r}{1-\delta}}\sup_{\left\Vert
f\right\Vert _{L^{\infty}}\leq1}2^{d\delta r}2^{-r\frac{d-1}{2}}\left\Vert
f_{1}\right\Vert _{L^{\infty}}2^{\frac{d}{q}r}\left\Vert \mathbb{E}%
_{\mathbb{G}}\mathcal{E}M_{\psi_{2}}\mathsf{Q}_{\kappa,s_{2}}^{\eta}%
f_{2}\right\Vert _{L^{q}\left(  A\left(  0,2^{r}\right)  \right)  }\nonumber\\
& \lesssim2^{d\delta r}\sup_{\frac{r}{1+\delta}<s_{2}<\frac{r}{1-\delta}}%
\sup_{\left\Vert f\right\Vert _{L^{\infty}}\leq1}\left\Vert \mathbb{E}%
_{\mathbb{G}}\mathcal{E}M_{\psi_{2}}\mathsf{Q}_{\kappa,s_{2}}^{\eta}%
f_{2}\right\Vert _{L^{q}\left(  A\left(  0,2^{r}\right)  \right)  },\nonumber
\end{align}
provided $q\geq\frac{2d}{d-1}$.

\subsubsection{Wrapup}

Now we can finish the proof that (\ref{suff cutoff}) holds. With $f_{k}%
=f\mid_{U_{k}}$ for some $f\in L^{\infty}\left(  U\right)  $ and $R=2^{r}$, we
obtain,%
\begin{align}
& \left(  \int_{A\left(  0,R\right)  }\left\vert \mathbb{E}_{\mathbb{G}%
}\mathcal{E}M_{\psi_{1}}f_{1}\left(  \xi\right)  \mathbb{E}_{\mathbb{G}%
}\mathcal{E}M_{\psi_{2}}f_{2}\left(  \xi\right)  \right\vert ^{\frac{q}{2}%
}\ d\xi\right)  ^{\frac{2}{q}}\label{start}\\
& =\left(  \int_{A\left(  0,2^{r}\right)  }\left\vert \mathbb{E}_{\mathbb{G}%
}\mathcal{E}M_{\psi_{1}}\sum_{I_{1}\in\mathcal{G}}\bigtriangleup_{I_{1}%
;\kappa}^{\eta}f\left(  \xi\right)  \mathbb{E}_{\mathbb{G}}\mathcal{E}%
M_{\psi_{2}}\sum_{I_{2}\in\mathcal{G}}\bigtriangleup_{I_{2};\kappa}^{\eta
}f\left(  \xi\right)  \right\vert ^{\frac{q}{2}}\ d\xi\right)  ^{\frac{2}{q}%
}\nonumber\\
& =\left(  \int_{A\left(  0,2^{r}\right)  }\left\vert \mathbb{E}_{\mathbb{G}%
}\mathcal{E}M_{\psi_{1}}\left(  A_{1}+B_{1}\right)  \left(  \xi\right)
\mathbb{E}_{\mathbb{G}}\mathcal{E}M_{\psi_{2}}\left(  A_{2}+B_{2}\right)
\left(  \xi\right)  \right\vert ^{\frac{q}{2}}\ d\xi\right)  ^{\frac{2}{q}%
}\nonumber\\
& \lesssim S_{\operatorname*{main}}+S_{\operatorname*{absorb}}%
+S_{\operatorname*{error}}\ ,\nonumber
\end{align}
where%
\begin{align*}
S_{\operatorname*{main}}  & \equiv\left(  \int_{A\left(  0,2^{r}\right)
}\left\vert \mathbb{E}_{\mathbb{G}}\mathcal{E}M_{\psi_{1}}A_{1}\left(
\xi\right)  \mathbb{E}_{\mathbb{G}}\mathcal{E}M_{\psi_{2}}A_{2}\left(
\xi\right)  \right\vert ^{\frac{q}{2}}\ d\xi\right)  ^{\frac{2}{q}},\\
S_{\operatorname*{absorb}}  & \equiv\sum_{\left(  E_{1},E_{2}\right)
\in\left\{  \left(  B_{1},A_{2}\right)  ,\left(  A_{1},B_{2}\right)  \right\}
}\left(  \int_{A\left(  0,2^{r}\right)  }\left\vert \mathbb{E}_{\mathbb{G}%
}\mathcal{E}M_{\psi_{1}}E_{1}\left(  \xi\right)  \mathbb{E}_{\mathbb{G}%
}\mathcal{E}M_{\psi_{2}}E_{2}\left(  \xi\right)  \right\vert ^{\frac{q}{2}%
}\ d\xi\right)  ^{\frac{2}{q}}\\
S_{\operatorname*{error}}  & \equiv\left(  \int_{A\left(  0,2^{r}\right)
}\left\vert \mathbb{E}_{\mathbb{G}}\mathcal{E}M_{\psi_{1}}B_{1}\left(
\xi\right)  \mathbb{E}_{\mathbb{G}}\mathcal{E}M_{\psi_{2}}B_{2}\left(
\xi\right)  \right\vert ^{\frac{q}{2}}\ d\xi\right)  ^{\frac{2}{q}}.
\end{align*}
We recall that $\operatorname*{Supp}\psi_{j}\subset U_{j}^{\ast}$ and $U_{j}$
is an interior grandchild of $U_{j}^{\ast}$, and from (\ref{f decomp}) we
have
\begin{align*}
f  & =\sum_{I_{j}\in\mathcal{G}}\bigtriangleup_{I_{j};\kappa}^{\eta}%
f=A_{j}+B_{j}\text{ and }A_{j}\equiv\sum_{I\in\mathcal{G}^{\ast}\left[
U_{j}^{\ast}\right]  }\bigtriangleup_{I;\kappa}^{\eta}f_{j}\text{ and }%
B_{j}\equiv\sum_{I\in\mathcal{G\setminus G}^{\ast}\left[  U_{j}^{\ast}\right]
}\bigtriangleup_{I;\kappa}^{\eta}f_{j}\\
\text{and }\mathcal{G}^{\ast}\left[  U\right]   & \equiv\left\{
I\in\mathcal{G}:\ I_{\eta}\cap U\neq\emptyset\text{ and }\ell\left(  I\right)
\leq\ell\left(  U\right)  \right\}  \ .
\end{align*}

Now we\ write $\nu_{j}=2^{-v_{j}}$ and use%
\[
M_{\psi_{1}}\sum_{I\in\mathcal{G}^{\ast}\left[  U_{1}^{\ast}\right]
}\bigtriangleup_{I;\kappa}^{\eta}f=M_{\psi_{1}}\sum_{s_{1}\geq v_{1}%
}\mathsf{Q}_{s_{1}}^{\eta,\mathcal{G}}f,
\]
to write, keeping in mind Remark \ref{think}, that says we may think of
$v_{j}$ as being $0$,%
\begin{align*}
& S_{\operatorname*{main}}=\left(  \int_{A\left(  0,2^{r}\right)  }\left\vert
\mathbb{E}_{\mathbb{G}}\mathcal{E}M_{\psi_{1}}A_{1}\left(  \xi\right)
\mathbb{E}_{\mathbb{G}}\mathcal{E}M_{\psi_{2}}A_{2}\left(  \xi\right)
\right\vert ^{\frac{q}{2}}\ d\xi\right)  ^{\frac{2}{q}}\\
& =\left(  \int_{A\left(  0,2^{r}\right)  }\left\vert \mathbb{E}_{\mathbb{G}%
}\mathcal{E}M_{\psi_{1}}\sum_{s_{1}\geq0}\mathsf{Q}_{s_{1}}^{\eta,\mathcal{G}%
}f\left(  \xi\right)  \mathbb{E}_{\mathbb{G}}\mathcal{E}M_{\psi_{2}}%
\sum_{s_{2}\geq0}\mathsf{Q}_{s_{2}}^{\eta,\mathcal{G}}f\left(  \xi\right)
\right\vert ^{\frac{q}{2}}\ d\xi\right)  ^{\frac{2}{q}}\\
& \leq\sum_{s_{1}\geq0}\sum_{s_{2}\geq0}\left(  \int_{A\left(  0,2^{r}\right)
}\left\vert \mathbb{E}_{\mathbb{G}}\mathcal{E}M_{\psi_{1}}\mathsf{Q}_{s_{1}%
}^{\eta,\mathcal{G}}f\left(  \xi\right)  \mathbb{E}_{\mathbb{G}}%
\mathcal{E}M_{\psi_{2}}\mathsf{Q}_{s_{2}}^{\eta,\mathcal{G}}f\left(
\xi\right)  \right\vert ^{\frac{q}{2}}\ d\xi\right)  ^{\frac{2}{q}}.
\end{align*}
Thus by (\ref{first half}), (\ref{second half}) and (\ref{third half}) in
Cases (\textbf{1}), (\textbf{2}) and (\textbf{3}) above, we have
\begin{align*}
& \left\{  \sum_{\delta r<s_{1}\leq s_{2}\text{ and }\frac{r}{1+\delta}%
<s_{2}<\frac{r}{1-\delta}}+\sum_{0\leq s_{1}\leq s_{2}\text{ and }s_{2}%
>\frac{r}{1-\delta}}+\sum_{0\leq s_{1}\leq s_{2}\text{ and }s_{2}<\frac
{r}{1+\delta}}+\sum_{0\leq s_{1}<\delta r\text{ and }\frac{r}{1+\delta}%
<s_{2}<\frac{r}{1-\delta}}\right\} \\
& \ \ \ \ \ \ \ \ \ \ \ \ \ \ \ \ \ \ \ \ \times\left\Vert \mathbb{E}%
_{\mathbb{G}}\mathcal{E}M_{\psi_{1}}\mathsf{Q}_{s_{1}}^{\eta}f\mathbb{E}%
_{\mathbb{G}}\mathcal{E}M_{\psi_{2}}\mathsf{Q}_{s_{2}}^{\eta}f\right\Vert
_{L^{\frac{q}{2}}\left(  A\left(  0,2^{r}\right)  \right)  }\\
& \ \ \ \ \ \ \ \ \ \ \lesssim2^{\delta r}\left\Vert f_{1}\right\Vert
_{L^{\infty}}\left(  \left\Vert f_{2}\right\Vert _{L^{\infty}}+\left\Vert
\mathbb{E}_{\mathbb{G}}\mathcal{E}M_{\psi_{2}}\mathsf{Q}_{s_{2}}^{\eta
}f\right\Vert _{L^{q}\left(  A\left(  0,2^{r}\right)  \right)  }\right)  \ ,
\end{align*}
and hence%
\begin{align}
& \ \ \ \ \ \ \ \ \ \ \ \ \ \ \ \ \ \ \ \ S_{\operatorname*{main}}%
\label{SAA}\\
& \lesssim\left\{  \sum_{\delta r<s_{1}\leq s_{2}\text{ and }\frac{r}%
{1+\delta}<s_{2}<\frac{r}{1-\delta}}+\sum_{\delta r\leq s_{1}\leq s_{2}\text{
and }s_{2}>\frac{r}{1-\delta}}+\sum_{0\leq s_{1}\leq s_{2}\text{ and }%
s_{2}<\frac{r}{1+\delta}}+\sum_{0\leq s_{1}<\delta r\text{ and }\frac
{r}{1+\delta}<s_{2}<\frac{r}{1-\delta}}\right\} \nonumber\\
& \ \ \ \ \ \ \ \ \ \ \ \ \ \ \ \times\left(  \int_{A\left(  0,2^{r}\right)
}\left\vert \mathbb{E}_{\mathbb{G}}\mathcal{E}M_{\psi_{1}}\mathsf{Q}_{s_{1}%
}^{\eta}f\left(  \xi\right)  \ \mathbb{E}_{\mathbb{G}}\mathcal{E}M_{\psi_{2}%
}\mathsf{Q}_{s_{2}}^{\eta}f\left(  \xi\right)  \right\vert ^{\frac{q}{2}%
}\ d\xi\right)  ^{\frac{2}{q}}\nonumber\\
& \lesssim\sum_{\delta r<s_{1}\leq s_{2}\text{ and }\frac{r}{1+\delta}\leq
s_{2}\leq\frac{r}{1-\delta}}\left(  \int_{A\left(  0,2^{r}\right)  }\left\vert
\mathbb{E}_{\mathbb{G}}\mathcal{E}M_{\psi_{1}}\mathsf{Q}_{s_{1}}^{\eta
}f\left(  \xi\right)  \ \mathbb{E}_{\mathbb{G}}\mathcal{E}M_{\psi_{2}%
}\mathsf{Q}_{s_{2}}^{\eta}f\left(  \xi\right)  \right\vert ^{\frac{q}{2}%
}\ d\xi\right)  ^{\frac{2}{q}}\nonumber\\
& \ \ \ \ \ \ \ \ \ \ \ \ \ \ \ +2^{d\delta r}\left\Vert f_{1}\right\Vert
_{L^{\infty}}\left(  \left\Vert f_{2}\right\Vert _{L^{\infty}}+\sup_{\frac
{r}{1+\delta}\leq s_{2}\leq\frac{r}{1-\delta}}\left\Vert \mathbb{E}%
_{\mathbb{G}}\mathcal{E}M_{\psi_{2}}\mathsf{Q}_{s_{2}}^{\eta}f\right\Vert
_{L^{q}\left(  A\left(  0,2^{r}\right)  \right)  }\right)  .\nonumber
\end{align}

Next we show that the error term%
\[
S_{\operatorname*{error}}=\left(  \int_{A\left(  0,2^{r}\right)  }\left\vert
\mathbb{E}_{\mathbb{G}}\mathcal{E}M_{\psi_{1}}\sum_{I\in\mathcal{G\setminus
G}^{\ast}\left[  U_{1}^{\ast}\right]  }\bigtriangleup_{I;\kappa}^{\eta
}f\left(  \xi\right)  \mathbb{E}_{\mathbb{G}}\mathcal{E}M_{\psi_{2}}\sum
_{I\in\mathcal{G\setminus G}^{\ast}\left[  U_{2}^{\ast}\right]  }%
\bigtriangleup_{I;\kappa}^{\eta}f\left(  \xi\right)  \right\vert ^{\frac{q}%
{2}}\ d\xi\right)  ^{\frac{2}{q}}\ ,
\]
satisfies%
\begin{equation}
S_{\operatorname*{error}}\lesssim\left\Vert f\right\Vert _{L^{\infty}}%
^{2}\ .\label{SBBB}%
\end{equation}
To see this we apply the argument proving the reduction to a truncated
inequality in \cite[Subsubsection 1.4.1]{Saw7}. Here are the
details\footnote{The setting is probabilistic in \cite[Subsubsection
1.4.1]{Saw7}, and so we repeat all the details for the determinisitic case
here.}.

Recall $\mathcal{G}^{\ast}\left[  U\right]  \equiv\left\{  I\in\mathcal{G}%
:\ I_{\eta}\cap U\neq\emptyset\text{ and }\ell\left(  I\right)  \leq
\ell\left(  U\right)  \right\}  $. Using the formulas for $A_{j}$ and $B_{j}$
we write%
\[
M_{\psi_{j}}f=M_{\psi_{j}}\sum_{I\in\mathcal{G}}\bigtriangleup_{I;\kappa
}^{\eta}f=M_{\psi_{j}}\sum_{I\in\mathcal{G}^{\ast}\left[  U_{j}^{\ast}\right]
}\bigtriangleup_{I;\kappa}^{\eta}f+M_{\psi_{j}}\sum_{k=0}^{\infty}\sum
_{I\in\mathcal{N}\left(  \pi^{\left(  k\right)  }U_{j}^{\ast}\right)
}\bigtriangleup_{I;\kappa}^{\eta}f=M_{\psi_{j}}L_{j}^{1}+M_{\psi_{j}}L_{j}%
^{2},
\]
since
\[
B_{j}=\sum_{I\in\mathcal{G\setminus G}^{\ast}\left[  U_{j}^{\ast}\right]
}\bigtriangleup_{I;\kappa}^{\eta}f_{j}=\sum_{k=0}^{\infty}\sum_{I\in
\mathcal{N}\left(  \pi^{\left(  k\right)  }U_{j}^{\ast}\right)  }%
\bigtriangleup_{I;\kappa}^{\eta}f.
\]
We will next show that%
\begin{equation}
\left\Vert TB_{j}\right\Vert _{L^{q}}=\left\Vert TL_{j}^{2}f\right\Vert
_{L^{q}}=\left\Vert TM_{\psi_{j}}\sum_{k=1}^{\infty}\left(  \sum
_{I\in\mathcal{N}\left(  \pi^{\left(  k\right)  }U_{j}^{\ast}\right)
}\bigtriangleup_{I;\kappa}^{\eta}f\right)  \right\Vert _{L^{q}}\lesssim
\left\Vert f\right\Vert _{L^{\infty}\left(  B\left(  0,1\right)  \right)
}\ ,\label{two ineq}%
\end{equation}
which will prove (\ref{SBBB}) upon using H\"{o}lder's inequality $\left\Vert
TB_{1}\ TB_{2}\right\Vert _{L^{\frac{q}{2}}}\leq\left\Vert TB_{1}\right\Vert
_{L^{q}}\left\Vert TB_{2}\right\Vert _{L^{q}}$.

So let $R_{j}^{\ast}$ be the rectangle $U_{j}^{\ast}\times\left[  0,2\right]
$. Then,
\begin{align*}
& \left\Vert TB_{j}\right\Vert _{L^{q}}=\left\Vert TM_{\psi_{j}}\sum
_{k=1}^{\infty}\sum_{I\in\mathcal{N}\left(  \pi^{\left(  k\right)  }%
U_{j}^{\ast}\right)  }\bigtriangleup_{I;\kappa}^{\eta}f\right\Vert _{L^{q}%
}=\left\Vert \mathcal{F}\Phi_{\ast}M_{\psi_{j}}L_{j}^{2}f\right\Vert _{L^{q}%
}=\left\Vert \mathcal{F}\Phi_{\ast}M_{\mathbf{1}_{U_{j}^{\ast}}}M_{\psi_{j}%
}L_{j}^{2}f\right\Vert _{L^{q}}\\
& =\left\Vert \mathcal{F}M_{\mathbf{1}_{R_{j}^{\ast}}}\Phi_{\ast}\left(
\psi_{j}L_{2}f\right)  \right\Vert _{L^{q}}=\left\Vert \left(  \mathcal{F}%
M_{\mathbf{1}_{R_{j}^{\ast}}}\mathcal{F}^{-1}\right)  \mathcal{F}\Phi_{\ast
}\psi_{j}L_{2}f\right\Vert _{L^{q}}\lesssim\left\Vert \mathcal{F}\Phi_{\ast
}\psi_{j}L_{2}f\right\Vert _{L^{q}}=\left\Vert T\psi_{j}L_{2}f\right\Vert
_{L^{q}}\ ,
\end{align*}
where%
\[
\psi_{j}L_{2}f=\sum_{k=1}^{\infty}\sum_{I\in\mathcal{N}\left(  \pi^{\left(
k\right)  }U_{0}\right)  }\psi_{j}\bigtriangleup_{I;\kappa}^{n-1,\eta}%
f=\sum_{k=1}^{\infty}\sum_{I\in\mathcal{N}\left(  \pi^{\left(  k\right)
}U_{0}\right)  }\left\langle \left(  S_{\kappa,\eta}\right)  ^{-1}%
f,h_{I;\kappa}\right\rangle \psi_{j}h_{I;\kappa}^{\eta}\ .
\]
Thus we see that $\psi_{j}L_{2}f$ is smooth at scale $1$, compactly supported
and bounded by $\left\Vert f\right\Vert _{L^{p}}$, upon using that
(\textbf{i}) the functions $\psi_{j}h_{I;\kappa}^{\eta}$ are smooth and
compactly supported uniformly in $k$, and that (\textbf{ii}) we have the
pointwise inquality,
\begin{align*}
& \left\vert \sum_{k=1}^{\infty}\sum_{I\in\mathcal{N}\left(  \pi^{\left(
k\right)  }U_{0}\right)  }\left\langle \left(  S_{\kappa,\eta}\right)
^{-1}f,h_{I;\kappa}\right\rangle \psi_{j}h_{I;\kappa}^{\eta}\right\vert
\lesssim\left\Vert \psi_{j}\right\Vert _{L^{\infty}}\sum_{k=1}^{\infty
}\left\Vert \left(  S_{\kappa,\eta}\right)  ^{-1}f\right\Vert _{L^{p}%
}\left\Vert h_{I;\kappa}^{\eta}\right\Vert _{L^{\infty}}\left\Vert
h_{I;\kappa}\right\Vert _{L^{p^{\prime}}}\\
& \lesssim\left\Vert f\right\Vert _{L^{p}}\sum_{k=0}^{\infty}2^{-k\frac
{d-1}{2}}\left(  2^{-k\frac{n-1}{2}p^{\prime}}2^{k\left(  n-1\right)
}\right)  ^{\frac{1}{p^{\prime}}}\lesssim\left\Vert f\right\Vert _{L^{p}}%
\sum_{k=0}^{\infty}2^{-k\frac{n-1}{p}}\lesssim\left\Vert f\right\Vert _{L^{p}%
},
\end{align*}
for any $1<p<\infty$.

Consequently, the Fourier transform $\widehat{\Phi_{\ast}\left(  \psi_{j}%
L_{j}^{2}f\right)  }$ of the smooth surface measure $\Phi_{\ast}\left(
\psi_{j}L_{j}^{2}f\right)  $ has decay
\[
\left\vert \widehat{\Phi_{\ast}\left(  \psi_{j}L_{j}^{2}f\right)  }\left(
\xi\right)  \right\vert \lesssim\left\Vert \psi_{j}\right\Vert _{C^{\frac
{n}{2}+2}}\left\Vert f\right\Vert _{L^{p}}\left(  1+\left\vert \xi\right\vert
\right)  ^{-\frac{n-1}{2}},
\]
by e.g. \cite[Theorem 1 page 348]{Ste2} or Theorem 29 in \cite{Saw7}. Since
this function is in $L^{q}\left(  \mathbb{R}^{n}\right)  $ for all
$q>\frac{2n}{n-1}$, it follows that%
\[
\left\Vert TM_{\psi_{j}}L_{j}^{2}f\right\Vert _{L^{q}}\lesssim C_{p,q}%
\left\Vert f\right\Vert _{L^{p}\left(  U\right)  },\ \ \ \ \text{\ for }p>1,
\]
which is (\ref{two ineq}). As already mentioned, this then proves (\ref{SBBB}).

Finally we turn to estimating the absorbable term $S_{\operatorname*{absorb}}%
$, and by symmetry it is enough to estimate the term%
\[
S_{\operatorname*{absorb}}\left(  BA\right)  \equiv\left(  \int_{A\left(
0,2^{r}\right)  }\left\vert \mathbb{E}_{\mathbb{G}}\mathcal{E}M_{\psi_{1}%
}B_{1}\left(  \xi\right)  \mathbb{E}_{\mathbb{G}}\mathcal{E}M_{\psi_{2}}%
A_{2}\left(  \xi\right)  \right\vert ^{\frac{q}{2}}\ d\xi\right)  ^{\frac
{2}{q}},
\]
where the factor $M_{\psi_{1}}B_{1}$ is smooth at scale roughly $1$. Applying
(\ref{third half}) in Case (\textbf{3}) above, we can reduce to the case where
$f$ is replaced by $\mathsf{Q}_{s_{2}}^{\eta}f$ with $\frac{r}{1+\delta}\leq
s_{2}\leq\frac{r}{1-\delta}$ since $M_{\psi_{1}}B_{1}$ corresponds to $s=0$ -
see the corresponding part of the argument used for $S_{\operatorname*{main}}$
above. Then by the Cauchy-Schwarz inequality, we have%
\begin{align*}
S_{\operatorname*{absorb}}\left(  BA\right)   & \lesssim\left(  \int_{A\left(
0,2^{r}\right)  }\left\vert \mathbb{E}_{\mathbb{G}}\mathcal{E}M_{\psi_{1}%
}B_{1}\left(  \xi\right)  \right\vert ^{q}d\xi\right)  ^{\frac{1}{q}}\left(
\int_{A\left(  0,2^{r}\right)  }\left\vert \mathbb{E}_{\mathbb{G}}%
\mathcal{E}M_{\psi_{2}}A_{2}\left(  \xi\right)  \right\vert ^{q}d\xi\right)
^{\frac{1}{q}}\\
& \lesssim C_{p,q}\left\Vert f\right\Vert _{L^{\infty}\left(  U\right)
}\left(  \int_{A\left(  0,2^{r}\right)  }\left\vert \mathbb{E}_{\mathbb{G}%
}\mathcal{E}M_{\psi_{2}}A_{2}\left(  \xi\right)  \right\vert ^{q}d\xi\right)
^{\frac{1}{q}},
\end{align*}
where the final inequality follows from (\ref{two ineq}).

Combining the estimates for $S_{\operatorname*{main}}$,
$S_{\operatorname*{error}}$\ and $S_{\operatorname*{absorb}}$ shows that%
\begin{align}
& \left(  \int_{A\left(  0,R\right)  }\left\vert \mathbb{E}_{\mathbb{G}%
}\mathcal{E}M_{\psi_{1}}f\left(  \xi\right)  \ \mathbb{E}_{\mathbb{G}%
}\mathcal{E}M_{\psi_{2}}f\left(  \xi\right)  \right\vert ^{\frac{q}{2}}%
\ d\xi\right)  ^{\frac{2}{q}}=S_{\operatorname*{main}}%
+S_{\operatorname*{absorb}}+S_{\operatorname*{error}}\label{end}\\
& \lesssim\sum_{\delta r<s_{1}\leq s_{2}\text{ and }\frac{r}{1+\delta}\leq
s_{2}\leq\frac{r}{1-\delta}}\left(  \int_{A\left(  0,2^{r}\right)  }\left\vert
\mathbb{E}_{\mathbb{G}}\mathcal{E}M_{\psi_{1}}\mathsf{Q}_{s_{1}}^{\eta
}f\left(  \xi\right)  \ \mathbb{E}_{\mathbb{G}}\mathcal{E}M_{\psi_{2}%
}\mathsf{Q}_{s_{2}}^{\eta}f\left(  \xi\right)  \right\vert ^{\frac{q}{2}%
}\ d\xi\right)  ^{\frac{2}{q}}\nonumber\\
& +2^{d\delta r}\left(  \left\Vert f\right\Vert _{L^{\infty}}+\sup_{\frac
{r}{1+\delta}\leq s_{2}\leq\frac{r}{1-\delta}}\left\Vert \mathbb{E}%
_{\mathbb{G}}\mathcal{E}M_{\psi_{2}}\mathsf{Q}_{s_{2}}^{\eta}f\right\Vert
_{L^{q}\left(  A\left(  0,2^{r}\right)  \right)  }\right)  ,\nonumber
\end{align}
where all of the multipliers on the right hand side are standard, and not yet
discrete. This completes the proof of (\ref{suff cutoff}) after dividing
through by $\left\Vert f_{1}\right\Vert _{L^{\infty}}\left\Vert f_{2}%
\right\Vert _{L^{\infty}}$ and taking the supremum over $\left\Vert
f_{j}\right\Vert _{L^{\infty}}\leq1$.

\subsection{The good lambda inequality}

We will now use (\ref{suff cutoff}), together with a good lambda inequality,
to show that for $R\geq1$,%
\begin{equation}
\mathfrak{N}_{R}^{\left(  q\right)  }\lesssim R^{\varepsilon}\otimes
_{2}\Lambda_{\delta}\left(  R\right)  \lesssim R^{\varepsilon}\delta
r\otimes_{2}\Lambda_{\delta}^{\ast}\left(  R\right)  +R^{\varepsilon}%
R^{\frac{d\delta}{2}}\delta r\sqrt{1+\Lambda_{\delta}^{\ast}\left(  R\right)
},\label{will show}%
\end{equation}
where the first inequality has already been proved in (\ref{important}). So
fix a pair $\left(  s_{1},s_{2}\right)  $ satisfying $\delta r\leq s_{1}\leq
s_{2}$ and $\frac{r}{1+\delta}\leq s_{2}\leq\frac{r}{1-\delta}$, so that we
have
\begin{align}
& \sup_{\left\Vert f\right\Vert _{L^{\infty}}\leq1}\left\Vert \prod_{j=1}%
^{2}\mathbb{E}_{\mathbb{G}}\mathcal{E}M_{\psi_{j}}\mathsf{Q}_{s_{j}}%
^{\eta,\mathcal{G}}f\right\Vert _{L^{\frac{q}{2}}\left(  A\left(  0,R\right)
\right)  }^{\frac{1}{2}}\leq C\sup_{\left\Vert f\right\Vert _{L^{\infty}}%
\leq1}\left\Vert \prod_{j=1}^{2}\mathbb{E}_{\mathbb{G}}\mathcal{E}M_{\psi_{j}%
}^{\mathcal{G}}\mathsf{Q}_{s_{j}}^{\eta,\mathcal{G}}f\right\Vert _{L^{\frac
{q}{2}}\left(  A\left(  0,R\right)  \right)  }^{\frac{1}{2}}\label{satis}\\
& +C\sup_{\left\Vert f\right\Vert _{L^{\infty}}\leq1}\left\Vert \left\{
\left(  \mathbb{E}_{\mathbb{G}}a_{1}^{s_{1}}-\mathbb{E}_{\mathbb{G}}%
b_{1}^{s_{1}}\right)  \mathbb{E}_{\mathbb{G}}a_{2}^{s_{2}}+\mathbb{E}%
_{\mathbb{G}}b_{1}^{s_{1}}\left(  \mathbb{E}_{\mathbb{G}}a_{2}^{s_{2}%
}-\mathbb{E}_{\mathbb{G}}b_{2}^{s_{2}}\right)  \right\}  \right\Vert
_{L^{\frac{q}{2}}\left(  A\left(  0,R\right)  \right)  }^{\frac{1}{2}%
}+C\nonumber\\
& \leq C\otimes_{2}^{\mathbb{E}}\Lambda_{\delta}^{\ast}\left(  R\right)
+C\sup_{\left\Vert f\right\Vert _{L^{\infty}}\leq1}\left\Vert \left\{  \left(
\mathbb{E}_{\mathbb{G}}a_{1}^{s_{1}}-\mathbb{E}_{\mathbb{G}}b_{1}^{s_{1}%
}\right)  \mathbb{E}_{\mathbb{G}}a_{2}^{s_{2}}+\mathbb{E}_{\mathbb{G}}%
b_{1}^{s_{1}}\left(  \mathbb{E}_{\mathbb{G}}a_{2}^{s_{2}}-\mathbb{E}%
_{\mathbb{G}}b_{2}^{s_{2}}\right)  \right\}  \right\Vert _{L^{\frac{q}{2}%
}\left(  A\left(  0,R\right)  \right)  }^{\frac{1}{2}}+C,\nonumber
\end{align}
where%
\begin{align*}
a_{j}^{s_{j}}  & =\mathcal{E}M_{\psi_{j}}\mathsf{Q}_{s_{j}}^{\eta,\mathcal{G}%
}f\text{ and }b_{j}^{s_{j}}=\mathcal{E}M_{\psi_{j}}^{\mathcal{G}}%
\mathsf{Q}_{s_{j}}^{\eta,\mathcal{G}}f,\\
\text{and }a_{j}^{s_{j}}  & =b_{j}^{s_{j}}+\left(  a_{j}^{s_{j}}-b_{j}^{s_{j}%
}\right)  ,
\end{align*}
and where we have used%
\[
\mathbb{E}_{\mathbb{G}}a_{1}^{s_{1}}\mathbb{E}_{\mathbb{G}}a_{2}^{s_{2}%
}-\mathbb{E}_{\mathbb{G}}b_{1}^{s_{1}}\mathbb{E}_{\mathbb{G}}b_{2}^{s_{2}%
}=\left(  \mathbb{E}_{\mathbb{G}}a_{1}^{s_{1}}-\mathbb{E}_{\mathbb{G}}%
b_{1}^{s_{1}}\right)  \mathbb{E}_{\mathbb{G}}a_{2}^{s_{2}}+\mathbb{E}%
_{\mathbb{G}}b_{1}^{s_{1}}\left(  \mathbb{E}_{\mathbb{G}}a_{2}^{s_{2}%
}-\mathbb{E}_{\mathbb{G}}b_{2}^{s_{2}}\right)  ,
\]
which allows us to exploit the small size $2^{-\frac{1}{2}s_{j}}$ of
$a_{j}^{s_{j}}-b_{j}^{s_{j}}$ provided $s_{j}\geq\delta r$ (see below).

Now we use Minkowski's inequality to estimate%
\begin{align}
& \left\Vert \left\{  \left(  \mathbb{E}_{\mathbb{G}}a_{1}^{s_{1}}%
-\mathbb{E}_{\mathbb{G}}b_{1}^{s_{1}}\right)  \mathbb{E}_{\mathbb{G}}%
a_{2}^{s_{2}}+\mathbb{E}_{\mathbb{G}}b_{1}^{s_{1}}\left(  \mathbb{E}%
_{\mathbb{G}}a_{2}^{s_{2}}-\mathbb{E}_{\mathbb{G}}b_{2}^{s_{2}}\right)
\right\}  \right\Vert _{L^{\frac{q}{2}}\left(  A\left(  0,R\right)  \right)
}^{\frac{1}{2}}\label{summands}\\
& \lesssim\left\Vert \left(  \mathbb{E}_{\mathbb{G}}a_{1}^{s_{1}}%
-\mathbb{E}_{\mathbb{G}}b_{1}^{s_{1}}\right)  \mathbb{E}_{\mathbb{G}}%
a_{2}^{s_{2}}\right\Vert _{L^{\frac{q}{2}}\left(  A\left(  0,R\right)
\right)  }^{\frac{1}{2}}+\left\Vert \mathbb{E}_{\mathbb{G}}b_{1}^{s_{1}%
}\left(  \mathbb{E}_{\mathbb{G}}a_{2}^{s_{2}}-\mathbb{E}_{\mathbb{G}}%
b_{2}^{s_{2}}\right)  \right\Vert _{L^{\frac{q}{2}}\left(  A\left(
0,R\right)  \right)  }^{\frac{1}{2}},\nonumber
\end{align}
and for the moment we will concentrate on the first summand. If we choose a
bump function $\psi_{j}^{\ast}$ satisfying $\psi_{j}^{\ast}\equiv1$ on
$\operatorname*{Supp}\psi_{j}$, then we have%
\begin{align*}
\mathbb{E}_{\mathbb{G}}a_{j}^{s_{j},\mathcal{G}}-\mathbb{E}_{\mathbb{G}}%
b_{j}^{s_{j},\mathcal{G}}  & =\mathbb{E}_{\mathbb{G}}\mathcal{E}\left(
M_{\psi_{j}}-M_{\psi_{j}}^{\mathcal{G}}\right)  \mathsf{Q}_{s_{j}}%
^{\eta,\mathcal{G}}f=\mathbb{E}_{\mathbb{G}}\mathcal{E}\psi_{j}^{\ast}\left(
M_{\psi_{j}}-M_{\psi_{j}}^{\mathcal{G}}\right)  \mathsf{Q}_{_{s}j}%
^{\eta,\mathcal{G}}f\\
& =\mathbb{E}_{\mathbb{G}}\mathcal{E}M_{\psi_{j}^{\ast}}g_{j}^{s_{j}%
,\mathcal{G}}=C2^{-s_{j}}\mathbb{E}_{\mathbb{G}}\mathcal{E}M_{\psi_{j}^{\ast}%
}h_{j}^{s_{j},\mathcal{G}}%
\end{align*}
where $g_{j}^{s_{j},\mathcal{G}}\equiv\left(  M_{\psi_{j}}-M_{\psi_{j}%
}^{\mathcal{G}}\right)  \mathsf{Q}_{s_{j}}^{\eta,\mathcal{G}}f$ satisfies
$\left\Vert g_{j}^{s_{j},\mathcal{G}}\right\Vert _{L^{\infty}}\leq C2^{-s_{j}%
}\left\Vert f\right\Vert _{L^{\infty}}$, and $h_{j}^{s_{j,\mathcal{G}}}%
=\frac{1}{C}2^{s_{j}}g_{j}^{s_{j}}$ has $L^{\infty}$ norm at most $1$. We also
have%
\begin{equation}
a_{j}^{s_{j},\mathcal{G}}=\mathbb{E}_{\mathbb{G}}\mathcal{E}M_{\psi_{j}%
}\mathsf{Q}_{s_{j}}^{\eta,\mathcal{G}}f=\mathbb{E}_{\mathbb{G}}C\mathcal{E}%
M_{\psi_{j}}k_{j}^{s_{j}}\text{ and }b_{j}^{s_{j},\mathcal{G}}=\mathbb{E}%
_{\mathbb{G}}\mathcal{E}M_{\psi_{j}}^{\mathcal{G}}\mathsf{Q}_{s_{j}}%
^{\eta,\mathcal{G}}f=\mathbb{E}_{\mathbb{G}}\mathcal{E}M_{\psi_{j}^{\ast}%
}M_{\psi_{j}}^{\mathcal{G}}\mathsf{Q}_{s_{j}}^{\eta,\mathcal{G}}%
f=\mathbb{E}_{\mathbb{G}}\mathcal{E}M_{\psi_{j}^{\ast}}k_{j}^{s_{j}%
,\mathcal{G}}\ ,\label{ab formual}%
\end{equation}
where $k_{j}^{s_{j},\mathcal{G}}=\frac{1}{C}\mathsf{Q}_{s_{j}}^{\eta
,\mathcal{G}}f$ has $L^{\infty}$ norm at most $1$. Finally we note that the
function%
\[
f^{\mathcal{G}}\equiv\left\{
\begin{array}
[c]{ccc}%
h_{1}^{s_{1},\mathcal{G}} & \text{ on } & \operatorname*{Supp}\psi_{1}\\
k_{2}^{s_{2},\mathcal{G}} & \text{ on } & \operatorname*{Supp}\psi_{2}%
\end{array}
\right.
\]
has $L^{\infty}$ norm at most $1$, and so%
\begin{align*}
\left\Vert \left(  \mathbb{E}_{\mathbb{G}}a_{1}^{s_{1}}-\mathbb{E}%
_{\mathbb{G}}b_{1}^{s_{1}}\right)  \mathbb{E}_{\mathbb{G}}a_{2}^{s_{2}%
}\right\Vert _{L^{\frac{q}{3}}\left(  A\left(  0,R\right)  \right)  }%
^{\frac{1}{3}}  & =\left\Vert \mathbb{E}_{\mathbb{G}}C2^{-s}\mathcal{E}%
M_{\psi_{1}^{\ast}}h_{1}^{s_{1}}\ \mathbb{E}_{\mathbb{G}}C\mathcal{E}%
M_{\psi_{2}}k_{2}^{s_{2}}\right\Vert _{L^{\frac{q}{2}}\left(  A\left(
0,R\right)  \right)  }^{\frac{1}{2}}\\
& =C2^{-\frac{1}{2}s_{1}}\left\Vert \mathbb{E}_{\mathbb{G}}\mathcal{E}%
M_{\psi_{1}^{\ast}}f^{\mathcal{G}}\ \mathbb{E}_{\mathbb{G}}\mathcal{E}%
M_{\psi_{2}}f\right\Vert _{L^{\frac{q}{2}}\left(  A\left(  0,R\right)
\right)  }^{\frac{1}{2}}.
\end{align*}

Now we further specialize the bump function $\psi_{j}^{\ast}$ to have the form
$\psi_{j}^{\ast}=\sum_{i\ \operatorname*{finite}}\psi_{j}^{i}$ where each
$\psi_{j}^{i}$ is a translate of $\psi_{j}$, and satisfies $\psi_{j}^{\ast
}\equiv1$ on $\operatorname*{Supp}\psi_{j}$. Then we get for $R\geq A$,%
\begin{align*}
\sup_{\left\Vert f\right\Vert _{L^{\infty}}\leq1}\left\Vert \left(
\mathbb{E}_{\mathbb{G}}a_{1}^{s_{1}}-\mathbb{E}_{\mathbb{G}}b_{1}^{s_{1}%
}\right)  \mathbb{E}_{\mathbb{G}}a_{2}^{s_{2}}\right\Vert _{L^{\frac{q}{2}%
}\left(  A\left(  0,R\right)  \right)  }^{\frac{1}{2}}  & \leq C\sup
_{\left\Vert f\right\Vert _{L^{\infty}}\leq1}2^{-\frac{1}{2}s_{1}}\left\Vert
\left\{  \mathbb{E}_{\mathbb{G}}\mathcal{E}\left(  \sum
_{i\ \operatorname*{finite}}M_{\psi_{1}^{i}}f^{\mathcal{G}}\right)
\ \mathbb{E}_{\mathbb{G}}\mathcal{E}M_{\psi_{2}}f^{\mathcal{G}}\right\}
\right\Vert _{L^{\frac{q}{2}}\left(  A\left(  0,R\right)  \right)  }^{\frac
{1}{2}}\\
& \leq C2^{-\frac{1}{2}s_{1}}\sup_{\left\Vert f\right\Vert _{L^{\infty}}\leq
1}\sum_{i\ \operatorname*{finite}}\left\Vert \mathbb{E}_{\mathbb{G}%
}\mathcal{E}M_{\psi_{1}^{i}}f^{\mathcal{G}}\ \mathbb{E}_{\mathbb{G}%
}\mathcal{E}M_{\psi_{2}}f^{\mathcal{G}}\right\Vert _{L^{\frac{q}{2}}\left(
A\left(  0,R\right)  \right)  }^{\frac{1}{2}}.
\end{align*}
Now we note that by the translation and dilation invariance (\ref{dil''}) in
the appendix, we may dilate and translate the functions $M_{\psi_{1}^{i}%
}f^{\mathcal{G}}$ and $M_{\psi_{2}}f^{\mathcal{G}}$\ in the product
$\mathcal{E}M_{\psi_{1}^{i}}f^{\mathcal{G}}\ \mathcal{E}M_{\psi_{2}%
}f^{\mathcal{G}}$ by $\delta_{t}\tau_{c}$ with $c\leq t\leq1$ if necessary, so
that both $\psi_{1}^{i}$ and $\psi_{2}$ have support in $U$ and are still
separated by $\upsilon$. Hence%
\begin{equation}
\left\Vert \mathbb{E}_{\mathbb{G}}\mathcal{E}M_{\psi_{1}^{i}}f^{\mathcal{G}%
}\ \mathbb{E}_{\mathbb{G}}\mathcal{E}M_{\psi_{2}}f^{\mathcal{G}}\right\Vert
_{L^{\frac{q}{2}}\left(  A\left(  0,R\right)  \right)  }^{\frac{1}{2}}%
\lesssim\otimes_{2}\Lambda_{\delta}\left(  R\right)  ,\text{ \ \ \ \ for each
}i,\label{trans inv}%
\end{equation}
because the parabolic dilation $\delta_{t}$ with $t\leq1$ expands the supports
of $M_{\psi_{1}^{i}}f^{\mathcal{G}}$ and $M_{\psi_{2}}f^{\mathcal{G}}$, which
results in the dual parabolic dilation $\frac{1}{t^{d+1}}\delta_{\frac{1}{t}}$
being applied to the Fourier extension, which shrinks its support inside the
ball $B\left(  0,R\right)  $. Similarly for $j=2$.

The remaining summand in (\ref{summands}) satisfies the same estimate with
$s_{2}$ in place of $s_{1}$ by using (\ref{ab formual}), the fact that
$\psi_{j}^{\ast}$ is a finite sum of translates of $\psi_{j}$, and by the
translation invariance in (\ref{trans inv}). If we assemble these estimates in
(\ref{satis}) we obtain the `good lambda inequality'%
\begin{align*}
& \otimes_{2}\Lambda_{\delta}\left(  R\right)  \lesssim C\left(  \delta
r\right)  ^{2}\otimes_{2}\Lambda_{\delta}^{\ast}\left(  R\right)  +C\left(
\delta r\right)  ^{2}\\
& +\left(  \delta r\right)  ^{2}\sup_{\delta r\leq s_{1}\leq s_{2}\text{ and
}\frac{r}{1+\delta}\leq s_{2}\leq\frac{r}{1-\delta}}\sup_{\psi_{1},\psi_{2}%
}\sup_{\left\Vert f\right\Vert _{L^{\infty}}\leq1}\left\Vert \left\{  \left(
\mathbb{E}_{\mathbb{G}}a_{1}^{s_{1}}-\mathbb{E}_{\mathbb{G}}b_{1}^{s_{1}%
}\right)  \mathbb{E}_{\mathbb{G}}a_{2}^{s_{2}}+\mathbb{E}_{\mathbb{G}}%
b_{1}^{s_{1}}\left(  \mathbb{E}_{\mathbb{G}}a_{2}^{s_{2}}-\mathbb{E}%
_{\mathbb{G}}b_{2}^{s_{2}}\right)  \right\}  \right\Vert _{L^{\frac{q}{2}%
}\left(  A\left(  0,R\right)  \right)  }^{\frac{1}{2}}\\
& \lesssim C\left(  \delta r\right)  ^{2}\otimes_{2}\Lambda_{\delta}^{\ast
}\left(  R\right)  +C\left(  \delta r\right)  ^{2}+C\left(  \delta r\right)
^{2}\left(  2^{-\frac{1}{2}\delta r}+2^{-\frac{1}{2}r}\right)  \otimes
_{2}\Lambda_{\delta}\left(  R\right)  ,
\end{align*}
where the factor in front of $\otimes_{2}\Lambda_{\delta}\left(  R\right)  $
at the end of the last line, can be made as small as we wish by taking $r$
large enough. Thus if we take $A=A\left(  \delta\right)  $ sufficiently large,
we obtain using (\ref{def Lambda}) that%
\[
\otimes_{2}\Lambda_{\delta}\left(  R\right)  \lesssim C\left(  \delta
r\right)  ^{2}\otimes_{2}\Lambda_{\delta}^{\ast}\left(  R\right)  +C\left(
\delta r\right)  ^{2}+\frac{1}{2}\otimes_{2}\Lambda_{\delta}\left(  R\right)
,\ \ \ \ \ R\geq A\left(  \delta\right)  ,
\]
and hence upon absorbing the last term into the left hand side, we get the
bilinear inequality,%
\[
\otimes_{2}\Lambda_{\delta}\left(  R\right)  \lesssim C\left(  \delta
r\right)  ^{2}\otimes_{2}\Lambda_{\delta}^{\ast}\left(  R\right)  +C\left(
\delta r\right)  ^{2},\ \ \ \ \ R\geq A\left(  \delta\right)  .
\]

Arguing in the same way, but with fewer technicalities, we obtain the
\emph{linear} analogue of this inequality, namely%
\[
\Lambda_{\delta}\left(  R\right)  \lesssim C\left(  \delta r\right)
^{2}\Lambda_{\delta}^{\ast}\left(  R\right)  +C\left(  \delta r\right)
^{2},\ \ \ \ \ R\geq A\left(  \delta\right)  .
\]

Plugging the linear and bilinear inequalities into (\ref{end}) results in
replacing the standard multipliers on the right hand side with discrete
multipliers, i.e.%
\begin{align*}
& \left(  \int_{A\left(  0,R\right)  }\left\vert \mathbb{E}_{\mathbb{G}%
}\mathcal{E}M_{\psi_{1}}f\left(  \xi\right)  \ \mathbb{E}_{\mathbb{G}%
}\mathcal{E}M_{\psi_{2}}f\left(  \xi\right)  \right\vert ^{\frac{q}{2}}%
\ d\xi\right)  ^{\frac{2}{q}}\\
& \lesssim\sum_{\delta r<s_{1}\leq s_{2}\text{ and }\frac{r}{1+\delta}\leq
s_{2}\leq\frac{r}{1-\delta}}\left(  \int_{A\left(  0,2^{r}\right)  }\left\vert
\mathbb{E}_{\mathbb{G}}\mathcal{E}M_{\psi_{1}}\mathsf{Q}_{s_{1}}^{\eta
}f\left(  \xi\right)  \ \mathbb{E}_{\mathbb{G}}\mathcal{E}M_{\psi_{2}%
}\mathsf{Q}_{s_{2}}^{\eta}f\left(  \xi\right)  \right\vert ^{\frac{q}{2}%
}\ d\xi\right)  ^{\frac{2}{q}}\\
& +2^{d\delta r}\left\Vert f\right\Vert _{L^{\infty}}\left(  \left\Vert
f\right\Vert _{L^{\infty}}+\sup_{\frac{r}{1+\delta}\leq s_{2}\leq\frac
{r}{1-\delta}}\left\Vert \mathbb{E}_{\mathbb{G}}\mathcal{E}M_{\psi_{2}%
}\mathsf{Q}_{s_{2}}^{\eta}f\right\Vert _{L^{q}\left(  A\left(  0,2^{r}\right)
\right)  }\right)
\end{align*}
which for $R\geq A\left(  \delta\right)  $ is at most%
\begin{align*}
& \lesssim\left(  \delta r\right)  ^{2}\sum_{\delta r<s_{1}\leq s_{2}\text{
and }\frac{r}{1+\delta}\leq s_{2}\leq\frac{r}{1-\delta}}\left(  \int_{A\left(
0,2^{r}\right)  }\left\vert \mathbb{E}_{\mathbb{G}}\mathcal{E}M_{\psi_{1}%
}^{\mathcal{G}}\mathsf{Q}_{s_{1}}^{\eta}f\left(  \xi\right)  \ \mathbb{E}%
_{\mathbb{G}}\mathcal{E}M_{\psi_{2}}^{\mathcal{G}}\mathsf{Q}_{s_{2}}^{\eta
}f\left(  \xi\right)  \right\vert ^{\frac{q}{2}}\ d\xi\right)  ^{\frac{2}{q}%
}\\
& +2^{d\delta r}\left(  \delta r\right)  ^{2}\left\Vert f\right\Vert
_{L^{\infty}}\left(  \left\Vert f\right\Vert _{L^{\infty}}+\sup_{\frac
{r}{1+\delta}\leq s_{2}\leq\frac{r}{1-\delta}}\left\Vert \mathbb{E}%
_{\mathbb{G}}\mathcal{E}M_{\psi_{2}}^{\mathcal{G}}\mathsf{Q}_{s_{2}}^{\eta
}f\right\Vert _{L^{q}\left(  A\left(  0,2^{r}\right)  \right)  }\right)
+C\left(  \delta r\right)  ^{2}2^{d\delta r}.
\end{align*}

Thus in the notation of (\ref{def Lambda}), we have shown that%
\begin{align*}
\left[  \otimes_{2}\Lambda_{\delta}\left(  R\right)  \right]  ^{2}  &
\lesssim\left(  \delta r\right)  ^{2}\left[  \otimes_{2}\Lambda_{\delta}%
^{\ast}\left(  R\right)  \right]  ^{2}+2^{d\delta r}\left(  \delta r\right)
^{2}\left(  1+\Lambda_{\delta}^{\ast}\left(  R\right)  \right)
,\ \ \ \ \ R\geq A\left(  \delta\right)  ,\\
\text{i.e. }\otimes_{2}\Lambda_{\delta}\left(  R\right)   & \lesssim\delta
r\otimes_{2}\Lambda_{\delta}^{\ast}\left(  R\right)  +\sqrt{2^{d\delta
r}\left(  \delta r\right)  ^{2}\left(  1+\Lambda_{\delta}^{\ast}\left(
R\right)  \right)  },\ \ \ \ \ R\geq A\left(  \delta\right)  ,
\end{align*}
and combining this with (\ref{important}), namely $\mathfrak{N}_{R}^{\left(
q\right)  }\leq C_{\varepsilon,\delta,\kappa,q}R^{\varepsilon}\otimes
_{2}\Lambda_{\delta}\left(  R\right)  $, we obtain
\begin{equation}
\mathfrak{N}_{R}^{\left(  q\right)  }\lesssim R^{\varepsilon}\otimes
_{2}\Lambda_{\delta}\left(  R\right)  \lesssim R^{\varepsilon}\delta
r\otimes_{2}\Lambda_{\delta}^{\ast}\left(  R\right)  +R^{\varepsilon}%
R^{\frac{d\delta}{2}}\delta r\sqrt{1+\Lambda_{\delta}^{\ast}\left(  R\right)
},\ \ \ \ \ R\geq A\left(  \delta\right)  ,\label{important'}%
\end{equation}
which completes the proof of (\ref{will show}) since for $R<A\left(
\delta\right)  $, we simply use%
\[
\mathfrak{N}_{R}^{\left(  q\right)  }\lesssim R^{d}\leq A\left(
\delta\right)  ^{d}=C_{\delta}\text{.}%
\]

\subsection{The linear discretely mollified averaged testing condition}

We can now prove Theorem \ref{Thm LSST}.

\begin{proof}
[Proof of Theorem \ref{Thm LSST}]Clearly $\operatorname*{LSST}\left(
d,\kappa,\delta\right)  $ is necessary for the Fourier extension conjecture.
Conversely, suppose that (\ref{lin d}) holds for some $q\geq\frac{2d}{d-1}$.
Let $\delta r<s_{1}\leq s_{2}$ and $\frac{r}{1+\delta}<s_{2}<\frac{r}%
{1-\delta}$. Then H\"{o}lder's inequality gives%
\begin{align}
& \left\Vert \prod_{j=1}^{2}\mathbb{E}_{\mathbb{G}}\mathcal{E}M_{\psi_{j}%
}^{\mathcal{G}}\mathsf{Q}_{s_{j}}^{\eta,\mathcal{G}}f\right\Vert _{L^{\frac
{q}{2}}\left(  A\left(  0,R\right)  \right)  }=\left(  \int_{A\left(
0,R\right)  }\left\vert \left(  \mathbb{E}_{\mathbb{G}}\mathcal{E}M_{\psi_{1}%
}^{\mathcal{G}}\mathsf{Q}_{s_{1}}^{\eta,\mathcal{G}}f\left(  \xi\right)
\right)  \left(  \mathbb{E}_{\mathbb{G}}\mathcal{E}M_{\psi_{2}}^{\mathcal{G}%
}\mathsf{Q}_{s_{2}}^{\eta,\mathcal{G}}f\left(  \xi\right)  \right)
\right\vert ^{\frac{q}{2}}d\xi\right)  ^{\frac{2}{q}}\label{bilin to lin}\\
& \leq\left(  \int_{A\left(  0,2^{r}\right)  }\left\vert \mathbb{E}%
_{\mathbb{G}}\mathcal{E}M_{\psi_{1}}^{\mathcal{G}}\mathsf{Q}_{s_{1}}%
^{\eta,\mathcal{G}}f\left(  \xi\right)  \right\vert ^{q}d\xi\right)
^{\frac{1}{q}}\left(  \int_{A\left(  0,2^{r}\right)  }\left\vert
\mathbb{E}_{\mathbb{G}}\mathcal{E}M_{\psi_{2}}^{\mathcal{G}}\mathsf{Q}_{s_{2}%
}^{\eta,\mathcal{G}}f\left(  \xi\right)  \right\vert ^{q}d\xi\right)
^{\frac{1}{q}}\nonumber\\
& \leq\left(  \mathfrak{N}_{R}^{\left(  q\right)  }\left\Vert f\right\Vert
_{L^{\infty}\left(  U\right)  }\right)  \left(  \int_{A\left(  0,2^{r}\right)
}\left\vert \mathbb{E}_{\mathbb{G}}\mathcal{E}M_{\psi_{2}}^{\mathcal{G}%
}\mathsf{Q}_{s_{2}}^{\eta,\mathcal{G}}f\left(  \xi\right)  \right\vert
^{q}d\xi\right)  ^{\frac{1}{q}}.\nonumber
\end{align}
Thus we have%
\[
\left[  \otimes_{2}\Lambda_{\delta}^{\ast}\left(  R\right)  \right]  ^{2}%
\leq\mathfrak{N}_{R}^{\left(  q\right)  }\Lambda_{\delta}^{\ast}\left(
R\right)  ,
\]
and combining this with (\ref{will show}) gives,%
\begin{align*}
\mathfrak{N}_{R}^{\left(  q\right)  }  & \lesssim R^{\varepsilon}\delta
r\otimes_{2}\Lambda_{\delta}^{\ast}\left(  R\right)  +R^{\varepsilon}%
R^{\frac{d\delta}{2}}\delta r\sqrt{1+\Lambda_{\delta}^{\ast}\left(  R\right)
}\\
& \lesssim R^{\varepsilon}\delta r\sqrt{\mathfrak{N}_{R}^{\left(  q\right)
}\Lambda_{\delta}^{\ast}\left(  R\right)  }+R^{\varepsilon}R^{\frac{d\delta
}{2}}\delta r\sqrt{1+\Lambda_{\delta}^{\ast}\left(  R\right)  },
\end{align*}
which establishes the good lambda inequality (\ref{gli}) mentioned earlier as
the goal of this section.

Thus either%
\[
\mathfrak{N}_{R}^{\left(  q\right)  }\lesssim R^{\varepsilon}\delta
r\sqrt{\mathfrak{N}_{R}^{\left(  q\right)  }\Lambda_{\delta}^{\ast}\left(
R\right)  }\text{ or }\mathfrak{N}_{R}^{\left(  q\right)  }\lesssim
R^{\varepsilon}R^{\frac{d\delta}{2}}\delta r\sqrt{1+\Lambda_{\delta}^{\ast
}\left(  R\right)  },
\]
and so altogether using (\ref{lin d}) we have%
\[
\mathfrak{N}_{R}^{\left(  q\right)  }\lesssim R^{\varepsilon}\delta
r\sqrt{\mathfrak{N}_{R}^{\left(  q\right)  }2^{\varepsilon^{\prime}s}}\text{
or }\mathfrak{N}_{R}^{\left(  q\right)  }\lesssim R^{\varepsilon}%
R^{\frac{d\delta}{2}}\delta r\sqrt{2^{\varepsilon^{\prime}s}},
\]
hence%
\[
\mathfrak{N}_{R}^{\left(  q\right)  }\lesssim\max\left\{  R^{2\varepsilon
}\left(  \delta r\right)  ^{2}2^{\varepsilon^{\prime}r},R^{\varepsilon
}R^{\frac{d\delta}{2}}\delta r2^{\frac{\varepsilon^{\prime}}{2}s}\right\}
\lesssim R^{2\varepsilon}\left(  \log_{2}R\right)  ^{2}R^{\varepsilon^{\prime
}}R^{\frac{d\delta}{2}}\leq R^{2\varepsilon+\varepsilon+\varepsilon^{\prime
}+\frac{d\delta}{2}}.
\]
This and Theorem \ref{main} now complete the proof of Theorem \ref{Thm LSST}.
\end{proof}

\subsubsection{An improved projection}

At this point we have a characterization (\ref{lin d}) of the Fourier
extension conjecture in terms of the norms $\left\Vert \mathbb{E}_{\mathbb{G}%
}\mathcal{E}M_{\psi}^{\mathcal{G}}\mathsf{Q}_{s}^{\eta,\mathcal{G}%
}f\right\Vert _{L^{q}\left(  B\left(  0,2^{\frac{s}{1-\delta}}\right)
\right)  }$ taken over the balls $B\left(  0,2^{\frac{s}{1-\delta}}\right)  $
of averaged extensions of projections $\mathsf{Q}_{s}^{\eta,\mathcal{G}}f$. We
now show that we can replace the projections $\mathsf{Q}_{s}^{\eta
,\mathcal{G}}$ with perturbed pseudoprojections $\widetilde{\mathsf{Q}}%
_{s}^{\eta,\mathcal{G}}$ defined below that enjoy both a local property and a
periodicity property.

\begin{definition}
For $s\in\mathbb{N}$ and $\mathcal{G}\in\mathbb{G}$ denote by%
\[
\widetilde{\mathsf{Q}}_{s}^{\eta,\mathcal{G}}f\equiv\mathsf{Q}_{s}%
^{\eta,\mathcal{G}}T_{\mathcal{G}}^{\eta}f=\sum_{I\in\mathcal{G}_{s}%
}\left\langle f,h_{I;\kappa}^{\eta}\right\rangle h_{I;\kappa}^{\eta}\ ,
\]
the \emph{perturbed} pseudoprojection obtained by composing the projection
$\mathsf{Q}_{s}^{\eta,\mathcal{G}}$ with $T_{\mathcal{G}}^{\eta}$.
\end{definition}

Note that $\widetilde{\mathsf{Q}}_{s}^{\eta,\mathcal{G}}$ is a perturbation of
$\mathsf{Q}_{s}^{\eta,\mathcal{G}}$ in the sense that%
\begin{equation}
\mathsf{Q}_{s}^{\eta,\mathcal{G}}f-\widetilde{\mathsf{Q}}_{s}^{\eta
,\mathcal{G}}f=\sum_{I\in\mathcal{G}_{s}}\left\langle \left[  \left(
T_{\mathcal{G}}^{\eta}\right)  ^{-1}-\mathbb{I}\right]  f,h_{I;\kappa}^{\eta
}\right\rangle h_{I;\kappa}^{\eta}=\widetilde{\mathsf{Q}}_{s}^{\eta
,\mathcal{G}}\left[  \left(  T_{\mathcal{G}}^{\eta}\right)  ^{-1}%
-\mathbb{I}\right]  f\ ,\label{pert id}%
\end{equation}
where the operator norm $\left\Vert \left(  T_{\mathcal{G}}^{\eta}\right)
^{-1}-\mathbb{I}\right\Vert _{L^{p}\left(  \mathbb{R}^{d-1}\right)
\rightarrow L^{p}\left(  \mathbb{R}^{d-1}\right)  }$ is $O\left(  \eta\right)
$.

\begin{description}
\item[Important point] The perturbed projection has the property that
$M_{\psi}^{\mathcal{G}}\widetilde{\mathsf{Q}}_{s}^{\eta,\mathcal{G}%
}=\widetilde{\mathsf{Q}}_{s}^{\eta,\mathcal{G}}$ if $\operatorname*{Supp}%
f\subset\frac{1}{2}U$ and $\psi$ in Definition \ref{LSST} is chosen to have
support in $U$ and to equal $1$ on $\frac{1}{2}U$, because now the support of
$f$ in the inner product $\left\langle f,h_{I;\kappa}^{\eta}\right\rangle $
restricts the sums to the set where $M_{\psi}^{\mathcal{G}}$ is $1$. As a
consequence, we can remove the discrete multiplier $M_{\psi}^{\mathcal{G}}$,
with $\psi$ as above, from in front of $\widetilde{\mathsf{Q}}_{s}%
^{\eta,\mathcal{G}}$, and reinstate it when necessary below. From now on we
suppose that$\ \psi$ is smooth and $f$ is bounded with
\begin{equation}
\mathbf{1}_{\frac{1}{2}U}\leq\psi\leq\mathbf{1}_{\frac{3}{4}U}\text{ and
}\operatorname*{Supp}f\subset\frac{1}{2}U.\label{psi and f}%
\end{equation}

\end{description}

\begin{corollary}
\label{Cor LSST}Let $\frac{2d}{d-1}\leq q\leq p<\infty$. Then the inequality
(\ref{lin d}) holds if and only if%
\begin{equation}
\left\Vert \mathbb{E}_{\mathbb{G}}\mathcal{E}\widetilde{\mathsf{Q}}_{s}%
^{\eta,\mathcal{G}}f\right\Vert _{L^{q}\left(  B\left(  0,2^{\frac{s}%
{1-\delta}}\right)  \right)  }\leq C_{\varepsilon^{\prime},\delta,\kappa
,p,q}2^{s\varepsilon^{\prime}}\left\Vert f\right\Vert _{L^{p}\left(  U\right)
},\ \ \ \ \ \text{for all }f\in L^{p}\left(  U\right)  ,\label{lin d'}%
\end{equation}
and so is equivalent to the Fourier extension conjecture.
\end{corollary}

\begin{proof}
Since
\begin{align*}
\widetilde{\mathsf{Q}}_{s}^{\eta,\mathcal{G}}\mathbf{1}_{U}f  & =\mathsf{Q}%
_{s}^{\eta,\mathcal{G}}T_{\mathcal{G}}^{\eta}\mathbf{1}_{U}f=\mathsf{Q}%
_{s}^{\eta,\mathcal{G}}\left(  \mathbf{1}_{U}+\sum_{n=1}^{\infty}%
\mathbf{1}_{2^{n}U\setminus2^{n-1}U}\right)  T_{\mathcal{G}}^{\eta}%
\mathbf{1}_{U}f=\sum_{n=0}^{\infty}\mathsf{Q}_{s}^{\eta,\mathcal{G}}f_{n}\\
\text{where }f_{0}  & \equiv\mathbf{1}_{U}T_{\mathcal{G}}^{\eta}\mathbf{1}%
_{U}f\text{ and }f_{n}\equiv\mathbf{1}_{2^{n}U\setminus2^{n-1}U}%
T_{\mathcal{G}}^{\eta}\mathbf{1}_{U}f\text{ for }n\geq1,
\end{align*}
we have by (\ref{lin d}) that%
\begin{align*}
\left\Vert \mathbb{E}_{\mathbb{G}}\mathcal{E}\widetilde{\mathsf{Q}}_{s}%
^{\eta,\mathcal{G}}f\right\Vert _{L^{q}\left(  B\left(  0,2^{\frac{s}%
{1-\delta}}\right)  \right)  }  & \leq\sum_{n=0}^{\infty}\left\Vert
\mathbb{E}_{\mathbb{G}}\mathcal{E}\mathsf{Q}_{s}^{\eta,\mathcal{G}}%
f_{n}\right\Vert _{L^{q}\left(  B\left(  0,2^{\frac{s}{1-\delta}}\right)
\right)  }\\
& \leq C_{\varepsilon^{\prime},\delta,\kappa,p,q}2^{s\varepsilon^{\prime}%
}\left\Vert f_{0}\right\Vert _{L^{p}\left(  U\right)  }+\sum_{n=0}^{\infty
}A^{n}\left\Vert f_{n}\right\Vert _{L^{p}\left(  2^{n}U\right)  }%
\end{align*}
where the constant $A^{n}$ arises from the rescaling of $U$ to $2^{n}U$. From
the invertibility of $T_{\mathcal{G}}^{\eta}$ on $L^{p}$, and the rapid decay
of $f_{n}=\mathbf{1}_{2^{n}U\setminus2^{n-1}U}T_{\mathcal{G}}^{\eta}%
\mathbf{1}_{U}f$ given in the well localized properties of $T_{\mathcal{G}%
}^{\eta}$ in Lemma \ref{lem special}, we then claim that%
\[
\left\Vert \mathbb{E}_{\mathbb{G}}\mathcal{E}\widetilde{\mathsf{Q}}_{s}%
^{\eta,\mathcal{G}}f\right\Vert _{L^{q}\left(  B\left(  0,2^{\frac{s}%
{1-\delta}}\right)  \right)  }\leq C_{\varepsilon^{\prime},\delta,\kappa
,p,q}2^{s\varepsilon^{\prime}}\left\Vert f\right\Vert _{L^{p}\left(  U\right)
}\ .
\]

Indeed, it suffices to prove that there is $\tau>A$ such that
\[
\left\Vert f_{n}\right\Vert _{L^{p}\left(  U\right)  }\leq C_{\tau}2^{-\tau
n}\left\Vert f\right\Vert _{L^{p}\left(  U\right)  },\ \ \ \ \ \text{for
}n=0,1,2,...
\]
which by duality is equivalent to
\[
\left\vert \left\langle T_{\mathcal{G}}^{\eta}\mathbf{1}_{U}f,\mathbf{1}%
_{2^{n}U\setminus2^{n-1}U}g\right\rangle \right\vert =\left\vert \left\langle
f_{n},g\right\rangle \right\vert \leq C_{\tau}2^{-\tau n}\left\Vert
f\right\Vert _{L^{p}\left(  U\right)  }\left\Vert g\right\Vert _{L^{p^{\prime
}}\left(  U\right)  }.
\]
Now if $f$ is supported in $U$ and has vanishing moments up to order $\kappa$,
then%
\[
f=\sum_{I\in\mathcal{G}\left[  U\right]  }\left\langle f,h_{I;\kappa
}\right\rangle h_{I;\kappa}\ \text{and }g=\sum_{K\in\mathcal{G}}\left\langle
g,h_{K;\kappa}\right\rangle h_{K;\kappa}\ ,
\]
and so%
\begin{align}
\left\langle T_{\mathcal{G}}^{\eta}\mathbf{1}_{U}f,\mathbf{1}_{2^{n}%
U\setminus2^{n-1}U}g\right\rangle  & =\left\langle T_{\mathcal{G}}^{\eta}%
\sum_{I\in\mathcal{G}\left[  U\right]  }\left\langle f,h_{I;\kappa
}\right\rangle h_{I;\kappa},\mathbf{1}_{2^{n}U\setminus2^{n-1}U}\sum
_{K\in\mathcal{G}}\left\langle g,h_{K;\kappa}\right\rangle h_{K;\kappa
}\right\rangle \label{<Tf,g>}\\
& =\left\langle T_{\mathcal{G}}^{\eta}\sum_{I\in\mathcal{G}\left[  U\right]
}\left\langle f,h_{I;\kappa}\right\rangle h_{I;\kappa},\sum_{K\in
\mathcal{G}\left[  2^{n}U\setminus2^{n-1}U\right]  }\left\langle
g,h_{K;\kappa}\right\rangle h_{K;\kappa}\right\rangle \nonumber\\
& +\left\langle T_{\mathcal{G}}^{\eta}\sum_{I\in\mathcal{G}\left[  U\right]
}\left\langle f,h_{I;\kappa}\right\rangle h_{I;\kappa},\sum_{K\in
\mathcal{G}\setminus\mathcal{G}\left[  2^{n}U\setminus2^{n-1}U\right]
}\left\langle g,h_{K;\kappa}\right\rangle h_{K;\kappa}\right\rangle .\nonumber
\end{align}

For the inner product in the second line we have%
\begin{equation}
\left\langle T_{\mathcal{G}}^{\eta}\sum_{I\in\mathcal{G}\left[  U\right]
}\left\langle f,h_{I;\kappa}\right\rangle h_{I;\kappa},\sum_{K\in
\mathcal{G}\left[  2^{n}U\setminus2^{n-1}U\right]  }\left\langle
g,h_{K;\kappa}\right\rangle h_{K;\kappa}\right\rangle =\sum_{I\in
\mathcal{G}\left[  U\right]  }\sum_{J\in\mathcal{G}\left[  2^{n}%
U\setminus2^{n-1}U\right]  }\left\langle f,h_{I;\kappa}\right\rangle
\left\langle g,h_{J;\kappa}\right\rangle \left\langle T_{\mathcal{G}}^{\eta
}h_{I;\kappa},h_{J;\kappa}\right\rangle ,\label{well loc}%
\end{equation}
for all $I,J\in\mathcal{G}$, where we have replaced the dummy variable $K$ by
$J$, so as to avoid confusion when applying Lemma \ref{lem special},%
\begin{align*}
& \left\vert \left\langle T_{\mathcal{G}}^{\eta}\sum_{I\in\mathcal{G}\left[
U\right]  }\left\langle f,h_{I;\kappa}\right\rangle h_{I;\kappa},\sum
_{K\in\mathcal{G}\left[  2^{n}U\setminus2^{n-1}U\right]  }\left\langle
g,h_{K;\kappa}\right\rangle h_{K;\kappa}\right\rangle \right\vert \\
& \leq\sum_{I\in\mathcal{G}\left[  U\right]  }\sum_{J\in\mathcal{G}\left[
2^{n}U\setminus2^{n-1}U\right]  }\left\vert \left\langle f,h_{I;\kappa
}\right\rangle \right\vert \left\vert \left\langle g,h_{J;\kappa}\right\rangle
\right\vert \left\vert \left\langle T_{\mathcal{G}}^{\eta}h_{I;\kappa
},h_{J;\kappa}\right\rangle \right\vert \\
& \leq\sum_{I\in\mathcal{G}\left[  U\right]  }\sum_{J\in\mathcal{G}\left[
2^{n}U\setminus2^{n-1}U\right]  }\left\vert \left\langle f,h_{I;\kappa
}\right\rangle \right\vert \left\vert \left\langle g,h_{J;\kappa}\right\rangle
\right\vert C_{d-1}\frac{1}{\eta^{\kappa}}\left(  \frac{\ell\left(  I\right)
}{2^{n}}\right)  ^{\kappa+\frac{d-1}{2}}\left(  \frac{\ell\left(  J\right)
}{2^{n}}\right)  ^{\kappa+\frac{d-1}{2}},
\end{align*}
where we have used that $T_{\mathcal{G}}^{\eta}$ is self-adjoint.

Now we see that the inner product in the middle line of (\ref{<Tf,g>}) is
controlled by%
\begin{align*}
& \left\vert \left\langle T_{\mathcal{G}}^{\eta}\sum_{I\in\mathcal{G}\left[
U\right]  }\left\langle f,h_{I;\kappa}\right\rangle h_{I;\kappa},\sum
_{J\in\mathcal{G}\left[  2^{n}U\setminus2^{n-1}U\right]  }\left\langle
g,h_{J;\kappa}\right\rangle h_{J;\kappa}\right\rangle \right\vert \\
& \leq\sum_{I\in\mathcal{G}\left[  U\right]  }\sum_{J\in\mathcal{G}\left[
2^{n}U\setminus2^{n-1}U\right]  }\left\vert \left\langle f,h_{I;\kappa
}\right\rangle \left\langle g,h_{J;\kappa}\right\rangle \right\vert
C_{d-1}\frac{1}{\eta^{\kappa}}\left(  \frac{\ell\left(  I\right)  }{2^{n}%
}\right)  ^{\kappa+\frac{d-1}{2}}\left(  \frac{\ell\left(  J\right)  }{2^{n}%
}\right)  ^{\kappa+\frac{d-1}{2}},
\end{align*}
which is equal to%
\begin{align*}
& \sum_{I\in\mathcal{G}\left[  U\right]  }\sum_{J\in\mathcal{G}\left[
2^{n}U\setminus2^{n-1}U\right]  }\frac{\left(  \int\left\vert \bigtriangleup
_{I;\kappa}f\right\vert \right)  \left(  \int\left\vert \bigtriangleup
_{J;\kappa}g\right\vert \right)  }{\ell\left(  I\right)  ^{-\frac{d-1}{2}}%
\ell\left(  J\right)  ^{-\frac{d-1}{2}}}C_{d-1}\frac{1}{\eta^{\kappa}}\left(
\frac{\ell\left(  I\right)  }{2^{n}}\right)  ^{\kappa+\frac{d-1}{2}}\left(
\frac{\ell\left(  J\right)  }{2^{n}}\right)  ^{\kappa+\frac{d-1}{2}}\\
& =\left(  \int\sum_{I\in\mathcal{G}\left[  U\right]  }\frac{\left\vert
\bigtriangleup_{I;\kappa}f\right\vert }{\ell\left(  I\right)  ^{-\frac{d-1}%
{2}}}\left(  \frac{\ell\left(  I\right)  }{2^{n}}\right)  ^{\kappa+\frac
{d-1}{2}}\right)  \left(  \int\sum_{J\in\mathcal{G}\left[  2^{n}%
U\setminus2^{n-1}U\right]  }\frac{\left\vert \bigtriangleup_{J;\kappa
}g\right\vert }{\ell\left(  J\right)  ^{-\frac{d-1}{2}}}C_{d-1}\frac{1}%
{\eta^{\kappa}}\left(  \frac{\ell\left(  J\right)  }{2^{n}}\right)
^{\kappa+\frac{d-1}{2}}\right)  \equiv AB.
\end{align*}

The first factor $A$ satisfies,%
\begin{align*}
A  & =\int\sum_{I\in\mathcal{G}\left[  U\right]  }\frac{\left\vert
\bigtriangleup_{I;\kappa}f\right\vert }{\ell\left(  I\right)  ^{-\frac{d-1}%
{2}}}\left(  \frac{\ell\left(  I\right)  }{2^{n}}\right)  ^{\kappa+\frac
{d-1}{2}}\\
& \leq\int_{\left(  1+\eta\right)  U}\left(  \sum_{I\in\mathcal{G}\left[
U\right]  }\left\vert \bigtriangleup_{I;\kappa}f\right\vert ^{2}\right)
^{\frac{1}{2}}\left(  \sum_{I\in\mathcal{G}\left[  U\right]  }\left(  \frac
{1}{\ell\left(  I\right)  ^{-\frac{d-1}{2}}}\left(  \frac{\ell\left(
I\right)  }{2^{n}}\right)  ^{\kappa+\frac{d-1}{2}}\right)  ^{2}\right)
^{\frac{1}{2}}\\
& \leq\left(  \int\left(  \sum_{I\in\mathcal{G}\left[  U\right]  }\left\vert
\bigtriangleup_{I;\kappa}f\right\vert ^{2}\right)  ^{\frac{p}{2}}\right)
^{\frac{1}{p}}\left(  \sum_{I\in\mathcal{G}\left[  U\right]  }\left(  \frac
{1}{\ell\left(  I\right)  ^{-\frac{d-1}{2}}}\left(  \frac{\ell\left(
I\right)  }{2^{n}}\right)  ^{\kappa+\frac{d-1}{2}}\right)  ^{2}\right)
^{\frac{1}{2}}\\
& \leq\left\Vert f\right\Vert _{L^{p}}\left(  \sum_{I\in\mathcal{G}\left[
U\right]  }\left(  \frac{1}{\ell\left(  I\right)  ^{-\frac{d-1}{2}}}\left(
\frac{\ell\left(  I\right)  }{2^{n}}\right)  ^{\kappa+\frac{d-1}{2}}\right)
^{2}\right)  ^{\frac{1}{2}}\lesssim2^{-n\left(  2\kappa+d-1\right)
}\left\Vert f\right\Vert _{L^{p}}%
\end{align*}
by the square function theorem for Alpert wavelets, and since%
\begin{align*}
& \sum_{I\in\mathcal{G}\left[  U\right]  }\left(  \frac{1}{\ell\left(
I\right)  ^{-\frac{d-1}{2}}}\left(  \frac{\ell\left(  I\right)  }{2^{n}%
}\right)  ^{\kappa+\frac{d-1}{2}}\right)  ^{2}=\sum_{I\in\mathcal{G}\left[
U\right]  }\frac{1}{\ell\left(  I\right)  ^{-\left(  d-1\right)  }}\left(
\frac{\ell\left(  I\right)  }{2^{n}}\right)  ^{2\kappa+d-1}\\
& =\sum_{m=0}^{\infty}\frac{2^{m\left(  d-1\right)  -\left(  2m\kappa
+2m\left(  d-1\right)  \right)  }}{2^{n\left(  2\kappa+d-1\right)  }}%
=\sum_{m=0}^{\infty}\frac{2^{-m\left(  2\kappa+\left(  d-1\right)  \right)  }%
}{2^{n\left(  2\kappa+d-1\right)  }}=C2^{-n\left(  2\kappa+d-1\right)  }.
\end{align*}
The second factor is handled similarly to obtain%
\[
\left\vert \int\sum_{J\in\mathcal{G}\left[  2^{n}U\setminus2^{n-1}U\right]
}\frac{\left\vert \bigtriangleup_{J;\kappa}g\right\vert }{\ell\left(
J\right)  ^{-\frac{d-1}{2}}}C_{d-1}\frac{1}{\eta^{\kappa}}\left(  \frac
{\ell\left(  J\right)  }{2^{n}}\right)  ^{\kappa+\frac{d-1}{2}}\right\vert
\lesssim2^{-n\left(  2\kappa+d-1\right)  }\left\Vert f\right\Vert _{L^{p}}.
\]

It remains to control the term in the third line of (\ref{<Tf,g>}), namely%
\begin{align*}
& \left\langle T_{\mathcal{G}}^{\eta}\sum_{I\in\mathcal{G}\left[  U\right]
}\left\langle f,h_{I;\kappa}\right\rangle h_{I;\kappa},\sum_{K\in
\mathcal{G}\setminus\mathcal{G}\left[  2^{n}U\setminus2^{n-1}U\right]
}\left\langle g,h_{K;\kappa}\right\rangle h_{K;\kappa}\right\rangle \\
& =\sum_{I\in\mathcal{G}\left[  U\right]  }\sum_{K\in\mathcal{G}%
\setminus\mathcal{G}\left[  2^{n}U\setminus2^{n-1}U\right]  }\left\langle
f,h_{I;\kappa}\right\rangle \left\langle g,h_{K;\kappa}\right\rangle
\left\langle T_{\mathcal{G}}^{\eta}h_{I;\kappa},h_{K;\kappa}\right\rangle \\
& =\sum_{I\in\mathcal{G}\left[  U\right]  }\sum_{J\in\mathcal{G}%
\setminus\mathcal{G}\left[  2^{n}U\setminus2^{n-1}U\right]  }\left\langle
f,h_{I;\kappa}\right\rangle \left\langle g,h_{J;\kappa}\right\rangle
\left\langle T_{\mathcal{G}}^{\eta}h_{I;\kappa},h_{J;\kappa}\right\rangle ,
\end{align*}
where $\operatorname*{Supp}g\subset2^{n}U\setminus2^{n-1}U$, and we gave again
replaced the dummy variable $K$ with $J$. Then $\left\langle g,h_{J;\kappa
}\right\rangle \neq0$ implies $J\cap2^{n}U\setminus2^{n-1}U\neq\emptyset$, and
we see that such $J\,$must be far from $I$. But when $J$ is far from $I$, this
sum can be handled by the same techniques we just used for the second line in
(\ref{<Tf,g>}). This completes the proof of (\ref{lin d'}) and Corollary
\ref{Cor LSST}.
\end{proof}

We now introduce an expectation over a tiling.

\begin{definition}
For $s\in\mathbb{N}$, we denote by $\mathbb{E}_{\mathbb{T}_{s}}$ the average
over the family $\mathbb{T}_{s}\equiv\left\{  \mathcal{D}_{s}+v\right\}
_{\left\vert v\right\vert \leq2^{-s}}$ of tilings $\mathcal{D}_{s}+v$ of
$\mathbb{R}^{d-1}$, endowed with the uniform probability measure in $v$.
\end{definition}

Now we note that the average of the function $\mathcal{E}\widetilde
{\mathsf{Q}}_{s}^{\eta,\mathcal{G}}f$ over all grids $\mathcal{G}$ coincides
with its average over the tilings $\left\{  \mathcal{D}_{s}+v\right\}
_{\left\vert v\right\vert \leq2^{-s}}$, i.e.
\begin{equation}
\mathbb{E}_{\mathbb{G}}\mathcal{E}\widetilde{\mathsf{Q}}_{s}^{\eta
,\mathcal{G}}f=\mathbb{E}_{\mathbb{T}_{s}}\mathcal{E}\widetilde{\mathsf{Q}%
}_{s}^{\eta,\mathcal{G}},\ \ \ \ \ \text{ for }s\in\mathbb{N}\text{.}%
\label{eq}%
\end{equation}

\begin{definition}
We use the notation $\operatorname*{LSST}^{\prime}\left(  d,\kappa
,\delta\right)  $ to refer to the condition in Definition \ref{LSST} taken
with the pseudoprojection $\widetilde{\mathsf{Q}}_{s}^{\eta,\mathcal{G}}$ in
place of $\mathsf{Q}_{s}^{\eta,\mathcal{G}}$, i.e. for every $\varepsilon
^{\prime}>0$, there is a positive constant $C_{d,\varepsilon^{\prime}%
,\delta,\kappa}$ such that (\ref{lin d'}) holds, namely%
\begin{equation}
\left\Vert \mathbb{E}_{\mathbb{T}_{s}}\mathcal{E}\widetilde{\mathsf{Q}}%
_{s}^{\eta,\mathcal{G}}f\right\Vert _{L^{q}\left(  B\left(  0,2^{\frac
{s}{1-\delta}}\right)  \right)  }\leq C_{\varepsilon^{\prime},\delta
,\kappa,p,q}2^{s\varepsilon^{\prime}}\left\Vert f\right\Vert _{L^{p}\left(
U\right)  },\ \ \ \ \ \text{for all }f\in L^{p}\left(  U\right)
,\label{lind''}%
\end{equation}
when taken over all $s\in\mathbb{N}$, all smooth perturbed Alpert projections
$\widetilde{\mathsf{Q}}_{s}^{\eta,\mathcal{G}}$ in the grid $\mathcal{G}$ with
moment vanishing parameter $\kappa$, and all $f\in L^{q}\left(  U\right)  $
with $q\geq\frac{2d}{d-1}$.
\end{definition}

As a consequence of (\ref{eq}) and Corollary \ref{Cor LSST}, we obtain the
following reduction of the Fourier extension conjecture.

\begin{theorem}
\label{Thm LSST'}The Fourier extension conjecture for a smooth compact piece
of the paraboloid $\mathbb{P}^{d-1}$ in $\mathbb{R}^{d}$ holds if and only if
for every $0<\delta\leq\frac{1}{2}$, there exists $\kappa\in\mathbb{N} $ such
that $\operatorname*{LSST}^{\prime}\left(  d,\kappa,\delta\right)  $ holds.
\end{theorem}

\subsubsection{A convolution reformulation}

We can now prove the main theorem stated in the introduction that reformulates
the averaged linear testing characterization in Theorem \ref{Thm LSST'} as a
convolution testing condition.

\begin{theorem}
\label{Thm LSST''}Let $h_{I_{s};\kappa}^{\eta}\left(  x\right)  =2^{\frac
{s}{2}}h_{\left[  0,1\right)  ^{d-1};\kappa}^{\eta}\left(  2^{s}x\right)  $ be
the $L^{2}$-normalized dilation of the smooth Alpert mother wavelet
$h_{\left[  0,1\right)  ^{d-1};\kappa}^{\eta}$ constructed in \cite{Saw7}, and
having vanishing moments up to order $\kappa$. Then the Fourier extension
conjecture holds for a smooth compact piece $S$ of the paraboloid
$\mathbb{P}^{d-1}$ in $\mathbb{R}^{d}$ \emph{if and only if} for every
$\varepsilon>0$ there is a positive constant $C_{\varepsilon}$ and a positive
integer $\kappa=\kappa_{\varepsilon}$ such that%
\begin{align*}
& \left\Vert \mathcal{E}_{S}\left(  h_{I_{s};\kappa}^{\eta}\ast f\right)
\right\Vert _{L^{\frac{2d}{d-1}}\left(  B\left(  0,2^{\left(  1+\varepsilon
\right)  s}\right)  \right)  }\leq C_{\varepsilon}2^{\varepsilon s}\left\Vert
f\right\Vert _{L^{\frac{2d}{d-1}}\left(  U\right)  },\\
& \ \ \ \ \ \ \ \ \ \ \text{for all }s\in\mathbb{N}\text{ and }f\in
L^{\frac{2d}{d-1}}\left(  U\right)  .
\end{align*}

\end{theorem}

\begin{proof}
In view of Theorem \ref{Thm LSST'}, it suffices to show that
\[
\mathbb{E}_{\mathbb{T}_{s}}\mathcal{E}\widetilde{\mathsf{Q}}_{s}%
^{\eta,\mathcal{G}}f\left(  \xi\right)  =\mathcal{E}\left(  k_{I_{s};\kappa
}^{\eta}\ast k_{I_{s};\kappa}^{\eta}\ast f\right)  \left(  \xi\right)  ,
\]
where $k_{I_{s};\kappa}^{\eta}\equiv2^{\frac{d-1}{2}s}h_{I_{s};\kappa}^{\eta}$
is the $L^{1}$-normalized smooth Alpert wavelet. We compute%
\begin{align*}
& \mathbb{E}_{\mathbb{T}_{s}}\mathcal{E}\widetilde{\mathsf{Q}}_{s}%
^{\eta,\mathcal{G}}f\left(  \xi\right)  =2^{\left(  d-1\right)  s}%
\int_{\left[  0,2^{-s}\right)  ^{d-1}}\int e^{-i\xi\cdot\Phi\left(  x\right)
}\sum_{I\in\mathcal{D}_{s}+v}\left\langle f,h_{I;\kappa}^{\eta}\right\rangle
h_{I;\kappa}^{\eta}\left(  x\right)  dxdv\\
& =2^{\left(  d-1\right)  s}\int_{\left[  0,2^{-s}\right)  ^{d-1}}\int
e^{-i\xi\cdot\Phi\left(  x\right)  }\sum_{I\in\mathcal{D}_{s}}\left\langle
f,h_{I+v;\kappa}^{\eta}\right\rangle h_{I_{s};\kappa}^{\eta}\left(
x+v+\lambda_{I}\right)  dxdv\\
& =\int e^{-i\xi\cdot\Phi\left(  x\right)  }\sum_{I\in\mathcal{D}_{s}%
}2^{\left(  d-1\right)  s}\int_{\left[  0,2^{-s}\right)  ^{d-1}}\int f\left(
y\right)  h_{I+v;\kappa}^{\eta}\left(  y\right)  dyh_{I_{s};\kappa}^{\eta
}\left(  x+v+\lambda_{I}\right)  dvdx\\
& =\int e^{-i\xi\cdot\Phi\left(  x\right)  }\sum_{I\in\mathcal{D}_{s}%
}2^{\left(  d-1\right)  s}\int_{\left[  0,2^{-s}\right)  ^{d-1}}\int f\left(
y\right)  h_{I_{s};\kappa}^{\eta}\left(  y+v+\lambda_{I}\right)
dyh_{I_{s};\kappa}^{\eta}\left(  x+v+\lambda_{I}\right)  dvdx\\
& =\int\int\left\{  \sum_{I\in\mathcal{D}_{s}}2^{\left(  d-1\right)  s}%
\int_{\left[  0,2^{-s}\right)  ^{d-1}}e^{-i\xi\cdot\Phi\left(  x-v-\lambda
_{I}\right)  }f\left(  y-v-\lambda_{I}\right)  dv\right\}  h_{I_{s};\kappa
}^{\eta}\left(  y\right)  h_{I_{s};\kappa}^{\eta}\left(  x\right)  dydx
\end{align*}
where%
\begin{align*}
& \sum_{I\in\mathcal{D}_{s}}2^{\left(  d-1\right)  s}\int_{\left[
0,2^{-s}\right)  ^{d-1}}e^{-i\xi\cdot\Phi\left(  x-v\right)  }f\left(
y-v\right)  dv\\
& =\sum_{I\in\mathcal{D}_{s}}2^{\left(  d-1\right)  s}\int_{\left[
0,2^{-s}\right)  ^{d-1}+\lambda_{I}}e^{-i\xi\cdot\Phi\left(  x-v\right)
}f\left(  y-v\right)  dv\\
& =2^{\left(  d-1\right)  s}\int e^{-i\xi\cdot\Phi\left(  x-w\right)
}f\left(  y-w\right)  dw=2^{\left(  d-1\right)  s}\int e^{-i\xi\cdot
\Phi\left(  x-y+w\right)  }f\left(  w\right)  dw\\
& =2^{\left(  d-1\right)  s}\int e^{-i\xi\cdot\Phi\left(  w\right)  }f\left(
w+y-x\right)  dw\\
& =2^{\left(  d-1\right)  s}\int e^{-i\xi\cdot\Phi\left(  w\right)  }%
\tau_{y-x}f\left(  w\right)  dw=2^{\left(  d-1\right)  s}\mathcal{E}\left(
\tau_{y-x}f\right)  \left(  \xi\right)  .
\end{align*}
So altogether we have
\begin{align*}
\mathbb{E}_{\mathbb{T}_{s}}\mathcal{E}\widetilde{\mathsf{Q}}_{s}%
^{\eta,\mathcal{G}}f\left(  \xi\right)   & =\int\int2^{\left(  d-1\right)
s}\mathcal{E}\left(  \tau_{y-x}f\right)  \left(  \xi\right)  h_{I_{s};\kappa
}^{\eta}\left(  y\right)  h_{I_{s};\kappa}^{\eta}\left(  x\right)  dydx\\
& =\int\int\mathcal{E}\left(  \tau_{y-x}f\right)  \left(  \xi\right)
k_{I_{s};\kappa}^{\eta}\left(  y\right)  k_{I_{s};\kappa}^{\eta}\left(
x\right)  dydx\\
& =\mathcal{E}\left(  \int\int k_{I_{s};\kappa}^{\eta}\left(  y\right)
k_{I_{s};\kappa}^{\eta}\left(  x\right)  \tau_{y-x}f\left(  \cdot\right)
dydx\right)  \left(  \xi\right)  =\mathcal{E}\widetilde{f}\left(  \xi\right)
,
\end{align*}
where%
\[
\widetilde{f}\left(  z\right)  =\int\int k_{I_{s};\kappa}^{\eta}\left(
y\right)  k_{I_{s};\kappa}^{\eta}\left(  x\right)  f\left(  z+y-x\right)
dydx=k_{I_{s};\kappa}^{\eta}\ast k_{I_{s};\kappa}^{\eta}\ast f\left(
z\right)  .
\]

Finally, we note that for $\eta>0$ sufficiently small, $k_{I_{s};\kappa}%
^{\eta}\ast k_{I_{s};\kappa}^{\eta}$ is itself an $L^{1}$-normalized smooth
Alpert wavelet, which permits us to write $k_{I_{s};\kappa}^{\eta}$ in place
of $k_{I_{s};\kappa}^{\eta}\ast k_{I_{s};\kappa}^{\eta}$.
\end{proof}

\section{Appendix A: Bilinear norm ratio invariance}

If translation is defined by $\tau_{c}g\left(  x\right)  =g\left(  x+c\right)
$, then%
\begin{align*}
&  \mathcal{E}\tau_{c}f\left(  \xi\right)  =\int e^{-i\Phi\left(  x\right)
\cdot\xi}\tau_{c}f\left(  x\right)  dx=\int e^{-i\Phi\left(  x\right)
\cdot\xi}f\left(  x+c\right)  dx\\
&  =\int e^{-i\Phi\left(  x-c\right)  \cdot\xi}f\left(  x\right)  dx=\int
e^{-i\left\{  \xi^{\prime}\cdot\left(  x-c\right)  +\xi_{d}\left\vert
x-c\right\vert ^{2}\right\}  }f\left(  x\right)  dx=e^{-i\xi^{\prime}\cdot
c}e^{-i\xi_{d}\left\vert c\right\vert ^{2}}\mathcal{E}f\left(  \xi^{\prime
}-2\xi_{d}c,\xi_{d}\right)  ,
\end{align*}
which gives translation invariance in norm,%
\begin{align}
&  \left\Vert \left(  \mathcal{E}\tau_{c}f_{1}\right)  \left(  \mathcal{E}%
\tau_{c}f_{2}\right)  \right\Vert _{L^{\frac{q}{2}}\left(  \mathbb{R}%
^{d}\right)  }^{\frac{q}{2}}=\int_{\mathbb{R}}\left\{  \int_{\mathbb{R}^{d-1}%
}\left\vert \mathcal{E}f_{1}\left(  \xi^{\prime}-2\xi_{d}c,\xi_{d}\right)
\ \mathcal{E}f_{2}\left(  \xi^{\prime}-2\xi_{d}c,\xi_{d}\right)  \right\vert
^{\frac{q}{2}}d\xi^{\prime}\right\}  d\xi_{d}\label{trans}\\
&  =\int_{\mathbb{R}}\left\{  \int_{\mathbb{R}^{d-1}}\left\vert \mathcal{E}%
f_{1}\left(  \xi^{\prime},\xi_{d}\right)  \ \mathcal{E}f_{2}\left(
\xi^{\prime},\xi_{d}\right)  \right\vert ^{\frac{q}{2}}d\xi^{\prime}\right\}
d\xi_{d}=\left\Vert \left(  \mathcal{E}f_{1}\right)  \left(  \mathcal{E}%
f_{2}\right)  \right\Vert _{L^{\frac{q}{2}}\left(  \mathbb{R}^{d}\right)
}^{\frac{q}{2}}.\nonumber
\end{align}

If parabolic rescaling $\delta_{t}$ is defined by
\[
\delta_{t}F\left(  \xi^{\prime},\xi_{d}\right)  =F\left(  t\xi^{\prime}%
,t^{2}\xi_{d}\right)  ,
\]
then $\delta_{t}$ preserves functions defined on the paraboloid $\mathbb{P}%
^{d-1}=\left\{  \xi_{d}=\left\vert \xi^{\prime}\right\vert ^{2}\right\}  $,
and we have%
\begin{align*}
&  \left\Vert \left(  \delta_{t}\mathcal{E}f_{1}\right)  \left(  \delta
_{t}\mathcal{E}f_{2}\right)  \right\Vert _{L^{\frac{q}{2}}\left(
\mathbb{R}^{d}\right)  }^{\frac{q}{2}}=\int\left\vert \left(  \delta
_{t}\mathcal{E}f_{1}\right)  \left(  \delta_{t}\mathcal{E}f_{2}\right)
\right\vert ^{\frac{q}{2}}d\xi=\int_{\mathbb{R}}\int_{\mathbb{R}^{d-1}%
}\left\vert \mathcal{E}f_{1}\left(  t\xi^{\prime},t^{2}\xi_{d}\right)
\mathcal{E}f_{2}\left(  t\xi^{\prime},t^{2}\xi_{d}\right)  \right\vert
^{\frac{q}{2}}d\xi^{\prime}d\xi_{d}\\
&  =\int_{\mathbb{R}}\int_{\mathbb{R}^{d-1}}\left\vert \int e^{-i\left(
x,\left\vert x\right\vert ^{2}\right)  \cdot\left(  t\xi^{\prime},t^{2}\xi
_{d}\right)  }f_{1}\left(  x\right)  dx\int e^{-i\left(  y,\left\vert
y\right\vert ^{2}\right)  \cdot\left(  t\xi^{\prime},t^{2}\xi_{d}\right)
}f_{2}\left(  y\right)  dy\right\vert ^{\frac{q}{2}}d\xi^{\prime}d\xi_{d}\\
&  =\int_{\mathbb{R}}\int_{\mathbb{R}^{d-1}}\left\vert \int e^{-i\left(
tx,\left\vert tx\right\vert ^{2}\right)  \cdot\left(  \xi^{\prime},\xi
_{d}\right)  }f_{1}\left(  x\right)  dx\int e^{-i\left(  ty,\left\vert
ty\right\vert ^{2}\right)  \cdot\left(  \xi^{\prime},\xi_{d}\right)  }%
f_{2}\left(  y\right)  dy\right\vert ^{\frac{q}{2}}d\xi^{\prime}d\xi_{d}\\
&  =\int_{\mathbb{R}}\int_{\mathbb{R}^{d-1}}\left\vert \int e^{-i\left(
u,\left\vert u\right\vert ^{2}\right)  \cdot\left(  \xi^{\prime},\xi
_{d}\right)  }f_{1}\left(  \frac{u}{t}\right)  \frac{du}{t^{d-1}}\int
e^{-i\left(  v,\left\vert v\right\vert ^{2}\right)  \cdot\left(  \xi^{\prime
},\xi_{d}\right)  }f_{2}\left(  \frac{v}{t}\right)  \frac{dv}{t^{d-1}%
}\right\vert ^{\frac{q}{2}}d\xi^{\prime}d\xi_{d}=t^{-\left(  d-1\right)
q}\left\Vert \left(  \mathcal{E}\delta_{\frac{1}{t}}f\right)  \left(
\mathcal{E}\delta_{\frac{1}{t}}f\right)  \right\Vert _{L^{q}\left(
\mathbb{R}^{d}\right)  }^{\frac{q}{2}}.
\end{align*}

We also have%
\begin{align*}
\left\Vert \left(  \delta_{t}F_{1}\right)  \left(  \delta_{t}F_{2}\right)
\right\Vert _{L^{\frac{q}{2}}\left(  \mathbb{R}^{d}\right)  }^{\frac{q}{2}} &
=\int_{\mathbb{R}^{d}}\left\vert \delta_{t}F_{1}\left(  \xi\right)  \delta
_{t}F_{2}\left(  \xi\right)  \right\vert ^{\frac{q}{2}}d\xi=\int_{\mathbb{R}%
}\int_{\mathbb{R}^{d-1}}\left\vert F_{1}\left(  t\xi^{\prime},t^{2}\xi
_{d}\right)  F_{2}\left(  t\xi^{\prime},t^{2}\xi_{d}\right)  \right\vert
^{\frac{q}{2}}d\xi^{\prime}d\xi_{d}\\
&  =\int_{\mathbb{R}}\int_{\mathbb{R}^{d-1}}\left\vert F_{1}\left(
\eta^{\prime},\eta_{d}\right)  F_{2}\left(  \eta^{\prime},\eta_{d}\right)
\right\vert ^{\frac{q}{2}}\frac{d\eta^{\prime}}{t^{d-1}}\frac{d\eta_{d}}%
{t^{2}}=\frac{1}{t^{d+1}}\left\Vert F\right\Vert _{L^{\frac{q}{2}}\left(
\mathbb{R}^{d}\right)  }^{\frac{q}{2}}%
\end{align*}
and%
\begin{align*}
\left\Vert \delta_{\frac{1}{t}}H\right\Vert _{L^{p}\left(  \mathbb{R}%
^{d-1}\right)  }^{p} &  =\int_{\mathbb{R}^{d-1}}\left\vert \delta_{\frac{1}%
{t}}H\left(  x\right)  \right\vert ^{p}dx=\int_{\mathbb{R}^{d-1}}\left\vert
H\left(  \frac{x}{t}\right)  \right\vert ^{p}dx\\
&  =\int_{\mathbb{R}^{d-1}}\left\vert H\left(  y\right)  \right\vert
^{p}t^{d-1}dy=t^{d-1}\left\Vert H\right\Vert _{L^{p}\left(  \mathbb{R}%
^{d-1}\right)  }^{p},
\end{align*}
and so\ altogether,%
\begin{align*}
& \frac{\left\Vert \left(  \mathcal{E}f_{1}\right)  \left(  \mathcal{E}%
f_{2}\right)  \right\Vert _{L^{\frac{q}{2}}\left(  \mathbb{R}^{d}\right)
}^{\frac{1}{2}}}{\left\Vert f_{1}\right\Vert _{L^{p}\left(  \mathbb{R}%
^{d-1}\right)  }^{\frac{1}{2}}\left\Vert f_{2}\right\Vert _{L^{p}\left(
\mathbb{R}^{d-1}\right)  }^{\frac{1}{2}}}=\frac{t^{\frac{d+1}{q}}\left\Vert
\left(  \delta_{t}\mathcal{E}f_{1}\right)  \left(  \delta_{t}\mathcal{E}%
f_{2}\right)  \right\Vert _{L^{\frac{q}{2}}\left(  \mathbb{R}^{d}\right)
}^{\frac{1}{2}}}{\left\Vert f_{1}\right\Vert _{L^{p}\left(  \mathbb{R}%
^{d-1}\right)  }^{\frac{1}{2}}\left\Vert f_{2}\right\Vert _{L^{p}\left(
\mathbb{R}^{d-1}\right)  }^{\frac{1}{2}}}\\
& =t^{\frac{d+1}{q}}\frac{t^{-\left(  d-1\right)  }\left\Vert \left(
\mathcal{E}\delta_{\frac{1}{t}}f_{1}\right)  \left(  \mathcal{E}\delta
_{\frac{1}{t}}f_{2}\right)  \right\Vert _{L^{\frac{q}{2}}\left(
\mathbb{R}^{d}\right)  }^{\frac{1}{2}}}{t^{-\frac{d-1}{p}}\left\Vert
\delta_{\frac{1}{t}}f_{1}\right\Vert _{L^{p}\left(  \mathbb{R}^{d-1}\right)
}^{\frac{1}{2}}\left\Vert \delta_{\frac{1}{t}}f_{2}\right\Vert _{L^{p}\left(
\mathbb{R}^{d-1}\right)  }^{\frac{1}{2}}}=t^{\frac{d+1}{q}-\frac
{d-1}{p^{\prime}}}\frac{\left\Vert \left(  \mathcal{E}\delta_{\frac{1}{t}%
}f_{1}\right)  \left(  \mathcal{E}\delta_{\frac{1}{t}}f_{2}\right)
\right\Vert _{L^{\frac{q}{2}}\left(  \mathbb{R}^{d}\right)  }^{\frac{1}{2}}%
}{\left\Vert \delta_{\frac{1}{t}}f_{1}\right\Vert _{L^{p}\left(
\mathbb{R}^{d-1}\right)  }^{\frac{1}{2}}\left\Vert \delta_{\frac{1}{t}}%
f_{2}\right\Vert _{L^{p}\left(  \mathbb{R}^{d-1}\right)  }^{\frac{1}{2}}}.
\end{align*}

Thus the bilinear norm ratio%
\begin{equation}
\left\Vert \mathcal{E}\left(  f_{1},f_{2}\right)  \right\Vert
_{\operatorname*{bilinear}}^{\left(  p,q\right)  }\equiv\frac{\left\Vert
\left(  \mathcal{E}f_{1}\right)  \left(  \mathcal{E}f_{2}\right)  \right\Vert
_{L^{\frac{q}{2}}\left(  \mathbb{R}^{d}\right)  }^{\frac{1}{2}}}{\left\Vert
f_{1}\right\Vert _{L^{p}\left(  \mathbb{R}^{d-1}\right)  }^{\frac{1}{2}%
}\left\Vert f_{2}\right\Vert _{L^{p}\left(  \mathbb{R}^{d-1}\right)  }%
^{\frac{1}{2}}}\label{def bil norm rat}%
\end{equation}
satisfies%
\begin{equation}
\left\Vert \mathcal{E}\left(  f_{1},f_{2}\right)  \right\Vert
_{\operatorname*{bilinear}}^{\left(  p,q\right)  }=t^{\frac{d+1}{q}-\frac
{d-1}{p^{\prime}}}\left\Vert \mathcal{E}\left(  \tau_{c}\delta_{\frac{1}{t}%
}f_{1},\tau_{c}\delta_{\frac{1}{t}}f_{2}\right)  \right\Vert
_{\operatorname*{bilinear}}^{\left(  p,q\right)  }.\label{dil'}%
\end{equation}
In particular we have $\frac{d+1}{q}=\frac{d-1}{p^{\prime}}$ when
$p=q=\frac{2d}{d-1}$, and also%
\begin{equation}
\left\Vert \mathcal{E}\left(  f_{1},f_{2}\right)  \right\Vert
_{\operatorname*{bilinear}}^{\left(  \infty,q\right)  }=t^{-\frac{d-1}%
{q}\left[  q-\frac{d+1}{d-1}\right]  }\left\Vert \mathcal{E}\left(  \tau
_{c}\delta_{\frac{1}{t}}f_{1},\tau_{c}\delta_{\frac{1}{t}}f_{2}\right)
\right\Vert _{\operatorname*{bilinear}}^{\left(  \infty,q\right)
}.\label{dil''}%
\end{equation}

\end{document}